\documentclass[11pt,a4paper,reqno]{amsart}
\usepackage[english]{babel}
\usepackage[utf8]{inputenc} 
\usepackage{fancyhdr}
\usepackage{indentfirst}
\usepackage{color}     
\usepackage{graphicx}   
\usepackage{newlfont}
\usepackage{upgreek}
\usepackage{commath}
\usepackage{amsfonts}
\usepackage{amsmath}
\usepackage{amssymb}
\usepackage{amsthm}
\usepackage{latexsym}
\usepackage{mathrsfs}
\usepackage{mathtools}
\mathtoolsset{showonlyrefs=true}
\usepackage{verbatim}
\usepackage[all]{xy}
\usepackage{lineno}
\usepackage{leftidx}
\usepackage[hidelinks]{hyperref}
\usepackage{units}
\usepackage{xfrac}
\usepackage{faktor}
\usepackage{float}
\usepackage{tikz}
\usetikzlibrary{arrows, positioning, automata}

\usepackage{chngcntr}
\usepackage{apptools}

\usepackage[flushleft]{threeparttable}

\DeclareMathAlphabet{\dutchcal}{U}{dutchcal}{m}{n}
\SetMathAlphabet{\dutchcal}{bold}{U}{dutchcal}{b}{n}
\DeclareMathAlphabet{\dutchbcal} {U}{dutchcal}{b}{n}

\theoremstyle{plain} 
\newtheorem*{theo*}{Theorem}
\newtheorem*{cor*}{Corollary}
\newtheorem*{con*}{Conjecture}
\newtheorem{theo}{Theorem}[section] 
\newtheorem{prop}[theo]{Proposition}
\newtheorem{cor}[theo]{Corollary}
\newtheorem{lem}[theo]{Lemma}
\newtheorem{con}[theo]{Conjecture}

\theoremstyle{definition}

\theoremstyle{definition}
\newtheorem{defin}[theo]{Definition}
\newtheorem{ex}[theo]{Example}

\theoremstyle{remark}
\newtheorem{rem}[theo]{Remark}
\newtheorem*{rem*}{Remark}
\newtheorem*{cap*}{Caption}

\AtAppendix{\counterwithin{theo}{section}}

\newcommand{\Z}{\mathbb{Z}}
\newcommand{\Zplus}{\Z_{>0}}
\newcommand{\Zpluseq}{\Z_{\geq0}}
\newcommand{\R}{\mathbb{R}}
\newcommand{\C}{\mathbb{C}}
\newcommand{\T}{\mathbb{T}}

\newcommand{\Vir}{\mathfrak{Vir}}
\newcommand{\SL}{\operatorname{SL}}
\newcommand{\SO}{\operatorname{SO}}
\newcommand{\GL}{\operatorname{GL}}

\newcommand{\half}{\frac{1}{2}}
\newcommand{\Y}{\mathcal{Y}} 
\newcommand{\cur}{\dutchcal{j}}
\newcommand{\latt}{\dutchcal{L}}

\newcommand{\scalar}{(\cdot|\cdot)}

\newcommand{\scalarang}{\langle \cdot|\cdot \rangle}
\newcommand{\bilinear}{(\cdot\,,\cdot)}

\newcommand{\pairing}{\langle\cdot,\cdot\rangle}

\newcommand{\Aut}{\operatorname{Aut}}
\newcommand{\End}{\operatorname{End}}
\newcommand{\Hom}{\operatorname{Hom}}

\newcommand{\Rep}{\operatorname{Rep}}

\newcommand{\catc}{\mathcal{C}}
\newcommand{\catd}{\mathcal{D}}
\newcommand{\catw}{\mathcal{W}}
\newcommand{\one}{\mathbf{1}}
\newcommand{\id}{\operatorname{id}}
\newcommand{\Ind}{\operatorname{Ind}}
\newcommand{\F}{\mathcal{F}}

\newcommand{\spm}{\operatorname{Spm}}
\newcommand{\rad}{\operatorname{rad}}
\newcommand{\nil}{\operatorname{Nil}}

\newcommand{\spec}{\operatorname{sp}}

\newcommand{\rcoset}[2]{{{#2}\backslash{#1}}}
\newcommand{\lcoset}[2]{{{#1}\slash{#2}}}

\begin{document}
	
\author{Sebastiano Carpi}
	\address{Dipartimento di Matematica, Universit\`a di Roma ``Tor Vergata'', Via della Ricerca Scientifica, 1, 00133 Roma, Italy\\
		E-mail: {\tt carpi@mat.uniroma2.it}
	}
\author{Tiziano Gaudio}
	\address{Dipartimento di Matematica, Universit\`a di Roma ``Tor Vergata'', Via della Ricerca Scientifica, 1, 00133 Roma, Italy\\
		E-mail: {\tt gaudio@mat.uniroma2.it}
	}
	
\author{Luca Giorgetti}
\address{Dipartimento di Matematica, Universit\`a di Roma ``Tor Vergata'', Via della Ricerca Scientifica, 1, 00133 Roma, Italy\\
	E-mail: {\tt giorgett@mat.uniroma2.it}
}

\title[Algebra objects in direct limit completions of compact Lie group duals]{Algebra objects in direct limit completions of compact Lie group duals and the classification of $c=1$ vertex operator algebras}

\date{}


\maketitle

\begin{center} {\small {\it Dedicated to the memory of Sergio Doplicher}}
\end{center}
	
\begin{abstract} 
	Let $G$ be a complex reductive affine algebraic group with its symmetric tensor category $\catc_G$ of finite-dimensional rational representations. We prove that every simple commutative algebra object with at most countable dimension in the  direct limit completion $\Ind(\catc_G)$ of $\catc_G$ is isomorphic to the algebra $\mathcal{O}(G/H)$ of regular functions on the homogeneous space 
	$G/H$, for some reductive algebraic subgroup $H$ of $G$. Then we apply this result to the theory of vertex operator algebra extensions. In particular, assuming the strong rationality of the $\operatorname{A}_5$-orbifold of the vertex operator algebra $V_{\latt_2}$ associated with the rank-one root lattice $\latt_2:= \sqrt{2}\mathbb{Z}$, we classify all the not necessarily rational simple CFT type preunitary vertex operator algebra extensions of the simple unitary Virasoro vertex operator algebra $L(1,0)$ with central charge $c=1$ satisfying a certain spectrum condition. This result is the vertex operator algebra analogue of a conformal net result by Feng Xu. Every strongly rational preunitary vertex operator algebra extension of $L(1,0)$ satisfies the above spectrum condition because of the congruence subgroup modularity property of its characters. As a consequence, we get a complete classification result for strongly rational $c=1$ vertex operator algebras, up to the  strong rationality of $V_{\latt_2}^{\operatorname{A}_5}$. 
\end{abstract}

\tableofcontents

\section{Introduction}
The classification program for rational two-dimensional conformal field theories (CFTs) arises naturally in theoretical physics, from string theory to critical phenomena, and it has been an impressive source of many remarkable mathematical problems and results over the last fifty years. 
A general solution of this problem corresponds to classifying all rational two-dimensional CFTs
with central charge $c$ for all allowed values of $c$. These values are typically rational numbers. Further constraints arise if one restricts to some special families of theories, such as unitary CFTs, holomorphic CFTs, superconformal CFTs and so on. For a given central charge $c$ a complete classification will typically require an intermediate step, that is the classification of rational chiral CFTs, namely two-dimensional CFTs generated by holomorphic or anti-holomorphic (in other words right-moving or left-moving) fields, having central charge $c$. These are often considered as the building blocks of the full two-dimensional theories. 

In order to place any classification result on a firm mathematical basis, one should  first fix an axiomatic framework for chiral CFT. In this paper, we are mainly interested in the well-established vertex operator algebra (VOA) approach  \cite{FLM88,FHL93,Kac01,LL04}, which is essentially equivalent to a suitable generalization of the Wightman setting \cite{SW64} allowing possibly non-unitary theories, see \cite{CRTT25}.  Another important framework for chiral CFTs is given by conformal nets on $S^1$, see e.g.\ \cite{CKLW18,KL04}. The latter is the chiral CFT version of algebraic quantum field theory \cite{Haag96} and it is suitable for unitary theories only. As we are going to explain, although the main results in this article concern VOAs, they are strongly inspired and motivated by the conformal net setting. 

Conformal nets are  deeply related to unitary VOAs. In \cite{CKLW18} it has been shown that for every VOA $V$ in a class of suitably nice simple unitary VOAs called strongly  local, one can define a conformal net $\mathcal{A}_V$. Moreover, it has been conjectured  that every simple unitary VOA $V$ is strongly local and that the map $V \mapsto \mathcal{A}_V$ is surjective.  It is also conjectured that a unitary VOA $V$ is strongly rational if and only if the corresponding conformal net is completely rational and that the corresponding representation categories are braided tensor equivalent, see \cite{Kaw19}. The correspondence in \cite{CKLW18} has been recently generalized to the super case in \cite{CGH} and similar conjectures have also been formulated for these super cases. Despite various relevant recent progress in these directions, see e.g.\ \cite{CTW22,CTW23,CT23,CGGH23,Ten24,CG25,Gui26,Gui,HT}, these conjectures are presently open.

For generic values of the central charges, without further assumptions, even the classification of chiral CFTs appears to be intractable and the research focused on specific directions. Two of these directions have attracted considerable attention. The first is the classification of holomorphic chiral CFTs with small central charge, here small means $c\leq 24$. This program was started by Schellekens in \cite{Sch93} and it is now essentially complete in the VOA setting up to the conjectured uniqueness of the moonshine VOA, see \cite{DM04b,LS19,EMS20,ELMS21,MS23}, see also \cite{MS} for the vertex operator superalgebra case and \cite{CGGH23,Gau25} for the corresponding conformal nets. The second is the classification of rational chiral CFTs with small central charge, here small means $c\leq 1$. 
For $c<1$ a complete classification of conformal nets has been found in  \cite{KL04}, see also \cite{KL04b} for the non-chiral two-dimensional case. The corresponding result for VOAs was obtained later in \cite{DL15} and it was shown to follow directly from the classification of conformal nets in \cite{CGGH23}. These remarkable results heavily rely on the rationality of the corresponding Virasoro subtheory associated to the conformal symmetry and cannot be adapted to values of $c\geq 1$ unless one assumes some superconformal symmetry, see \cite{CKL08,CHKLX15,CGGH23}.

Important results in the direction of the classification of $c=1$ two-dimensional conformal field theories have been first obtained at the level of characters/partition functions, see \cite{Gin88,DVV88,DVVV89,Kir89}. In particular, the main result of Kiritsis can be adapted to the VOA framework, see \cite[Theorem 2.9]{DJ14}, to prove that the vacuum character of a strongly rational VOA with $c=1$ and effective central charge $\tilde{c}=1$ (see Definition \ref{def:effective}) must coincide with the vacuum character of one of the following VOAs

\begin{equation}\label{eq:rationalc=1}
	V_{\latt_2}^{G} \,, \,\,\, 
	V_{\latt_{2n}} \,, \,\,\, 
	V_{\latt_{2n}}^+ \,, \,\,\, 
	n \in  \Zplus \, .
\end{equation}
Here, $V_{\latt_{2n}}$ is the lattice VOA associated to the rank-one lattice $\latt_{2n}:= \sqrt{2n}\mathbb{Z}$, $V_{\latt_{2n}}^{+}$ is the corresponding orbifold subalgebra with respect to the standard order two automorphism of   $V_{\latt_{2n}}$, and $G$ is a finite subgroup of 
$\SO(3):=\SO(3,\R)$ isomorphic to $\operatorname{S}_4$, $\operatorname{A}_4$ or $\operatorname{A}_5$. 

Without the assumption on the effective central charge, one can find counterexamples to the above VOA version of Kiritsis' theorem by using Example (f) in \cite[Section 4]{DM04}. Actually, one can rule out these counterexamples also by assuming that the strongly rational VOAs are preunitary VOA extensions of the $c=1$ Virasoro VOA $L(1,0)$. Preunitarity turns out to be equivalent to the requirement that these VOAs are direct sums of $L(1,0)$ highest weight modules, see Section \ref{subsec:discrete_VOA_extensions}. On the other hand, to the best of our knowledge, 
it is not known whether it is possible to replace the assumption $\tilde{c}=1$ with preunitarity in order to prove the VOA Kiritsis' theorem.

This VOA version of Kiritsis' theorem suggests the following conjecture:

\begin{con}[cf.\ Conjecture \ref{con:strongly_rational_VOAs_cc1}]
	\label{con:c=1rational_aIntro} 
	The isomorphism classes of strongly rational preunitary VOA extensions of $L(1,0)$ with $\tilde{c}=1$ coincide with the isomorphism classes of 
	\begin{equation}
		V_{\latt_2}^{\operatorname{S}_4} \,,\,\,\, 
		V_{\latt_2}^{\operatorname{A}_4}\,,\,\,\, 
		V_{\latt_2}^{\operatorname{A}_5}\,,\,\,\,
		V_{\latt_{2n}}\,,\,\,\, 
		V_{\latt_{2n}}^+ 
		\,,\,\,\, n \in  \Zplus \, .
	\end{equation}
\end{con}

Thanks to the VOA version of Kiritsis' theorem and the fact that all the VOAs in \eqref{eq:rationalc=1} are known to be preunitary VOA extensions of $L(1,0)$,  the proof of Conjecture \ref{con:c=1rational_aIntro} could be reduced to the following two steps: (A) prove that every VOA in \eqref{eq:rationalc=1} is strongly rational with $\tilde{c} =1$; (B) prove that the isomorphism classes of strongly rational preunitary VOA extensions of $L(1,0)$ with $c=\tilde{c} =1$ are completely determined by the corresponding vacuum characters.  

Concerning step (A), this has been done for all cases with the exception of the $\operatorname{A}_5$-orbifold of $V_{\latt_2}$ whose strong rationality has remained an important open problem in VOA theory for a long time, see \cite{DLM97,DN99b,Abe05,DJ13b,Lin15,CM}. This can be summarized by the fact that the 
lattice VOAs are strongly rational by \cite{DLM97} and all the VOAs in question are orbifold vertex subalgebras of rank-one lattice VOAs by solvable finite groups, which are strongly rational by \cite{CM} with the exception of the $\operatorname{A}_5$-orbifold of $V_{L_2}$ because $\operatorname{A}_5$ is simple and non-abelian. 
Note that in the recent preprint \cite{Xu}, a proof of the strong rationality of $V_{\latt_2}^{\operatorname{A}_5}$ is given.
Concerning step (B), there has been a lot of work on the characterization of strongly rational preunitary VOA extensions of $L(1,0)$ with $c=\tilde{c}=1$ \cite{DJ10,DJ13,DJ14,DJ15,Lin17}. The present status on this point is the following: if a strongly rational preunitary VOA extension $U$ of 
$L(1,0)$ has the same character of a VOA $V$ in \eqref{eq:rationalc=1} different from the $\operatorname{A}_5$-orbifold, then $U$ is isomorphic to $V$. The $\operatorname{A}_5$-orbifold case has been left open even assuming the rationality of the corresponding VOA.

The present status of knowledge in the setting of conformal nets is in a certain sense complementary to the one in the VOA setting. All the VOAs in \eqref{eq:rationalc=1} are known to be unitary and strongly local by \cite[Example 8.9]{CKLW18}.  There are also three other non rational unitary strongly local VOAs with $c=1$, namely the Heisenberg VOA $M(1)$, its $\mathbb{Z}_2$-orbifold  $M(1)^+$ and $L(1,0)$. 
In \cite{Xu05}, Xu gave a classification result for not necessarily rational $c=1$ conformal nets $\mathcal{A}$ satisfying a certain spectrum condition and conjectured that this condition is verified for every $c=1$ conformal net. Xu's spectrum condition can be described as follows: the representation of the Virasoro subnet of $\mathcal{A}$ on the vacuum Hilbert space $\mathcal{H}_\mathcal{A}$ is irreducible or it has some irreducible subrepresentations with lowest conformal energy $j^2$ with $j$ a non-zero integer.  In the language of \cite{CKLW18}, the classification result by Xu can be described as follows: every (irreducible) conformal net $\mathcal{A}$ with $c=1$ and satisfying the above spectrum condition must be isomorphic to one of the following conformal nets

\begin{equation}
	\mathcal{A}_{L(1,0)} \,,\,\,\,
	\mathcal{A}_{M(1)} \,,\,\,\, 
	\mathcal{A}_{M(1)^+} \,,\,\,\, 
	\mathcal{A}_{V_{\latt_2}^{\operatorname{A}_4}} \,,\,\,\, \mathcal{A}_{V_{\latt_2}^{\operatorname{S}_4}} \,,\,\,\,
	\mathcal{A}_{V_{\latt_2}^{\operatorname{A}_5}} \,,\,\,\, \mathcal{A}_{V_{\latt_{2n}}} \,,\,\,\,
	\mathcal{A}_{V_{\latt_{2n}}^+} \,,\,\,\, 
	n \in  \Zplus \,.
\end{equation}

The important special case of this classification, called {\it general discrete case} in \cite{Xu05}, was found independently by one of us in 
\cite{Car04}, where the corresponding nets are called {\it compact type} extensions. In both proofs of this general discrete case, the Doplicher-Roberts abstract duality for compact groups \cite{DR89,DR89b} and the corresponding field algebra extension of the observable net \cite{DR90}, cf.\ also \cite[Proposition 3.8]{Mug99}, play a crucial role. 

There is a clear analogue of Xu's spectrum condition in the VOA setting, which is the following: a preunitary VOA extension $V$ of $L(1,0)$ satisfies the {\it spectrum condition} if either $V=L(1,0)$ or $V$ has an irreducible $L(1,0)$ submodule isomorphic to $L(1,j^2)$ for some integer 
$j\neq 0$, cf.\ Definition \ref{defin:extensions_of_L(1,0)}. Since the appearance of \cite{Xu05}, there seems to be no relevant progress on the conjectured universal validity of the spectrum condition for $c=1$ conformal nets, even if one restricts to the very well-behaved class of completely rational conformal nets. The situation in the VOA setting is drastically different. In fact, it follows easily from Kiritsis' theorem, that every strongly rational preunitary VOA extension of $L(1,0)$ and $\tilde{c}=1$ satisfies the condition. Actually, as we will see in Proposition \ref{prop:every_strongly_rational_preunitary_exts_satisfies_spectrum_condition}, the spectrum condition follows directly from \cite{SS77}, which is also a crucial ingredient in \cite{Kir89}.  
This has two advantages. The first is simplification. The second is that one can drop the condition on the effective central charge $\tilde{c}$. This opens the way to look for a VOA version of the Xu classification in \cite{Xu05} in order to prove Conjecture \ref{con:c=1rational_aIntro}. This is what we are going to do in this paper, up to the rationality of $\operatorname{A}_5$-orbifold. Moreover, in analogy with \cite{Car04} and \cite{Xu05}, our results are not bound to the rational case. Somewhat surprisingly, it will not be necessary to use Deligne's abstract characterization of Tannakian categories \cite{Del90,Del02} as a non-unitary substitute for the Doplicher-Roberts duality theory in \cite{DR89,DR89b}. On the other hand, we are led to work in an algebraic geometric framework as in \cite{Del90,Del02}, and some tools from the theory of actions of algebraic groups play an essential role in our paper.

Let us now briefly describe our main strategy of proof and some of our main results. The starting point is that, in analogy with the conformal net situation, we have 
\begin{equation}
	L(1,0) = V_{\latt_2}^{\SO(3)} =  V_{\latt_2}^{\operatorname{PSL}(2,\C)} \, .
\end{equation}
This gives a vertex tensor category $\catc_{L(1,0)}$ of (ordinary) $L(1,0)$-modules which is ribbon tensor equivalent to the symmetric ribbon tensor category $\catc_{\SO(3)}$ of finite-dimensional continuous representations of the compact Lie group $\SO(3)$ (with  trivial twist) or equivalently  to the symmetric ribbon tensor category $\catc_{\operatorname{PSL}(2,\C)}$ of finite-dimensional rational (i.e.\ algebraic) representations of the complex reductive affine algebraic group $\operatorname{PSL}(2,\C)$. Note that $\SO(3)$ can be identified with a maximal compact subgroup of $\operatorname{PSL}(2,\C)$ through the isomorphism 
$\operatorname{PSL}(2,\C) \cong \SO(3,\mathbb{C})$. 

Now, in the present setting, the classification in the general discrete case in \cite{Xu05} can be approached in this way: classify all the simple CFT type VOA extensions of  $L(1,0)$ which, as $L(1,0)$-modules live in the direct limit completion 
$\Ind (\catc_{L(1,0)})$. We call these extensions {\it discrete series} VOA extensions of $L(1,0)$. In view of the correspondence between VOA extensions and algebra objects in vertex tensor categories \cite{KO02,HKL15,CMY22,CKM24}, this turns out to be equivalent to the classification of simple commutative algebra objects  in 
$\Ind (\catc_{L(1,0)})$ which are also ordinary $L(1,0)$-modules or  to the classification of simple commutative algebra objects in 
$\Ind( \catc_{\operatorname{PSL}(2,\C)})$ which are direct sums of at most countably many simple objects. We solve this problem for any arbitrary complex reductive affine algebraic group $G$, generalizing a result by Kirillov and Ostrik in \cite{KO02}. 

\begin{theo}[cf.\ Theorem \ref{theo:simple_commutative_algebras_in_Ind(C_G)}]
	\label{MainTheoremIntro1}
	Let $G$ be a complex reductive affine algebraic group $G$ and for any reductive algebraic subgroup $H \subseteq G$ let $\mathcal{O}(G/H)$ be the algebra of regular functions on the affine $G$-variety $G/H$. Then the map $H\mapsto \mathcal{O}(G/H)$ gives a one-to-one correspondence between the conjugacy classes of reductive algebraic subgroups of $G$ and isomorphism classes of simple algebra objects in $\Ind (\catc_{G})$ which are direct sums of at most countably many simple objects.
\end{theo}

In the final stage of the preparation of this paper, the preprint \cite{CS} was posted on the arXiv, where a stronger version of our Theorem \ref{MainTheoremIntro1} is given, see \cite[Theorem B]{CS}.

Using Theorem \ref{MainTheoremIntro1}, it is rather straightforward to see that, up to isomorphism, the only simple CFT type discrete series extensions of $L(1,0)$ are $L(1,0)$, $M(1)$, $M(1)^+$, $V_{\latt_{2n^2}}$, $V_{\latt_{2n^2}}^+$,  $n \in \Zplus$, $V_{\latt_2}^{\operatorname{S}_4}$, 
$V_{\latt_2}^{\operatorname{A}_4}$ and $V_{\latt_2}^{\operatorname{A}_5}$. Then, similarly to the conformal net case considered in \cite{Xu05}, in order  to classify all the 
simple CFT type preunitary VOA extensions of $L(1,0)$ satisfying the spectrum condition, we need to classify the corresponding extensions of 
$M(1)$, $M(1)^+$, $V_{\latt_2}^{\operatorname{S}_4}$, $V_{\latt_2}^{\operatorname{A}_4}$ and $V_{\latt_2}^{\operatorname{A}_5}$. The result is the following. 

\begin{theo}[cf.\ Theorem \ref{theo:classification_preunitary_spectrum_condition}]
	\label{MainTheoremIntro2}
	Assume that $V_{\latt_2}^{\operatorname{A}_5}$ is strongly rational. Then, every simple CFT type preunitary VOA extension of $L(1,0)$ satisfying the spectrum condition is isomorphic to one of the following VOAs
	\begin{equation}
		L(1,0)\,,\,\,\, 
		M(1)\,,\,\,\, 
		M(1)^+ \,,\,\,\, 
		V_{\latt_2}^{\operatorname{S}_4} \,,\,\,\, 
		V_{\latt_2}^{\operatorname{A}_4} \,,\,\,\, 
		V_{\latt_2}^{\operatorname{A}_5} \,,\,\,\, 
		V_{\latt_{2n}}\,,\,\,\, 
		V_{\latt_{2n}}^+ \,,\,\,\, 
		n \in \Zplus \,.
	\end{equation}
\end{theo} 

Note that we make no assumption on the effective central charge. Note also that the assumption on the strong rationality of the $\operatorname{A}_5$  orbifold is only needed to prevent the possible appearance in the list of hypothetical extensions of the latter which are not discrete series VOA extensions of $L(1,0)$. 

Now, recalling that every strongly rational preunitary VOA extension of $L(1,0)$ satisfies the spectrum condition, we get the following. 

\begin{theo}[cf.\ Theorem \ref{theo:classification_strongly rational_VOAs_cc_1}]
	\label{MainTheoremIntro3}
	Assume that $V_{\latt_2}^{\operatorname{A}_5}$ is strongly rational. Then the isomorphism classes of  strongly rational preunitary VOA extensions of $L(1,0)$ coincide with the isomorphism classes of the following VOAs
	\begin{equation}
		V_{\latt_2}^{\operatorname{S}_4}\,,\,\,\, 
		V_{\latt_2}^{\operatorname{A}_4} \,,\,\,\,
		V_{\latt_2}^{\operatorname{A}_5} \,,\,\,\, 
		V_{\latt_{2n}} \,,\,\,\, 
		V_{\latt_{2n}}^+ \,,\,\,\, 
		n \in \Zplus \,.
	\end{equation}
\end{theo}   

Note that also in Theorem \ref{MainTheoremIntro3} we make no assumptions on the effective central charge. If  $V_{\latt_2}^{\operatorname{A}_5}$ is not strongly rational then it must be removed from the list but possible hypothetical strongly rational VOA extensions of the latter could appear. These hypothetical strongly rational VOA extensions could be ruled out through Kiritsis' theorem if we make the assumption $\tilde{c}=1$ on the effective central charge, see Theorem \ref{theo:classification_strongly rational_VOAs_cc_cceff_1}. 

The results of this paper are mostly logically independent from the characterizations of Dong, Jiang and Lin in \cite{DJ10,DJ13,DJ14,DJ15,Lin17}, so that we obtain a new uniform proof of all these results. 

This paper is organized as follows. We start giving some useful preliminaries in Section \ref{sec:preliminaries}, also setting notations and terminology. In particular, we introduce commutative algebra objects in braided tensor categories in Section \ref{subsec:tensor_categories}. And we give some basics notions on VOAs, vertex tensor categories and their direct limit completions in Section \ref{subsec:voas_and_modules}.
In Section \ref{sec:simple_G-algebras}, we prove Theorem \ref{MainTheoremIntro1}, cf.\ Theorem \ref{theo:simple_commutative_algebras_in_Ind(C_G)}. 
This is done by classifying simple $G$-subalgebras of the commutative algebra $\mathcal{O}(G)$ of regular functions on a complex reductive affine algebraic group $G$, see Theorem \ref{theo:A_is_fnitely_generated}. 
In Section \ref{sec:DR_cross_product_voas}, we use $\mathcal{O}(G)$ to define a canonical simple VOA extension $W$ of a given VOA $V$ with a ribbon tensor category $\catc_V$ of $V$-modules, satisfying suitable conditions, see Subsection \ref{subsubsec:ind-categories}, and being ribbon tensor equivalent to $\catc_G$, see Theorem \ref{theo:DR_cross_product_VOAs}. 
Inspired by the Doplicher-Roberts duality theory \cite{DR89, DR89b}, and the corresponding field algebra extension of the observable net \cite{DR90}, we call $W$ the \textit{Doplicher-Roberts VOA extension} $V\rtimes \catc_V$ of $V$ by $\catc_V$.
It follows that the classification of simple $G$-subalgebras of $\mathcal{O}(G)$ given by Theorem \ref{theo:A_is_fnitely_generated} allows us to classify the simple vertex subalgebras of $V\rtimes \catc_V$ containing $V$. They turn out to be all fixed-point vertex subalgebras of $V\rtimes \catc_V$ by the reductive algebraic subgroups of $G$.
We apply this result to obtain a Galois correspondence for complex reductive affine algebraic groups acting rationally on VOAs by VOA automorphisms, see Corollary \ref{cor:galois_correspondence}.
Theorem \ref{MainTheoremIntro2}, cf.\ Theorem \ref{theo:classification_preunitary_spectrum_condition}, is proved in Section \ref{subsec:preunitary_spectrum_condition}. We reach this result by several intermediate classification steps: simple discrete series VOA extensions of $L(1,0)$ in Theorem \ref{theo:discrete_series_extensions}; simple CFT type VOA extensions of $L(1,0)$ containing a (primary) vector of conformal weight $1$ (conformal weight $4$) in Theorem \ref{theo:characterization_preunitary_with_currents} (Theorem \ref{theo:characterization_preunitary_exts_primary_cw4}). 
All these classification theorems are used to prove Theorem \ref{MainTheoremIntro3}, cf.\ Theorem \ref{theo:classification_strongly rational_VOAs_cc_1}.
Finally, in Section \ref{subsec:vertex_subalgebras_generalized_symmetries}, we give a classification result for non necessarily rational simple vertex subalgebras of rank-one lattice VOAs and we comment its relation with VOA generalized symmetries, see Remark \ref{rem:generalized_symmetries}. 

\bigskip
\noindent
	{\bf Note:}
A classification result for strongly rational $c=1$ VOAs has been independently obtained in \cite{GRb} by different methods.
We sincerely thank Terry Gannon and Brandon C. Rayhaun for coordinating the submission of \cite{GRb} to the arXiv on the same day.

\section{Preliminaries}  \label{sec:preliminaries}

\subsection{Tensor categories} \label{subsec:tensor_categories}

For the basic theory of \textit{tensor categories}, see e.g.\ \cite{Mac98, BK01, EGNO15}, see also \cite[Chapters I--II]{Tur16}. For us, a tensor category $\catc$ is a $\C$-linear abelian monoidal category.
In particular, for every objects $X,Y\in\catc$, the hom-space $\Hom(X,Y)$ is a complex vector space, where the composition and the tensor product of morphisms are $\C$-linear.
For the tensor category $\catc$, we will use the following symbols: $\oplus$ to denote the direct sum; $\one_X$ to denote the identity morphism of any object $X\in\catc$; $\boxtimes$ for the tensor product bifunctor and $\id_\catc$ for the tensor unit of $\catc$;
for $X,Y,Z\in\catc$, $a_{X,Y,Z}:X\boxtimes (Y\boxtimes Z)\to (X\boxtimes Y)\boxtimes Z$ is the associator, $l_X:\id_\catc\boxtimes X\to X$ and $r_X:X\boxtimes \id_\catc \to X$ are the left and right unitors respectively.
We do not require $\catc$ to be (locally) finite nor with simple tensor unit, although we will mostly work with categories satisfying these properties.
Justified by the applications to vertex operator algebras in Section \ref{subsec:voas_and_modules}, see \cite[Proposition 4.26]{HLZc}, we require that for every object $X\in\catc$, the functors $X\boxtimes \cdot$ and $\cdot\boxtimes X$ are right exact.

A non-zero object $X$ in a tensor category $\catc$ is \textit{simple} if any monomorphism $f:Y\to X$ with $Y\in\catc$ is either an isomorphism or the zero morphism. In other words, the zero object and $X$ itself are the only subobjects of $X$.
Then $\catc$ is \textit{semisimple} if every non-zero object is \textit{semisimple}, that is it can be written as a finite direct sum of simple ones. 
A tensor category $\catc$ is \textit{rigid} if every object $X\in\catc$ is \textit{rigid}, that is it has a left dual $X^*\in\catc$ and a right dual $\prescript{*}{}{X}\in\catc$ satisfying the rigidity axioms, see e.g.\ \cite[Eq.s (2.43)--(2.46)]{EGNO15}. 
A non-zero object $X$ in a tensor category $\catc$ is \textit{invertible} if there exists an object $X^\vee\in\catc$ such that $X\boxtimes X^\vee$ and $X^\vee\boxtimes X$ are isomorphic to the tensor unit $\id_\catc$.
Moreover, $X$ is called a \textit{$\Z_n$-simple current} for some $n\in\Z_{>1}$ if it is simple and $X^\vee\cong X^{n-1}$ and $X^m\not\cong \id_\catc$ for all positive integer $m<n$.
It follows from the natural adjunction isomorphisms \cite[Proposition 2.10.8]{EGNO15} that if $X$ is rigid and invertible, then $X^*\cong X^\vee\cong \prescript{*}{}{X}$.
Furthermore, if $X$ is rigid and invertible, the endofuctor of $\catc$ given by $\cdot\mapsto \cdot\boxtimes  X$ is an autoequivalence of $\catc$, so that $X$ is simple if and only if $\id_\catc$ is simple. 

A tensor category is called \textit{braided} if it is equipped with a \textit{braiding}, that is a natural isomorphism $b_{X,Y}:X\boxtimes Y\to Y\boxtimes X$ satisfying the usual hexagon axioms, see e.g.\ \cite[Eq.s (8.1)--(8.2)]{EGNO15}.
Then a braided tensor category $\catc$ is \textit{symmetric} if the \textit{monodromy} $\mathcal{M}_{X,Y}:=b_{Y,X}b_{X,Y}$  is trivial for every objects $X,Y\in\catc$.
In a braided tensor category $\catc$, a \textit{twist} (or balancing transformation) is a natural automorphism $\vartheta$ of the identity functor of $\catc$ such that $\vartheta_{\id_\catc}=\one_{\id_\catc}$ and $\vartheta_{X\boxtimes Y}=(\vartheta_X\boxtimes \vartheta_Y)\mathcal{M}_{X, Y}$ for all $X,Y\in\catc$.
In this case, $\catc$ is called a \textit{balanced} tensor category.
A \textit{ribbon} tensor category is a rigid braided tensor category equipped with a ribbon structure, that is a twist $\vartheta$ such that $\vartheta_{X^*}=(\vartheta_X)^*$ for all $X\in\catc$. 
For a ribbon tensor category $\catc$, we have a notion of \textit{categorical} (or quantum) \textit{dimension} $d(X)$ for any object $X\in\catc$, which is a complex number if $\catc$ has simple tensor unit, see e.g.\ \cite[Definition 4.7.11]{EGNO15}.
In this case, $\catc$ is called \textit{positive} if all dimensions are positive real numbers. 
For a \textit{fusion category}, it is meant a semisimple rigid tensor category with simple tensor unit and with a finite number of equivalence classes of simple objects.
In this case, left and right duals coincide, see e.g.\ \cite[Proposition 4.8.1]{EGNO15}.
Lastly, a \textit{modular} tensor category is a ribbon fusion category with non-degenerate $S$-matrix, see e.g.\ \cite[Definition 8.13.4]{EGNO15}.

A tensor equivalence between modular/ribbon/balanced/braided tensor categories is said to be \textit{modular}/\textit{ribbon}/\textit{balanced}/\textit{braided} if it preserves the corresponding additional structures.

\begin{defin}  \label{defin:algebras_in_catc}
	Let $\catc$ be a braided tensor category. Then a \textit{commutative algebra} in $\catc$ is a triple $A=(X,m,\iota)$, where $X$ is an object of $\catc$, $m\in\Hom(X\boxtimes X, X)$ is the multiplication morphism and $\iota\in\Hom(\id_\catc, X)$ is the unit morphism, satisfying: 
	\begin{itemize}
		\item $\iota$ is an \textit{injective} morphism;
		
		\item \textit{unit property}: $m(\iota\boxtimes \one_X)=l_X$;
		
		\item \textit{associativity}: $m(\one_X\boxtimes m)=m(m\boxtimes \one_X)a_{X,X,X}$;
		
		\item \textit{commutativity}: $m=mb_{X,X}$.
	\end{itemize}
	$A$ is said to be \textit{haploid} if $\Hom(\id_\catc, X)\cong\C$. Furthermore, if $\catc$ has a twist, then the \textit{twist} $\vartheta_A$ of $A$ is the twist of the object $X$ in $\catc$. Similarly, the \textit{categorical dimension} $d(A)$ of $A$ is the categorical dimension of the object $X$ in $\catc$, whenever the latter exists.
\end{defin}

\begin{rem}
	We point out that, in some references which we are going to use in the following, such as \cite{CKM24, CMY22}, the unit morphism of a commutative algebra is not required to be injective by definition. Nevertheless, we will not need this level of generalization.
\end{rem}

Note that for a commutative algebra $A=(X,m,\iota)$, the unit property and the commutativity imply also that $m(\one_X \boxtimes \iota)=r_X$.

\begin{defin}
	Let $A=(X, m, \iota)$ be a haploid commutative algebra in a ribbon tensor category $\catc$ with simple tensor unit. Let $\varepsilon$ be the unique morphism in $\Hom(X,\id_\catc)$ such that $\varepsilon\iota=\one_{\id_\catc}$.
	Then $A$ is said to be \textit{rigid} if $d(A)\not=0$ and the \textit{evaluation} morphism $e_A:=\varepsilon m:X\boxtimes X\to\id_{\catc}$ admits a \textit{coevaluation} morphism $i_A:\one_{\id_\catc}\to X\boxtimes X$ satisfying the rigidity axioms.
\end{defin}

\begin{defin}   \label{defin:aut_group_commutative_algebras_in_C}
	Let $A=(X, m, \iota)$ and $B=(Y,n,\varsigma)$ be commutative algebras in a braided tensor category $\catc$. 
	Then $A$ and $B$ are said to be \textit{isomorphic} if there exists an invertible morphism $\phi\in\Hom(X,Y)$ such that $\phi m=n(\phi\boxtimes \phi)$ and $\phi\iota=\varsigma $.
	Therefore, the \textit{automorphism group} of a commutative algebra $A$ is defined by
	\begin{equation}
		\Aut(A):=\left\{\phi\in\Hom(X,X) \mid \phi \mbox{ is invertible, }\,\, \phi m= m(\phi\boxtimes \phi) \,,\,\, \phi\iota=\iota \right\} \,.
	\end{equation} 
\end{defin}

\begin{defin}  
	Let $A=(X,m,\iota)$ be a commutative algebra in a braided tensor category $\catc$. Then $\Rep(A)$ is the category of (\textit{left}) \textit{$A$-modules} defined in the following way. Objects are pairs $(Y,m_Y)$, where $Y$ is an object of $\catc$ and $m_Y\in\Hom(X\boxtimes Y, Y)$ satisfy:
	\begin{itemize}
		\item \textit{unit property}: $m_Y(\iota\boxtimes \one_Y)=l_Y$;
		
		\item \textit{module property}: $m_Y(\one_X\boxtimes m_Y)=m_Y(m\boxtimes \one_Y)a_{X,X,Y}$.
	\end{itemize} 	
	Morphisms between objects $(Y_1,m_{Y_1})$ and $(Y_2,m_{Y_2})$ are given by $f\in\Hom(Y_1, Y_2)$ such that $m_{Y_2}(\one_X\boxtimes f)=fm_{Y_1}$.
	If it is not clear from the context, $\Rep(A)$ is denoted by $\Rep_\catc(A)$. 
\end{defin}

Let $A=(X,m,\iota)$ be a commutative algebra in a braided tensor category $\catc$.
It is known, see \cite[Theorem 2.53]{CKM24}, cf.\ also \cite[Proposition 1.4]{Par95} and \cite[Theorem 1.5]{KO02}, that $\Rep(A)$ is a tensor category, but it is not braided in general.
On the other hand, $\Rep^0(A)$, defined below, turns out to be a braided tensor category, see \cite[Remark 2.57]{CKM24}, cf.\ also \cite[Theorem 2.5]{Par95} and \cite[Theorem 1.10]{KO02}. 

\begin{defin}
	Let $A=(X,m,\iota)$ be a commutative algebra in a braided tensor category $\catc$.
	Then an $A$-module $(Y,m_Y)$ is said to be \textit{local} if $m_Y\mathcal{M}_{X,Y}=m_Y$. Then $\Rep^0(A)$ is the full subcategory of $\Rep(A)$ of local $A$-modules.
	If it is not clear from the context, $\Rep^0(A)$ is denoted by $\Rep_\catc^0(A)$. 
\end{defin}

\begin{defin}
	Let $A=(X,m,\iota)$ be a commutative algebra in a braided tensor category $\catc$.
	Then $A$ is said to be \textit{simple} if the $A$-module $(X,m_X:=m)$ is a simple object in $\Rep(A)$. 
\end{defin}

\begin{rem}  \label{rem:equivalence_categories_algebras}
	Let $\catc$ and $\catd$ be two braided tensor categories. 
	Let $T:\catc\to\catd$ be a braided tensor equivalence with tensorator $t_2:T(X)\boxtimes T(Y)\to T(X\boxtimes Y)$ for all objects $X,Y\in\catc$ and unitor $t_0:\id_\catd\to T(\id_\catc)$.
	Then it is a routine check to show that there is a one-to-one correspondence between equivalence classes of commutative algebras in $\catc$ and in $\catd$ given by 
	$$
	A=(X,m,\iota)\longmapsto T(A):=(T(X),T(m)t_2, T(\iota)t_0)
	\,.
	$$
	Note also that $A$ is haploid if and only if $T(A)$ is haploid. 
	Moreover, the groups $\Aut(A)$ and $\Aut(T(A))$ are isomorphic as $T$ is fully faithful.  
	We also have a braided tensor equivalence between $\Rep^0(A)$ and $\Rep^0(T(A))$, see \cite[Proposition A.1]{MY23} for details. In particular, this implies that $A$ is simple if and only if $T(A)$ is simple, see \cite[Corollary A.2]{MY23}.
	Categorical dimensions and rigidity are also preserved if $T$ is a ribbon tensor equivalence between ribbon tensor categories.
\end{rem}

We introduce the following fundamental tool:

\begin{defin} \label{defin:induction_restriction_functors}
	Let $A=(X,m,\iota)$ be a commutative algebra in a braided tensor category $\catc$. The \textit{induction functor} $\F:\catc\to \Rep(A)$ is defined on objects and morphisms respectively by
	\begin{equation}  \label{eq:induction_functor}
		\F(Y):=(X\boxtimes Y, (m\boxtimes \one_Y)a_{X,X,Y})
		\,,\,\,\,
		\F(f):=\one_X\boxtimes f \,.
	\end{equation}
\end{defin}

The induction functor is a tensor functor, see \cite[Theorem 1.6]{KO02} and \cite[Theorem 2.59]{CKM24}. It also satisfies the following properties: 

\begin{rem} \label{rem:properties_induction_functor}
	Let $A=(X,m,\iota)$ be a commutative algebra in a braided tensor category $\catc$ and consider the induction functor $\F:\catc\to \Rep(A)$.
	If $\catc^0$ is the full subcategory of $\catc$ whose objects induce to $\Rep^0(A)$, then $\catc^0$ is a braided tensor category and $\F:\catc^0\to\Rep^0(A)$ is a braided tensor functor, see \cite[Theorem 2.67]{CKM24}. 
	Therefore, $\F(\catc^0)$ is symmetric if $\catc^0$ is, see \cite[Corollary 2.68]{CKM24}.
	It is useful to recall from \cite[Proposition 2.65]{CKM24} that an object $Y\in \catc$ is in $\catc^0$ if and only if the monodromy $\mathcal{M}_{X, Y}$ is equal to $\one_{X\boxtimes Y}$.
	We also point out that if $Y\in\catc$ has a left dual $Y^*$, then $\F(Y^*)$ is a left dual for $\F(Y)$.
	Moreover, $\F(Y^*)$ is an object of $\Rep^0(A)$ if $Y\in\catc^0$. Similar statements hold for right duals, see \cite[Proposition 2.77 and Lemma 2.78]{CKM24} and references therein.
	Suppose that $\catc$ has a twist $\vartheta$ such that $\vartheta_X=\one_X$ and let $Z:=(Y,m_Y)$ be in $\Rep(A)$. 
	Then we have that, see \cite[Lemma 2.81 and Corollary 2.82]{CKM24} and references therein: 
	$Z$ is in $\Rep^0(A)$ if and only if $\vartheta_Y\in\Hom_{\Rep(A)}(Z, Z)$; 
	$\vartheta_Z:=\vartheta_Y$ gives a twist on $\Rep^0(A)$; 
	if $Y$ is in $\catc^0$, then $\F(\vartheta_Y)=\vartheta_{\F(Y)}$.
	Therefore, if $\catc$ is a ribbon tensor category such that the twist $\vartheta_X=\one_X$, then $\F(\catc^0)$ is a ribbon tensor category and $\F:\catc^0\to\Rep^0(A)$ is a ribbon tensor functor, see \cite[Proposition 2.86]{CKM24}.
	Thanks to \cite[Lemma 2.87]{CKM24}, if $\catc^0$ is a ribbon tensor category with simple tensor unit, then $\Rep^0(A)$ is so and $d(\F(Y))=d(Y)$ for all $Y\in \catc^0$.
\end{rem}

\begin{rem}
	The induction functor in Definition \ref{defin:induction_restriction_functors} is well known as \textit{$\alpha$-induction} in the framework of the operator algebraic approach to chiral conformal field theory. 
	See e.g.\ \cite{LR95, BE98} for details.
\end{rem}

\subsection{VOA representation theory} \label{subsec:voas_and_modules}

For a \textit{vertex algebra} $V$, see e.g.\ \cite{FLM88, FHL93, Kac01} and \cite{LL04} for the basic theory, we denote the \textit{vacuum vector} by $\Omega$ and the \textit{state-field correspondence} giving the \textit{vertex operators} by
\begin{equation}
	V\ni a \longmapsto Y(a,z)=\sum_{n\in \Z}a_{(n)} z^{-n-1}
	\,,\quad a_{(n)}\in \End(V) \,.
\end{equation}
We specify that the vector space $V$ has no further structures. In other words, we do not require our vertex algebras to be non-trivially $\Z_2$-graded by a parity operator (where $\Z_2$ denotes the cyclic group of order two).

\begin{defin}
	A \textit{vertex operator algebra (VOA)} is a vertex algebra $V$ equipped with a \textit{conformal vector} $\nu\in V$ such that the corresponding vertex operator 
	\begin{equation}
		Y(\nu,z)=\sum_{n\in \Z}L_nz^{-n-2}
	\end{equation}
	satisfies:
	\begin{itemize}
		\item the \textit{Virasoro algebra commutation relations} for a given \textit{central charge} $c\in\C$, that is
		\begin{equation}  \label{eq:virasoro_alg_crs}
			\forall n,m\in\Z \qquad
			[L_n,L_m]=(n-m)L_{n+m}+\frac{c(n^3-n)}{12}\delta_{n,-m}1_V \,;
		\end{equation}
		
		\item that $L_{-1}$ is the \textit{infinitesimal translation operator} of $V$, so that 
		\begin{equation}
			L_{-1}\Omega=0 
			\quad\mbox{and}\quad
			[L_{-1}, Y(a,z)]=\frac{\mathrm{d}}{\mathrm{d}z}Y(a,z)
			\,;
		\end{equation}
		
		\item that the \textit{conformal Hamiltonian} $L_0$ is diagonalizable on $V$ with finite-dimensional eigenspaces $V_n$ of integer eigenvalues $n$, which are trivial whenever $n\leq N$ for some negative number $N$.
	\end{itemize}
	Any vector $\hat{\nu}\in V$ satisfying \eqref{eq:virasoro_alg_crs} with $Y(\hat{\nu},z)=\sum_{n\in \Z}\hat{L}_nz^{-n-2}$ is called a \textit{Virasoro vector} with central charge $c$.
\end{defin}
	
	Any vector $a\in V$ such that $L_0a=na$ for some $n\in\Z$ is called \textit{homogeneous} of \textit{conformal weight} $d_a:=n$. 
	For all homogeneous vectors $a\in V$ and all $n\in\Z$, we set $a_n:=a_{(n+d_a-1)}$ and we write $Y(a,z)=\sum_{n\in\Z}a_nz^{-n-d_a}$. 
	Moreover, for all $a\in V$, we can write $Y(z^{L_0}a,z)=\sum_{n\in\Z}a_nz^{-n}$.
	Note that for all $n\in\Z$, $\nu_n=L_n$. 
	With this notation, the \textit{skew-symmetry} \cite[Section 4.2]{Kac01} becomes: for all homogeneous vectors $a,b\in V$ and all $n\in\Z$,
	\begin{equation} \label{eq:skew-symmetry}
		a_nb=\sum_{j=0}^{+\infty} (-1)^{j+n+d_a}L_{-1}^j(b_{j+n+d_a-d_b}a) \,;
	\end{equation}
	whereas the \textit{Borcherds commutator formula} \cite[Section 4.8]{Kac01} is given by: for all homogeneous vectors $a,b,c\in V$ and all $n,m\in\Z$,
	\begin{equation}  \label{eq:borcherds_commutator_formula}
		[a_n,b_m]c=\sum_{j=0}^{+\infty}\binom{n+d_a-1}{j}(a_{j-d_a+1}b)_{n+m}c \,.
	\end{equation}
	From the Virasoro algebra commutation relations \eqref{eq:virasoro_alg_crs}, we see that for all homogeneous $a,b\in V$ and all $n\in \Z$, $a_nb\in V_{d_b-n}$.
	An homogeneous vector $a\in V$ is called \textit{primary} if $L_na=0$ for all $n>0$, whereas it is called \textit{quasi-primary} if $L_1a=0$.
	Accordingly, $\Omega$ is a primary vector of conformal weight $0$; whereas $\nu$ is a quasi-primary vector of conformal weight $2$, which is never primary unless $c=0$.
	A VOA is said to be of \textit{CFT type} if $V_n=\{0\}$ for all $n<0$ and $V_0=\C\Omega$. 
	
\begin{defin}  \label{defin:VOA_isomorphisms}
	Let $V$ and $\tilde{V}$ be two VOAs with vacuum and conformal vectors $\Omega$, $\nu$ and $\tilde{\Omega}$, $\tilde{\nu}$ respectively. A linear vector space map $\phi:V\to \tilde{V}$ is said to be a \textit{VOA isomorphism} from $V$ to $\tilde{V}$ if $\phi(\Omega)=\tilde{\Omega}$, $\phi(\nu)=\tilde{\nu}$ and for all $a,b\in V$ and all $n\in\Z$, $\phi(a_nb)=\phi(a)_n\phi(b)$. If $\tilde{V}=V$, then $\phi$ is called a \textit{VOA automorphism} of $V$ and the group of VOA automorphisms of $V$ is denoted by $\Aut(V)$. 	
\end{defin}

\begin{rem}  \label{rem:topology_VOA_automorphism_group}
	The group of $L_0$-grading-preserving linear vector space automorphisms of a VOA $V$ is the direct product $\prod_{n\in\Z}\GL(V_n)$ of finite-dimensional Lie groups $\GL(V_n)$ for $n\in\Z$. It is a metrizable topological group when endowed with the product topology. 
	Hence, $\Aut(V)$ becomes a topological closed subgroup of $\prod_{n\in\Z}\GL(V_n)$, when endowed with the relative topology, see e.g.\ \cite[Section 4.3]{CKLW18} for details. Cf.\ also \cite{DG02}.
\end{rem}

\begin{rem} \label{rem:fixed-point_vertex_subalgebras}
	Let $V$ be a VOA and $G$ be a closed subgroup of $\Aut(V)$. The vector subspace
	\begin{equation}
		V^G:=\{a\in V\mid \forall\phi(\phi\in G) \,\,\,\Rightarrow\,\,\, \phi(a)=a\}
	\end{equation}
	is a \textit{vertex subalgebra} of $V$, see e.g.\ \cite[Section 4.3]{Kac01}, containing the conformal vector of $V$, which gives it a VOA structure. $V^G$ is called the \textit{fixed-point vertex subalgebra} of $V$ with respect to $G$.
\end{rem}

We work with the following standard definitions of VOA modules and their intertwining operators.

\begin{defin}
	Let $V$ be a VOA.
\begin{itemize}
	\item 
	A \textit{weak $V$-module} is a vector space $M$ equipped with vertex operators 
	\begin{equation}
		V\ni a \longmapsto Y^M(a,z)=\sum_{n\in \Z}a_{(n)}^M z^{-n-1}
		\,,\quad a_{(n)}^M\in \End(M)
	\end{equation}
	satisfying:
	\begin{itemize}
		\item the \textit{lower truncation property}, 
		that is for all $a\in V$ and all $v\in M$, there exists $t\in\Z$ such that $a_{(n)}^M v=0$ for all $n\geq t$;
		
		\item the \textit{vacuum property}, 
		that is $Y^M(\Omega,z)=1_M$;
		
		\item the \textit{Borcherds formula}, 
		that is for all $a,b\in V$, all $v\in M$ and all $m,n,k\in\Z$, it holds that
		\begin{align}  \label{eq:jacobi_id_module}
				\sum_{j=0}^\infty & \binom{m}{j} 
				\left(a_{(n+j)}b\right)^M_{(m+k-j)}v
				= \\
				&\sum_{j=0}^\infty(-1)^j\binom{n}{j}
				\left(
				a^M_{(m+n-j)}b^M_{(k+j)}
				-(-1)^n b^M_{(n+k-j)} a^M_{(m+j)}
				\right)v \,;
		\end{align}
		
		\item the \textit{$L_{-1}$-derivative property}, 
		that is for all $a\in V$, it holds that
		\begin{equation}
			Y^M(L_{-1}a,z)=\frac{\mathrm{d}}{\mathrm{d}z}Y^M(a,z) \,.
		\end{equation}
	\end{itemize}
	In particular, $Y^M(\nu,z)=\sum_{n\in\Z}L_n^Mz^{-n-2}$.
	
	\item A weak $V$-module $M$ is said to be \textit{$C_1$-cofinite} if the vector space
	\begin{equation}
		\lcoset{M}{C_1(M)}
		\,,\qquad
		C_1(M):=\mathrm{span}\{ u_{(-1)}v\mid u,v\in M\}
	\end{equation} 
	is finite-dimensional.
	
	\item 
	An \textit{admissible $V$-module} is a weak $V$-module $M$ which is also \textit{$\Zpluseq$-gradable}, that is if there exists a $\Zpluseq$-grading 
	\begin{equation}
		M=\bigoplus_{n\in\Zpluseq} M(n)
	\end{equation}
	such that for all homogeneous $a\in V$ and all $m\in \Z$, 
	\begin{equation}
		a_{(m)}^M M(n)\subseteq M(d_a+n-m-1) \,.
	\end{equation}
	
	\item 
	An \textit{(ordinary) $V$-module} is a weak $V$-module $M$ such that 
	\begin{equation}  \label{eq:decomposition_ordinary_modules}
		M=\bigoplus_{n\in \C}M_n
		\,,\qquad
		M_n:=\{v\in M\mid (L_0^M-n1_M)v=0\}
	\end{equation}
	and that it is also \textit{grading-restricted}, that is for all $n\in\C$, $M_n$ is finite-dimensional and there exists $t\in\Z$ such that $M_{n+m}=\{0\}$ whenever $\Z\ni m\leq t$.  
\end{itemize}
\end{defin}

\begin{defin} \label{defin:intertwining_operators}
	Let $V$ be a VOA and let $M$, $N$ and $O$ be $V$-modules. 
	Then an \textit{(ordinary) intertwining operator} $\Y$ of type $\binom{O}{M\,\, N}$ is a linear map
	\begin{equation}
		M \ni v \longmapsto
		\Y(v,z)=\sum_{n\in\C}
		v_n z^{-n-1}
		\,,\quad v_n\in\Hom(N,O)
	\end{equation}	
	satisfying:
	\begin{itemize}
		\item the \textit{lower truncation property}, 
		that is for all $v\in M$ and all $w\in N$, there exists $t\in\R$ such that $v_n w=0$ whenever $\operatorname{Re}(n)\geq t$;
		
		\item the \textit{Borcherds formula}, 
		that is for all $a\in V$, all $v\in M$, all $w\in N$, all $m,n\in\Z$ and all $k\in\C$, it holds that
		\begin{align}  \label{eq:jacobi_id_intertwining}
				\sum_{j=0}^\infty & \binom{m}{j} 
				\left(a_{(n+j)}^M v\right)_{m+k-j}w
				= \\
				&\sum_{j=0}^\infty(-1)^j\binom{n}{j}
				\left(
				a^O_{(m+n-j)}v_{k+j}
				-(-1)^n v_{n+k-j} a^N_{(m+j)}
				\right)w \,;
		\end{align}
		
		\item the \textit{$L_{-1}$-derivative property}, 
		that is for all $v\in M$, it holds that
		\begin{equation}
			\Y(L_{-1}^Mv,z)=\frac{\mathrm{d}}{\mathrm{d}z}\Y(v,z) \,.
		\end{equation}
	\end{itemize}
	If $O$ is not grading-restricted, then the lower truncation property above is relaxed in the following way:
	\begin{itemize}
		\item for all $n\in\C$, all $v\in M$ and all $w\in N$, there exists $t\in\Z$ such that $v_{n+k} w=0$ for all $k\geq t$ and all $m\in\Zpluseq$.
	\end{itemize}
	The vector space of intertwining operators of type $\binom{O}{M\,\, N}$ is denoted by $\mathcal{V}_{M\, N}^O$ and its dimension, denoted by $I_{M\, N}^O$, is called the \textit{fusion rule} for $M$, $N$ and $O$.
\end{defin}

Any VOA $V$ is clearly a $V$-module, called the \textit{adjoint module} of $V$ and denoted by the same symbol.
A $V$-module is \textit{irreducible} if its only $V$-submodules are $\{0\}$ and the adjoint module.
Therefore, $V$ is \textit{simple} if its adjoint module is irreducible.

\begin{rem}  \label{rem:ordinary_modules_are_admissible}
	Let $M$ be a $V$-module and for every $n\in\C$ such that $M_n\not=\{0\}$, consider $\lambda$ as the minimum of the set $\{m\in (\Z+n)\mid M_m\not=\{0\}\}$, which exists by the grading restriction condition. Call $\Lambda$ as the set of those $\lambda$ so obtained. For every $n\in\Zpluseq$, define $M(n):=\bigoplus_{\lambda\in\Lambda}M_{\lambda+n}$. Then $M=\bigoplus_{n\in\Zpluseq}M(n)$ and $M$ is an admissible $V$-module. Indeed, for all $a\in V$, all $m\in\Z$ and all $\lambda\in\Lambda$, $a_{(m)}M_\lambda\subseteq M_{(\lambda+d_a-m-1)}$, see \cite[Lemma 3.4]{DLM98} for details. In particular, note that if $M$ is irreducible, then $\Lambda$ has one element only, which we call the \textit{conformal weight} $\lambda_M$ of $M$, so that $M=\bigoplus_{n\in\Zpluseq}M_{\lambda_M+n}$.
\end{rem}

Recalling Remark \ref{rem:ordinary_modules_are_admissible}, we give the following important definition, see \cite{Zhu96}:

\begin{defin} \label{defin:q-character}
	Let $V$ be a VOA of central charge $c$. Then for any irreducible $V$-module $M$, the \textit{$q$-character} of $M$ is the formal trace function
	\begin{equation} \label{eq:q-character}
		\chi_M(q):= \operatorname{tr}_M(q^{L_0^M-\frac{c}{24}})
		=	q^{\lambda_M-\frac{c}{24}}
		\sum_{n\in\Zpluseq} \dim(M_{\lambda_M+n})q^n \,.
	\end{equation}
\end{defin}

\begin{rem}  \label{rem:q-character}
	Note that the $q$-character of the adjoint module of a VOA $V$ gives the dimensions of the eigenspaces of the conformal Hamiltonian $L_0$ of $V$. Accordingly, we will refer to it as the $q$-character of $V$.
\end{rem}

For every $V$-module $M$, we can define its \textit{contragredient module} $M'$, see \cite[Section 5.2]{FHL93}. This is the graded dual vector space of $M$, that is $M':=\prod_{n\in \C}M_n^*$, where $M_n^*$ is the usual dual vector space of $M_n$, equipped with the vertex operators defined by the formula:
\begin{equation} \label{eq:contragredient_module}
	\forall a\in V
	\,\,\,
	\forall b\in M'
	\,\,\,
	\forall c\in M
	\qquad
	\langle Y^{M'}(a,z)b, c\rangle 
	= \langle b, Y^{M}(e^{zL_1}(-z^{-2})^{L_0}a,z^{-1})c) \rangle
\end{equation} 
where $\pairing$ is the natural pairing between $M'$ and $M$.
A $V$-module $M$ is irreducible if and only if its contragredient module $M'$ is irreducible, see \cite[Proposition 5.3.2]{FHL93}.
Then $M$ is said to be \textit{self-contragredient} if it is isomorphic to $M'$.
Accordingly, a VOA $V$ is said to be self-contragredient if its adjoint module $V$ is isomorphic to its contragredient module $V'$. 

\begin{rem}  \label{rem:invariant_bilinear_form}
	The self-contragredient condition on a VOA $V$ is equivalent to the existence of a non-degenerate \textit{invariant bilinear form} on $V$, see \cite[Remark 5.3.3]{FHL93}, that is a $\C$-bilinear form $\bilinear$ on $V$ satisfying the following \textit{invariant property}: for all homogeneous vectors $a\in V$, all $b,c\in V$ and all $n\in\Z$,
	\begin{equation} \label{eq:invariant_bilinear_form}
		(a_nb,c)=(-1)^{d_a}\sum_{j=0}^{+\infty}\frac{(b,(L_1^ja)_{-n}c)}{j!} \,.
	\end{equation}
	This property implies that $(V_n,V_m)=0$ whenever $n,m\in \Z$ are such that $n\not=m$. 
	Moreover, every invariant bilinear form is automatically symmetric, see \cite[Proposition 2.6]{Li94}.
	By \cite[Theorem 3.1]{Li94}, the vector space of invariant bilinear forms on a VOA $V$ is naturally isomorphic to the dual vector space of $\lcoset{V_0}{L_1V_1}$.
	An invariant bilinear form $\bilinear$ on a VOA $V$ is said to be \textit{normalized} if $(\Omega,\Omega)=1$.
	Then if $V$ is a simple VOA, then every non-zero invariant bilinear form is non-degenerate and it is unique up to normalization, \cite[Proposition 4.6 (iii)]{CKLW18}.
	Conversely, if $V$ has a non-degenerate invariant bilinear form and it is such that $V_0=\C\Omega$, then it is also simple, \cite[Proposition 4.6 $(iv)$]{CKLW18}. 
\end{rem}

A VOA $V$ is said to be \textit{$C_2$-cofinite} if the vector space $\lcoset{V}{C_2(V)}$ is finite-dimensional, where $C_2(V):=\mathrm{span}\{ a_{(-2)}b\mid a,b\in V\}$.
A VOA $V$ is said to be \textit{rational}, resp.\ \textit{regular}, if every admissible, resp.\ weak, $V$-module is a direct sum of irreducible admissible, resp.\ ordinary, $V$-modules. 
By \cite{ABD04}, a VOA of CFT type is regular if and only if it is rational and $C_2$-cofinite. 
From \cite[Theorem 8.1]{DLM98}, we recall that a rational VOA $V$ has a finite number of isomorphism classes of irreducible admissible $V$-modules, which are in turn also ordinary. 

\begin{defin}
	A VOA is said to be \textit{strongly rational} if it is of CFT type, self-contragredient, $C_2$-cofinite and rational (so that it is also simple and regular).
\end{defin} 

In \cite[Theorem 1.1]{DLM00}, it is proved that the central charge and the conformal weights of irreducible VOA modules of strongly rational VOAs are rational numbers. 

\begin{defin}  \label{def:effective}
	Let $V$ be a strongly rational VOA of central charge $c$. 
	Let $\lambda_\mathrm{min}$ be the minimum conformal weight among the irreducible $V$-modules. Then the \textit{effective central charge} of $V$ is defined by $\tilde{c}:=c-24\lambda_\mathrm{min}$.
\end{defin}

Note that for a strongly rational VOA $V$ of central charge $c$, we have that $\lambda_\mathrm{min}\leq 0$ as the adjoint module has conformal weight $0$. Then the effective central charge $\tilde{c}$ is greater than or equal to $c$, and it is equal to $c$ if the following condition holds:

\begin{defin}
	Let $V$ be a strongly rational VOA.
	Then $V$ is said to satisfy the \textit{positivity condition on modules} if every irreducible $V$-module $M$, non-isomorphic to the adjoint one, has positive conformal weight, that is $\lambda_M>0$.
\end{defin}

A key aspect of strongly rational VOAs is the categorical properties of their representation theory that we are going to present in the following.

\begin{defin}
	Let $V$ be a VOA, then $\catw_V$ denotes the $\C$-linear abelian category whose objects and morphisms are given by weak $V$-modules along with their $V$-module homomorphisms respectively. Furthermore, $\Rep(V)$ denotes the full subcategory of $\catw_V$ whose objects are ordinary $V$-modules.
\end{defin}

Let $V$ be a strongly rational VOA.
Then $\Rep(V)$ has a structure of modular tensor category, see \cite{Hua08} and references therein. 

\begin{rem}  \label{rem:some_facts_vtc}
	We point out some useful facts about the categorical structure of $\Rep(V)$ for a strongly rational VOA $V$.
	The tensor unit is the adjoint module of $V$, so that the former is simple if and only if the latter is irreducible, that is $V$ is a simple VOA. For a $V$-module $M$, the twist $\vartheta_M$ is the operator $e^{i2\pi L_0^M}$, which is equal to $e^{i2\pi \lambda_M}$ when $M$ is irreducible, see Remark \ref{rem:ordinary_modules_are_admissible}.
	Moreover, the ribbon structure is given by the operation of taking contragredient modules, see \eqref{eq:contragredient_module}.
	If $V$ satisfies the positivity condition on modules, then $\Rep(V)$ is also positive as the quantum dimension of every irreducible $V$-module, defined by using the q-characters \cite[Definition 3.1 and Remark 3.2]{DJX13}, is a positive number \cite[Lemma 4.2]{DJX13} and it equals the categorical dimension \cite[Porposition 3.11]{DLN15}.
\end{rem}

Without going into details, we point out the fundamental fact that this categorical structure is inherited from a richer \textit{vertex tensor category} structure, see \cite[Section 2]{HKL15} and references therein, which encodes the complex analytic features of intertwining operators among $V$-modules. 
Despite we will not directly work with this structure in this paper, it is crucial for the results on VOA extensions that we are going to recall.

\begin{defin}
	Let $V$ and $U$ be VOAs such that $V$ is a vertex subalgebra of $U$. Then $U$ is said to be a \textit{VOA extension} of $V$ if they share the same conformal vector. Furthermore, $U$ is said to be a \textit{simple}, resp.\ \textit{CFT type}, resp.\ \textit{strongly rational}, VOA extension of $V$ if $U$ is simple, resp.\ of CFT type, resp.\ strongly rational, as VOA.  
\end{defin}

Note that any VOA extension can be naturally considered as a module over the VOA that it extends, so that they will be denoted by the same symbol. 
If $V$ is a strongly rational VOA, then there is a one-to-one correspondence between (simple) VOA extensions $U$ of $V$ and (simple) commutative algebras $A$ with trivial twist in $\Rep(V)$, see \cite[Theorem 3.2 and Remark 3.3]{HKL15}. Suppose that $V$ is strongly rational and it satisfies the positivity condition on modules. Then $U$ is of CFT type if and only if $A$ is haploid. If $A$ is haploid, then $U$ is simple as VOA if and only if $A$ is rigid by \cite[Lemma 1.20]{KO02}. 
Indeed, if $U$ is both simple and of CFT type, then it is also strongly rational and it satisfies the positivity condition on modules, see \cite[Theorem 3.6]{HKL15}. 
Moreover, $\Rep(U)$ is a vertex tensor category and it is equivalent to $\Rep^0(A)$ as modular tensor category, see also \cite[Theorem 3.65]{CKM24}.  
We also have that $\Rep(A)$ is semisimple by \cite[Theorem 3.3]{KO02}.

\begin{rem}  \label{rem:strong_rationality_positivity}
	If $V$ is a strongly rational VOA and we do not assume the positivity condition on modules, then the proof of \cite[Theorem 3.2]{HKL15} shows that a (simple) CFT type VOA extension $U$ of $V$ still corresponds to a (simple) haploid commutative algebra $A$ with trivial twist in $\Rep(V)$. 
	We also have by \cite[Lemma 1.20]{KO02} that if $A$ is haploid with $d(A)\not=0$, then $U$ is simple as VOA if and only if $A$ is rigid.
	Furthermore, we have that a simple CFT type VOA extension $U$ of $V$ is automatically strongly rational and the corresponding commutative algebra $A$ is such that $\Rep(A)$ is semisimple, see \cite[Theorem 4.10]{CMSY}.  
\end{rem}

On the other hand, if one relaxes the strong rationality condition on the VOA $V$, then
one may still have a vertex tensor category structure in the generalized sense of \cite{HLZ14}--\cite{HLZj}, see also \cite[Section 3.3]{CKM24}, on some suitable full subcategory $\catc_V$ of $\catw_V$, see Remark \ref{rem:generalized_modules} below. In this case, the vertex tensor category structure makes $\catc_V$ at least a balanced tensor category.
For example, this is the case if $V$ is of CFT type and $C_2$-cofinite, see \cite[Proposition 4.1 and Theorem 4.13]{Hua09}.
Moreover, the description of the tensor unit, the twist and the eventual rigidity as given in Remark \ref{rem:some_facts_vtc} still hold. 
In any case, \cite[Theorem 3.2 and Remark 3.3]{HKL15} applies also to this more general setting, see \cite[Theorem 3.42]{CKM24}, cf.\ the discussion at the end of \cite[Appendix]{HKL15}, so that we still have a one-to-one correspondence between (simple) VOA extensions $U$ of $V$ such that $U\in\catc_V$ and (simple) commutative algebras $A=(X,m,\iota)$ in $\catc_V$ with trivial twist and such that $X$ is a $V$-module. 
Moreover, the category $\Rep_{\catc_V}(U)$ of weak $U$-modules, which are also objects of $\catc_V$ when seen as weak $V$-modules, has a vertex tensor category structure in the sense of \cite{HLZ14}--\cite{HLZj}. 
By \cite[Theorem 3.65]{CKM24} again, $\Rep_{\catc_V}(U)$ is equivalent to $\Rep^0(A)$ as balanced tensor category.  

\begin{rem} \label{rem:generalized_modules}
	The vertex tensor category structure on $\catc_V$ in the sense of \cite{HLZ14}--\cite{HLZj} is given in terms of \textit{generalized} $V$-modules \cite[Definition 2.12]{HLZ14} and \textit{logarithmic} intertwining operators \cite[Definition 3.10]{HLZb}.
	To help the readability of some arguments, we recall the following definitions.
	A \textit{grading-restricted generalized $V$-module} is the same thing as a $V$-module, but where we allow the $M_n$'s in \eqref{eq:decomposition_ordinary_modules} to be generalized $L_0^M$-eigenspaces, see \cite[Definition 2.25 and Remark 2.27]{HLZ14}, see also \cite[Definition 3.1]{CKM24}.
	An irreducible generalized $V$-module (in the sense that it contains no non-trivial generalized $V$-module) is automatically a $V$-module if it is grading-restricted, see \cite[Proposition 1.6]{Hua09}. 
	In our applications, we will mostly work with $V$-modules, that is $\catc_V$ is a full subcategory of $\Rep(V)$, at most suppressing the grading restriction condition (so that these VOA modules are generalized in the sense of \cite[Definition 2.11]{HL95}, see also \cite[Remark 2.17]{HLZ14}). In this case, logarithmic intertwining operators reduce to ordinary ones, apart from the relaxation of the lower truncation property in Definition \ref{defin:intertwining_operators}, see \cite[Remark 3.12, Remark 3.19 and Remarks 3.23--3.25]{HLZb} for details, cf.\ also the last sentence in Remark \ref{rem:ordinary_modules_are_admissible}.
	In particular, if $V$ is strongly rational and $\catc_V$ is chosen as (a full subcategory of) $\Rep(V)$, then the two theories coincide.
\end{rem}

\begin{rem} \label{rem:finite-length_modules}
	A generalized $V$-module $M$ of a VOA $V$, see Remark \ref{rem:generalized_modules}, has \textit{length} $l\in\Zplus$ if there exists a filtration $M=M^1\supset\cdots\supset M^{l+1}=\{0\}$ of generalized $V$-submodules of $M$ such that the composition factors $\lcoset{M^j}{M^{j+1}}$ for $j\in\{1,\dots, l\}$ are irreducible. Accordingly, a \textit{finite-length} generalized $V$-module is one whose length is finite. 
	By the proof of \cite[Proposition 1.5]{Hua09}, we known that the contragredient module $M'$ of a generalized $V$-module $M$ (see \cite[Definition 2.35]{HLZ14} for the definition) of length $l$ has length $l$ too. Moreover, the composition factors of $M'$ are given by the contragredient modules of the composition factors of $M$.
\end{rem}

\subsubsection{Conditions for direct limit completions of vertex tensor categories}  \label{subsubsec:ind-categories}

For our applications, we will need to work with VOAs satisfying the conditions listed in this subsection.
Let $V$ be a VOA and suppose that there exists a full subcategory $\catc_V$ of $\Rep(V)$ satisfying the following assumptions:
\begin{itemize}
	\item[$(i)$] the adjoint module of $V$ is an object of $\catc_V$;
	
	\item[$(ii)$] $\catc_V$ is closed under taking submodules, quotients and finite direct sums;
	
	\item[$(iii)$] every $V$-module in $\catc_V$ is finitely generated.
\end{itemize}
Following \cite[Definition 4.4]{CMY22}, we denote by $\Ind(\catc_V)$ the \textit{direct limit completion} of $\catc_V$ in $\catw_V$, that is the full subcategory of objects in $\catw_V$ that are isomorphic to direct limits of direct systems of objects and morphisms in $\catc_V$.
We point out from \cite[Proposition 4.6]{CMY22} that $\Ind(\catc_V)$ is actually the full subcategory of $\mathcal{W}_V$ whose objects are weak $V$-modules that are unions of submodules which are objects of $\catc_V$. 
Note that if $\catc_V$ is semisimple, then objects in $\Ind(\catc_V)$ are (possibly infinite) direct sum of objects in $\catc_V$. 
In this case, if $M=\bigoplus_{i\in I}M^i$ is an object of $\Ind(\catc_V)$, then for any object $N\in\Ind(\catc_V)$, $\Hom(M,N)$ is naturally isomorphic to the direct product $\prod_{i\in I}\Hom(M^i,N)$.
As this will be always the case in our applications, we further assume that:
\begin{itemize}
	\item[$(iv$)] $\catc_V$ is semisimple. (Note that this condition implies $(ii)$ and $(iii)$ above).
\end{itemize}

\begin{rem}  \label{rem:ind-completions}
	Let $\catc$ be a generic category not necessarily related to any VOA. Then the symbol $\Ind(\catc)$ usually denotes the indization of $\catc$ in the sense of e.g.\ \cite[Chapter 6]{KS06} or in other words, the completion of $\catc$ for ind-objects. 
	In the case where $\catc=\catc_V$ for some VOA $V$, the relation between these two different notions of completion is not fully understood, see the beginning of \cite[Section 4]{CMY22} for a comment on this. 
	Nevertheless, for what concerns our applications, they give rise to equivalent categories under the assumption $(iv)$.
	Furthermore, in this case, $\Ind(\catc_V)$ is also equivalent to the direct sum completion $(\catc_V)_\oplus$ of $\catc_V$ as presented in \cite[Section 3]{AR18}, see also \cite[Section 1]{CMY22}.
\end{rem}

We also require that:
\begin{itemize}	
	\item[$(v)$] $\catc_V$ has a ribbon (or more generally also balanced) tensor category structure inherited from a vertex tensor category one in the sense of \cite{HLZ14}--\cite{HLZj};
	
	\item[$(vi)$] for any intertwining operator $\Y$ of type $\binom{O}{M\,\, N}$, where $M$ and $N$ are $V$-modules in $\catc_V$ and $O$ is a weak $V$-module in $\Ind(\catc_V)$, the vector space 
	\begin{equation}
		\operatorname{Im}\Y:=\operatorname{span}\left\{
		v_nw\mid v\in M\,,\,\, w\in N \,,\,\, n\in\C 
		\right\} 
	\end{equation}
	is an object in $\catc_V$.
\end{itemize}

\begin{rem}
	Note that the weak $V$-module $O$ in (vi) above is in general a $V$-module without the grading restriction condition.
\end{rem}

By \cite[Theorem 1.1]{CMY22}, the assumptions $(i)$--$(vi)$ above assure that $\Ind(\catc_V)$ has a vertex tensor category structure in the sense of \cite{HLZ14}--\cite{HLZj}, making it a balanced tensor category. Moreover, the vertex tensor category structure of $\Ind(\catc_V)$ extends the one of $\catc_V$, see \cite[Theorem 5.6 and Theorem 6.3]{CMY22} with the discussion thereafter for details. 
Therefore, still under the assumptions $(i)$--$(vi)$ above, (simple) VOA extensions $U$ of $V$ such that $U\in\Ind(\catc_V)$ are in one-to-one correspondence with (simple) commutative algebras $A=(X,m,\iota)$ in $\Ind(\catc_V)$ with trivial twist and such that $X$ is a $V$-module, see \cite[Theorem 7.5]{CMY22}. 
Moreover, see \cite[Theorem 7.7]{CMY22} and the discussion just before, the category $\Rep_{\Ind(\catc_V)}(U)$ of weak $U$-modules, which are also objects of $\Ind(\catc_V)$ when seen as weak $V$-modules, has a vertex tensor category structure in the sense of \cite{HLZ14}--\cite{HLZj}, and it is equivalent to $\Rep^0(A)$ as balanced tensor category.

\section{\texorpdfstring{Classification of simple algebras in $\Ind(\Rep(G))$}{Classification of simple algebras in Ind(Rep(G))}}
\label{sec:simple_G-algebras}

In this section, we will need some basic notions of algebraic geometry, other than the basic theory of reductive affine algebraic groups. Some standard references are e.g.\ \cite{Hum75, Eis95, TY05, Pro07}. 
Our aim is to prove Corollary \ref{cor:simple_G-algebras_countable_dimension} and its categorical translation Theorem \ref{theo:simple_commutative_algebras_in_Ind(C_G)}, which are the main achievements for this section.

In this section, for a \textit{commutative algebra}, we always mean a commutative associative unital algebra over $\C$.
The complex vector space dimension of $A$ is denoted by $\dim(A)$ and as usual $\aleph_0$ denotes the cardinality of natural numbers.
The \textit{maximal spectrum} $\spm(A)$ of a commutative algebra $A$ is the set of all maximal ideals of $A$. 
Moreover, we equip $\spm(A)$ with the usual Zariski topology, that is the one induced by the collection of closed subsets $\{Z(\{a\})\mid a\in A\}$, where for every subset $S$ of $A$, $Z(S)$ denotes the set of maximal ideals containing $S$.
The \textit{radical} $\rad(A)$ of $A$ is the intersection of all maximal ideals of $A$. 
For any ideal $I$ of $A$, we denote by $q_I:A\to\lcoset{A}{I}$ the corresponding quotient homomorphism. 
We also use $\nil(A)$ to denote the set of nilpotent elements of $A$.

For an affine variety $X$, we will denote its commutative algebra of regular functions by $\mathcal{O}(X)$. 
The algebra product of $\mathcal{O}(X)$ is the usual point-wise multiplication of functions and we will denote by $\one_X$ the unit of $\mathcal{O}(X)$, which is the regular function on $X$ constantly equal to $1$.

\begin{rem} \label{rem:hilbert_nullstellensatz}
	It is a well-known consequence of Hilbert's Nullstellensatz, see e.g.\ \cite[Theorem 1.6 and Corollaries 1.8--1.9]{Eis95}, that $A$ is \textit{finitely generated} and \textit{reduced} (that is $\nil(A)=\{0\}$) if and only if $A$ is isomorphic to the commutative algebra $\mathcal{O}(\spm(A))$ of regular functions on the affine variety $\spm(A)$.
	Another important consequence is that this correspondence between $A$ and $\spm(A)$ gives a dual equivalence between the category of reduced finitely generated commutative algebras and the one of affine varieties with corresponding morphisms, see e.g.\ \cite[Corollary 1.10]{Eis95} and the discussion there before.
\end{rem}

In the following, we will deal with commutative algebras which are a priori not necessarily reduced or finitely generated, which prevents us even to conclude that $\spm(A)$ is an affine variety, see Remark \ref{rem:hilbert_nullstellensatz}. 
Therefore, our aim will be to prove that $A$ is reduced and finitely generated under some conditions, which are natural to consider from our applications.

\begin{prop}  \label{prop:quotient_algebra_countable_dimension}
	Let $A$ be a commutative algebra such that $\dim(A)\leq \aleph_0$. Then $I\in \spm(A)$ if and only if $\lcoset{A}{I}\cong \C$.
\end{prop}

\begin{proof}
	Suppose by contradiction that $\lcoset{A}{I}\cong \C$ and that $I$ is not maximal. Then there exists a maximal ideal $J$ of $A$ containing $I$, so that $q_I(J)\in\C\backslash\{0\}$. In particular, there exists $b\in J\backslash I$ such that $q_I(b)=1=q_I(u)$, where $u$ is the unit element of $A$. It follows that $(b-u)\in I$ and thus that $u\in J$, leading to a contradiction. 
	
	Vice versa, if $I$ is maximal, then $\lcoset{A}{I}$ is isomorphic to a field $\mathbb{F}$ extending $\C$. Indeed, suppose that $q_I(a)$ is not invertible for some $a\not\in I$. Then $J_a:=q_I(a)(\lcoset{A}{I})$ is a non-trivial proper ideal of $\lcoset{A}{I}$, so that $q_I^{-1}(J_a)$ is also a non-trivial proper ideal of $A$ extending $I$, contradicting the maximality of the latter.
	Furthermore, $\dim(\mathbb{F})\leq \dim(A)\leq \aleph_0$, which is known to imply that $\mathbb{F}=\C$. 
\end{proof}

We use the following standard definitions and notations for groups:

\begin{defin}  \label{defin:groups_cosets}
	Let $G$ be a group. 
	For a subgroup $H$ of $G$, $\lcoset{G}{H}$, resp.\ $\rcoset{G}{H}$, denotes the set of \textit{left}, resp.\ \textit{right}, \textit{cosets} of $H$ in $G$, that is $\{gH\mid g\in G\}$, resp.\ $\{Hg\mid g\in G\}$.
\end{defin}

\begin{defin}
	Let $G$ be a group. 
	A set, resp.\ vector space, $X$ is called a \textit{$G$-set}, resp.\ \textit{$G$-module}, if there is a representation of $G$ on $X$, that is a group homomorphism $\pi$ from $G$ to the symmetric group $S(X)$, resp.\ the general linear group $\GL(X)$, of $X$.
	If the representation is not specified, than the \textit{$G$-action} is generically denoted by $g\cdot x$ for any element $g\in G$ and any $x\in X$.
	A function $f:X\to \tilde{X}$ between $G$-sets is said to be \textit{$G$-equivariant} if for all $g\in G$ and all $x\in X$, $f(g\cdot x)=g\cdot (f(x))$.
	For every $x\in X$, the \textit{$G$-orbit} of $x$ is denoted by
	\begin{equation}
			G\cdot x
			:=
			\{y\in X\mid \exists g\in G \,\,\, (y=g\cdot x)\} \,;
	\end{equation}
	whereas the \textit{stabilizer} of $x$ is denoted by
	\begin{equation}
		G_x
		:=
		\{g\in G\mid g\cdot x=x\} \,.
	\end{equation}
\end{defin}

\begin{defin}
	Let $G$ be a group and let $\pi$ be a representation on a $G$-module $X$. The \textit{dual representation} $\pi'$ of $\pi$ is the representation of $G$ on the dual vector space $X'$ of $X$, called the \textit{dual $G$-module} of $X$, given by the following $G$-action:
	\begin{equation}
		\forall g\in G
		\,\,\,\forall f\in X'
		\,\,\,\forall x\in X 
		\qquad
		[\pi'(g)(f)](x):=f(\pi(g^{-1})(x)) \,.
	\end{equation}
\end{defin}

Now, we move to consider reductive affine algebraic groups and their actions on affine varieties, see e.g.\ \cite[Chapters 21--22 and Chapter 27]{TY05} or \cite[Chapters 7--8]{Pro07} for details.

\begin{defin}  \label{defin:affine_G-variety}
	Let $G$ be a reductive affine algebraic group. 
	An affine variety $X$ is said to be an \textit{affine $G$-variety} if there is an action of $G$ on $X$ such that the function from $G\times X$ to $X$ given by $(g,x)\mapsto g\cdot x$ is a morphism of affine varieties.
\end{defin}

Note that in Definition \ref{defin:affine_G-variety} above, $G$ acts by affine variety automorphisms on $X$.
When we consider $G$ as an affine $G$-variety, we are using left multiplications for the $G$-action on $G$, that is for all $g,h\in G$, $g\cdot h:=gh$.

\begin{defin}  \label{defin:rational_G-modules}
	Let $G$ be a reductive affine algebraic group. 
	Let $A$ be a finite-dimensional $G$-module with the corresponding representation $\pi:G\to\GL(A)$. 
	Equip $A$ and $\GL(A)$ with the Zariski topology induced identifying $A$ with $\C^{\dim(A)}$.
	Then $\pi$ is said to be \textit{rational} if it is a morphism of affine algebraic groups. Accordingly, $A$ is said to be a \textit{rational} $G$-module.
\end{defin}

Note that in Definition \ref{defin:rational_G-modules} above, the induced function $\tilde{\pi}:G\times A\to A$ given by $(g,a)\mapsto \pi(g)(a)$ is a morphism of affine varieties making $A$ into an affine $G$-variety.

\begin{defin}
	Let $G$ be a reductive affine algebraic group. 
	Let $A$ be a possibly infinite dimensional $G$-module. Then $A$ is said to be \textit{locally finite} if for every $a\in A$, there exists a finite-dimensional $G$-submodule $A_a$ of $A$ containing $a$. Therefore, $A$ is still called a \textit{rational} $G$-module if it is locally finite and for every $a\in A$, the finite-dimensional $G$-submodule $A_a$ is rational. 
\end{defin}

The next three definitions encapsulate the primary interest of this section.

\begin{defin}  \label{defin:G-simple_algebras}
	Let $G$ be a reductive affine algebraic group and $A$ be a rational $G$-module. $A$ is said to be a $\textit{commutative $G$-algebra}$ if $A$ is a commutative algebra and the action of $G$ on $A$ is realized by algebra automorphisms.
	Furthermore, $A$ is said to be a \textit{simple} $G$-algebra if there exists no proper ideal $I\not=\{0\}$ of $A$ such that $G\cdot I\subseteq I$.
\end{defin}

\begin{defin} \label{defin:aut_group_commutative_G-algebras}
	Let $G$ be a reductive affine algebraic group. Let $A$ and $B$ be commutative $G$-algebras. Then a $G$-equivariant algebra homomorphism (isomorphism) $\alpha:A\to B$ is said to be a \textit{$G$-algebra homomorphism} (\textit{isomorphism}). A $G$-algebra isomorphism $\alpha:A\to A$ is said to be a \textit{$G$-algebra automorphism} and the group of such automorphisms is denoted by $\Aut(A)$.
\end{defin}

If a reductive affine algebraic group $G$ acts on a commutative $G$-algebra $A$, then it also acts on the set of ideals of $A$ via $I\mapsto g\cdot  I$ for all $g\in G$. In particular, $\spm(A)$ is a $G$-module and $g\cdot Z(I)=Z(g\cdot I)$ for every ideal $I$ of $A$. 
If $A$ is finitely generated and reduced, then $A$ and $\mathcal{O}(\spm(A))$ are isomorphic $G$-algebras, where $\spm(A)$ is a rational $G$-module and in particular an affine $G$-variety.
Indeed, if $X$ is an affine $G$-variety, then we have a representation $l:G\to \Aut(\mathcal{O}(X))$ given by:
\begin{equation}  
	\forall g\in G
	\,\,\,
	\forall f\in \mathcal{O}(X)
	\,\,\,
	\forall x\in X
	\qquad
	[l(g) (f)](x):= f(g^{-1} \cdot x) \,.
\end{equation}
Equipped with this $G$-action, $\mathcal{O}(X)$ is a finitely generated reduced commutative $G$-algebra. 
If $X=G$, then $l$ is called the \textit{left regular representation} of $G$ given by \textit{left translations} $l(g)$, $g\in G$. We also have the \textit{right regular representation} $r:G\to\Aut(\mathcal{O}(G))$ realized by \textit{right translations}:
\begin{equation}  
	\forall g,h\in G
	\,\,\,
	\forall f\in \mathcal{O}(G)
	\qquad
	[r(g) (f)](h):= f(hg) \,.
\end{equation}
Note that the actions of $G$ on $\mathcal{O}(G)$ given by left and right translations commute.

\begin{defin}  \label{defin:G-subalgebras}
	Let $G$ be a reductive affine algebraic group and let $A$ be a commutative $G$-algebra. If $B$ is a subalgebra of $A$ containing the unit of $A$ and such that $G\cdot B\subseteq B$, then $B$ is a commutative $G$-algebra and it is said to be a \textit{$G$-subalgebra} of $A$. In this case, $B$ is said to be a \textit{simple} $G$-subalgebra of $A$ if it is simple as commutative $G$-algebra.
\end{defin}

Let $G$ be a reductive affine algebraic group acting on a commutative $G$-algebra $A$. Two important $G$-subalgebras of $A$ are the following \textit{fixed-point subalgebras}:
\begin{equation}
	{^H}A := \{a\in A\mid H_a=H\}
\end{equation}
where $H$ is any closed subgroup of $G$ and
\begin{equation}
	A^{\tilde{H}} := \{a\in A\mid \tilde{H}_a=\tilde{H}\}
\end{equation}
where $\tilde{H}$ is any closed subgroup of $\Aut(A)$.
In the case where $A=\mathcal{O}(G)$, we write, with an abuse of notation, $\mathcal{O}(G)^H$ in place of $\mathcal{O}(G)^{r(H)}$ for every closed subgroup $H$ of $G$, being $r$ a faithful representation.

\begin{lem} \label{lem:trivial_fixed_point_subalg_O(G)}
	Let $G$ be a reductive affine algebraic group, then $^{G}\mathcal{O}(G)=\C\one_G=\mathcal{O}(G)^G$.
\end{lem}

\begin{proof}
	Denote by $e$ the unit of $G$.
	Let $f\in {^{G}}\mathcal{O}(G)$, so that for all $h,g\in G$, $f(gh)=[l(g^{-1})f](h)=f(h)$. Then choosing $h=e$, we have that for all $g\in G$,
	$$
	f(g)=f(ge)=[l(g^{-1})f](e)=f(e)
	$$ 
	that is $f=f(e)\one_G$, proving the first equality. Then the second one is proved similarly.
\end{proof}

The following is a key result.
\begin{theo}  \label{theo:A_is_subalgebra_of_O(G)}
	Let $G$ be a reductive affine algebraic group and $A$ be a simple commutative $G$-algebra. Then $\dim(A)\leq\aleph_0$ if and only if there is an injective $G$-algebra homomorphism from $A$ into $\mathcal{O}(G)$. 
\end{theo}

\begin{proof}
	The ``if'' part is trivial.
	To prove the ``only if'' part, fix a maximal ideal $I$ of $A$, so that we have a homomorphism of algebras $q_I:A\to \lcoset{A}{I}\cong \C$ by Proposition \ref{prop:quotient_algebra_countable_dimension}.
	For every $a\in A$, define the following function
	\begin{equation}
		T_I(a):G\to\C
		\,,\quad
		T_I(a)(g):=q_I(g^{-1}\cdot a) \,.
	\end{equation}
	As $A$ is locally finite, for all $a\in A$, we can choose a basis $\{e_1,\dots, e_k\}$ of $A_a$ and write $a=\sum_{i=1}^k \alpha_i e_i$ with $\alpha_i\in\C$.
	Moreover, for all $g\in G$
	\begin{equation}
		g^{-1}\cdot e_i=\sum_{j=1}^k U_{j,i}^{-1}(g)e_j
	\end{equation}
	where $U_{j,i}^{-1}\in\mathcal{O}(G)$ for all $i,j\in\{1,\dots,k\}$.
	Therefore, for all $a\in A$ and all $g\in G$, we have that
	\begin{equation}
		T_I(a)(g) 
			= 
		q_I(g^{-1}\cdot a) =\sum_{i=1}^k \alpha_i q_I(g^{-1}\cdot e_i) 
			= 
		\sum_{i,j=1}^k \alpha_i U_{j,i}^{-1}(g) q_I(e_j)
	\end{equation}
	that is
	\begin{equation}
		\forall a\in A \qquad
		T_I(a)= \sum_{i,j=1}^k \alpha_i q_I(e_j) U_{j,i}^{-1}
		\in\mathcal{O}(G) \,.
	\end{equation}
	Then $T_I$ is a function from $A$ to $\mathcal{O}(G)$. 
	Moreover, for all $a,b\in A$ and all $g\in G$
	\begin{equation}
		T_I(ab)(g)=q_I(g^{-1}\cdot ab)=q_I(g^{-1}\cdot a)q_I(g^{-1}\cdot b)=T_I(a)(g)T_I(b)(g)
	\end{equation}
	that is for all $a,b\in A$, $T_I(ab)=T_I(a)T_I(b)$.
	Similarly, for all $a,b\in A$, $T_I(a+b)=T_I(a)+T_I(b)$. 
	In other words, $T_I:A\to\mathcal{O}(G)$ is an algebra homomorphism.
	To prove the injectivity of $T_I$, suppose that $T_I(a)=0$ for some $a\in A$. Then $q_I(g^{-1}\cdot a)=0$ for all $g\in G$, that is $g^{-1}\cdot a\in I$ for all $g\in G$. This is equivalent to say that $a\in J:=\bigcap_{g\in G} g\cdot I$, which is a proper $G$-invariant ideal of $A$. By the simplicity of $A$, $J=\{0\}$ and thus $a=0$ too.
	Finally, we prove the $G$-equivariance of $T_I$. 
	Indeed, for all $a\in A$ and all $g,h\in G$, we have that
	\begin{equation}
		(g\cdot T_I(a))(h)=T_I(a)(g^{-1}h)=q_I((h^{-1}\cdot (g\cdot a)))=T_I(g\cdot a)(h)
	\end{equation}
	that is for all $a\in A$ and all $g\in G$, we have proved that $g\cdot T_I(a)=T_I(g\cdot a)$, as desired. 
\end{proof}

\begin{rem}
	Let $A$ be a simple commutative $G$-algebra for a reductive affine algebraic group $G$. If $\dim(A)\leq\aleph_0$, then $A$ is reduced as $\mathcal{O}(G)$ is thanks to Theorem \ref{theo:A_is_subalgebra_of_O(G)}.
	Actually, note that $A$ is already reduced without assuming $\dim(A)\leq\aleph_0$.		
	Indeed, $\rad(A)$ and $\nil(A)$ are both $G$-invariant ideals of $A$, which are also proper as $A$ is unital. Then they are trivial by the simplicity of $A$.
\end{rem}

\begin{rem}  \label{rem:decomposition_O(G)}
	Every reductive affine algebraic group $G$ has countably many isomorphism classes of finite-dimensional rational $G$-modules, whose corresponding set we denote by $\widehat{G}$. Moreover, for every isomorphism class $\omega\in\widehat{G}$, we denote by $M_\omega$ an arbitrary representative of this class. If $A$ is a rational $G$-module, for every $\omega\in\widehat{G}$, we call the \textit{isotypic component of type $\omega$} of $A$ the sum of the $G$-submodules of $A$ isomorphic to some representative $M_\omega$.
	Note that if we consider the actions of $G$ by left and the right translations on $\mathcal{O}(G)$ simultaneously, we get a $(G\times G)$-module.
	Moreover, by \cite[Section 27.3.9]{TY05}, we have the following isomorphism of $(G\times G)$-module:
	\begin{equation} \label{eq:peter-weyl_regular_algebra}
		\phi:\bigoplus_{\omega\in\widehat{G}} M_\omega'\otimes M_\omega\longrightarrow \mathcal{O}(G)
		\,,\quad
		(f,v)\mapsto d_{f,v}
	\end{equation}
	where for all $g\in G$, $d_{f,v}(g):=f(g\cdot v)$.
	In particular, for all $g,h,k\in G$, all $f\in M_\omega'$ and all $v\in M_\omega$, we have that
	\begin{equation}
		[l(h)r(k)\cdot d_{f,v}](g)
		=d_{h\cdot f, k\cdot v}(g) \,.
	\end{equation}
\end{rem}

The following lemma is a consequence of the above remark.

\begin{lem}  \label{lem:subalgebras_O(G)_ascending_sequence}
	Let $G$ be a reductive affine algebraic group and let $A$ be a $G$-subalgebra of $\mathcal{O}(G)$. Then there exists an ascending sequence $\{A_n\mid n\in\Zpluseq\}$ of finitely generated $G$-subalgebras of $\mathcal{O}(G)$ such that $A=\bigcup_{n\in\Zpluseq}A_n$.
\end{lem}

\begin{proof}
	Let us consider an ascending sequence $\{S_n\mid n\in\Zpluseq\}$ of finite subsets of $\widehat{G}$ such that $\widehat{G}=\bigcup_{n\in\Zpluseq}S_n$.
	For every $n\in\Zpluseq$, define $A_n$ as the subalgebra of $A$ generated by all the isotypic components whose types are in $S_n$. It is not difficult to check that this is the desired sequence, see Remark \ref{rem:decomposition_O(G)}.
\end{proof}

Now, we present some technical results on orbits of group actions, which will be crucially used to prove the main outcomes of this section. To this aim, recall Remark \ref{rem:hilbert_nullstellensatz}. 

\begin{prop} \label{prop:uniqueness_closed_orbit}
	Let $G$ be a reductive affine algebraic group. 
	\begin{itemize}
		\item[$(i)$] If $X$ is an affine $G$-variety such that $^{G}\mathcal{O}(X)\cong \C$, then there exists a unique closed $G$-orbit $O_X$ in $X$. Furthermore, 
		\begin{equation} \label{eq:orbit_of_minimal_dimension}
			\dim(O_X)=\min\{\dim(O)\mid O \,\,\mbox{ is a $G$-orbit in }\,\, X\}
		\end{equation}
		and if there exists a $G$-orbit $O$ in $X$ such that $\dim(O)=\dim(O_X)$, then $O=O_X$.
		
		\item[$(ii)$] If $X$ and $Y$ are affine $G$-varieties such that $^{G}\mathcal{O}(X)\cong\C\cong {^{G}\mathcal{O}(Y)}$ and $\varphi:Y\to X$ is a morphism of affine $G$-varieties, then $\dim(O_Y)\geq \dim(O_X)$.
		Furthermore, if $\dim(O_Y)= \dim(O_X)$, then $\varphi(O_Y)=O_X$.
		
		\item[$(iii)$] Let $\{A_n\mid n\in\Zpluseq\}$ be an ascending sequence of finitely generated $G$-subalgebras of $\mathcal{O}(G)$. 
		For all $n\in\Zpluseq$, set $X_n:=\spm(A_n)$, so that $A_n$ is isomorphic to $\mathcal{O}(X_n)$.
		For all $n,m\in\Zpluseq$ such that $m\geq n$, let $\varphi_{m,n}:X_m\to X_n$ be the morphism of affine $G$-varieties induced by the inclusion of $A_n$ into $A_m$. 
		Then there exists $k\in\Zpluseq$ such that if $m\geq n\geq k$, $\varphi_{m,n}(O_{X_m})=O_{X_n}$.
	\end{itemize}
\end{prop}

\begin{proof}
	First, we prove $(i)$.
	$^{G}\mathcal{O}(X)\cong \C$ is the algebra of regular functions over the set $X/\!\!/G$ of closed $G$-orbits in $X$, see \cite[27.5.1 and Remark 27.5.4]{TY05}. Therefore $X/\!\!/G$ must be a point, so that there is a unique closed $G$-orbit in $X$, which we denote by $O_X$.
	By \cite[Proposition 21.4.5 $(ii)$]{TY05}, $G$-orbits of minimal dimension are closed, so that \eqref{eq:orbit_of_minimal_dimension} holds by the uniqueness of $O_X$.
	For the same reason, if $O$ is a $G$-orbit such that $\dim(O)=\dim(O_X)$, then $O$ has minimal dimension and $O=\overline{O}=O_X$.
	
	Second, we prove $(ii)$. 
	Let $y\in Y$ such that $O_Y=G\cdot y$.
	As $\varphi$ is a morphism of affine $G$-varieties, $\varphi(O_Y)=G\cdot\varphi(y)$ is a $G$-orbit in $X$. 
	Note that $G_y\subseteq G_{\varphi(y)}$, so that 
	\begin{equation} \label{eq:dim_stabilizers_under_morphism}
		\dim(G_{\varphi(y)})-\dim(G_y)\geq 0 \,.
	\end{equation}
	Using twice \cite[Proposition 21.4.3 $(iii)$]{TY05} and \eqref{eq:dim_stabilizers_under_morphism}, we have that
	\begin{align}
		\dim(O_Y)
			&=
		\dim(G)-\dim(G_y) \\
			&=
		\dim(\varphi(O_Y))+\dim(G_{\varphi(y)}) -\dim(G_y) \\
			&\geq 
		\dim(\varphi(O_Y))\,.
	\end{align}
	Hence,
	\begin{equation}
		\dim(O_Y)\geq \dim(\varphi(O_Y))\geq \dim(O_X)
	\end{equation}
	where the last inequality is given by the minimality of $\dim(O_X)$ given by $(i)$.
	Furthermore, if $\dim(O_Y)=\dim(O_X)$, then $\dim(\varphi(O_Y))=\dim(O_X)$, which implies that $\varphi(O_Y)=O_X$ by part $(i)$.
	
	Third, we prove $(iii)$.
	By Lemma \ref{lem:trivial_fixed_point_subalg_O(G)}, $^{G}\mathcal{O}(G)=\C\one_G$, so that for all $n\in \Zpluseq$, $^{G}A_n\cong{^{G}\mathcal{O}(X_n)}\cong\C$.
	By part $(ii)$, the function $\Zpluseq\ni n\mapsto \dim(O_{X_n})\in\Zpluseq$ is increasing.
	Moreover, see \cite[Proposition 21.4.3$(iii)$]{TY05}, for all $n\in\Zpluseq$, $\dim(O_{X_n})\leq\dim(G)$, which is finite. Then there exists $k\in\Zpluseq$ such that for all $n,m\in\Zpluseq$ such that $m\geq n\geq k$, $\dim(O_{X_m})=\dim(O_{X_n})$. By part $(ii)$, this implies that $\varphi_{m,n}(O_{X_m})=O_n$ whenever $n,m\in\Zpluseq$ are such that $m\geq n\geq k$, as desired.
\end{proof}

The next theorem is the main achievement of this section:

\begin{theo} \label{theo:A_is_fnitely_generated}
	Let $G$ be a reductive affine algebraic group.
	Then all simple $G$-subalgebras of $\mathcal{O}(G)$ are finitely generated and the function $H\longmapsto \mathcal{O}(G)^H$ gives a one-to-one correspondence between reductive closed subgroups of $G$ and simple $G$-subalgebras of $\mathcal{O}(G)$.
	Furthermore, $\mathcal{O}(G)^H$ and $\mathcal{O}(G)^{\tilde{H}}$ are isomorphic commutative $G$-algebras if and only if $H$ and $\tilde{H}$ are in the same conjugacy class.
\end{theo}

\begin{proof}
	Let $A$ be a simple $G$-subalgebras of $\mathcal{O}(G)$.
	By Lemma \ref{lem:subalgebras_O(G)_ascending_sequence}, there exists an ascending sequence $\{A_n\mid n\in\Zpluseq\}$ of finitely generated $G$-subalgebras of $\mathcal{O}(G)$ such that $A=\bigcup_{n\in\Zpluseq}A_n$. Then we move to prove that there exists $\ell\in\Zpluseq$ such that for all $n\geq \ell$, $A=A_n=A_\ell$.
	We use the notations and the results given in Proposition \ref{prop:uniqueness_closed_orbit}. 
	In particular, for all $n\in\Zpluseq$, $X_n:=\spm(A_n)$, so that $A_n$ is isomorphic to $\mathcal{O}(X_n)$, see Remark \ref{rem:hilbert_nullstellensatz}. 
	Moreover, for all $m,n\in\Zpluseq$, $\varphi_{m,n}:X_m\to X_n$ is the dominant (i.e.\ it has dense image) morphism of affine $G$-varieties induced by the inclusion of $A_n$ into $A_m$.
	Note that for all ideals $I\in X_m$, $\varphi_{m,n}(I):=I\cap A_n$.
	Actually, we can also describe the inclusion of $\mathcal{O}(X_n)$ into $\mathcal{O}(X_m)$ by the injective $G$-algebra homomorphism
	\begin{equation}
		\varphi_{m,n}^\# \,:\,\,\,
		\mathcal{O}(X_n)\ni f\longmapsto (f\circ \varphi_{m,n})\in \mathcal{O}(X_m) \,.
	\end{equation}
	Similarly, for all $n\in \Zpluseq$, the inclusion of $A_n$ into $\mathcal{O}(G)$ induces maps $\rho_n:G\to X_n$ and $\rho_n^\#:\mathcal{O}(X_n)\to \mathcal{O}(G)$ (where $G$ is identified with its spectrum via the map sending $g\in G$ to $I_g$, that is the ideal of regular functions on $G$ annihilating in $g$).
	We note that for all $n\in\Zpluseq$, $\rho_n^\#(\mathcal{O}(X_n))=A_n$ by construction. Moreover, for all $m,n\in\Zpluseq$ such that $m\geq n$, $\rho_n^\# =\rho_m^\#\circ \varphi_{m,n}^\#$. 
	Proposition \ref{prop:uniqueness_closed_orbit} says that every $X_n$ has a unique closed $G$-orbit $O_{X_n}$ and we can fix $k\in\Zpluseq$ such that for all $m,n\in\Zpluseq$ such that $m\geq n\geq k$, $\varphi_{m,n}(O_{X_m})=O_{X_n}$.
	For every $n\in\Zpluseq$, set the following $G$-invariant ideal of $\mathcal{O}(X_n)$
	\begin{equation}
		I_n :=
		\left\{f\in\mathcal{O}(X_n)\mid \forall x (x\in O_{X_n})\,\,\,\Rightarrow\,\,\, f(x)=0\right\}   \,.
	\end{equation} 
	Then we note that for all $n\in\Zpluseq$, $I_n=\{0\}$ if and only if $X_n=O_{X_n}$. 
	As an intermediate step, we are going to prove that $X_n=O_{X_n}$ whenever $n\geq k$.
	
	Let $m,n\in\Zpluseq$ such that $m\geq n\geq k$. 
	Note that $f\in I_n$ if and only if $f(\varphi_{m,n}(y))=0$ for all $y\in O_{X_m}$, if and only if $(f\circ\varphi_{m,n})\in I_m$, that is $\varphi_{m,n}^\#(f)\in I_m$.
	Hence, for all $m,n\in\Zpluseq$ such that $m\geq n\geq k$, we have that $\varphi_{m,n}^\#(I_n)\subseteq I_m$.
	For all $n\in\Zpluseq$, set $J_n:=\rho_n^\#(I_n)$.
	According to what just proved, for all $m,n\in\Zpluseq$ such that $m\geq n\geq k$, we have that $\rho_n^\#(\mathcal{O}(X_n)I_n)=A_mJ_n\subseteq J_m$. 
	Note that for all $n\in\Zpluseq$, $I_n$ is a proper ideal of $\mathcal{O}(X_n)$ as the non-zero constant functions on $X_n$ do not belong to it.
	It follows that for all $m,n\in\Zpluseq$ such that $m\geq n\geq k$, the unit $\one_G$ of $\mathcal{O}(G)$ do not belong to $A_mJ_n$ and thus 
	\begin{equation}
		\one_G\not\in AJ_n=\bigcup_{\substack{m\in\Zpluseq \\ m\geq n\geq k}}A_mJ_n \,.
	\end{equation}
	Moreover, $AJ_n$ is a $G$-invariant ideal of $A$ and thus it must be trivial by the simplicity of $A$.
	Since $J_n\subseteq AJ_n=\{0\}$ for all $n\in\Zpluseq$, we have proved that for all $n\in\Zpluseq$ such that $n\geq k$, $X_n=O_{X_n}$. In particular, $X_n$ is a $G$-homogeneous space, see e.g.\ \cite[Section 10.6.2]{TY05}.
	
	Let $e$ be the unit of $G$ and let $k\in\Zpluseq$ as above.
	By \cite[Definition 25.3.1 and Section 25.4.6]{TY05}, for all $n\in\Zpluseq$ such that $n\geq k$, $(X_n=G\cdot x_n, \rho_n)$, where $x_n:=\rho_n(e)$, is a geometric quotient of $G$ by $H_n:=G_{x_n}$. This implies that $\rho_n^\#$ realizes a $G$-algebra isomorphism between $\mathcal{O}(X_n)$ and $\mathcal{O}(G)^{H_n}$. In other words, for all $n\in\Zpluseq$ such that $n\geq k$,
	\begin{equation}
		A_n=\rho_n^\#(\mathcal{O}(X_n))=\mathcal{O}(G)^{H_n} \,.
	\end{equation}
	It follows that there is a descending sequence of closed subgroups $\{H_n\mid n\in\Zpluseq \,,\,\,\, n\geq k\}$ of $G$, which is a Noetherian topological space as affine variety. 
	Therefore, the sequence must stabilized. This implies that there exists $\ell\in\Zpluseq$ such that for all $n\geq \ell$, $H_n=H_\ell$, implying that $A=A_n=A_\ell$.
	
	Now, we move to prove the remaining part of the theorem.
	According to what above, we have that $A=\mathcal{O}(G)^H$, where $H:=H_\ell$. By \cite[Section 25.4.7]{TY05}, we also have that $A=\mathcal{O}(G)^H$ is isomorphic to $\mathcal{O}(\lcoset{G}{H})$ as commutative $G$-algebras, where for all $s,g\in G$, $s\cdot gH=sgH$. Indeed, the $G$-homogeneous space $\lcoset{G}{H}$ is an affine variety.  
	Suppose that $\mathcal{O}(\lcoset{G}{H})$ is isomorphic to $\mathcal{O}(\lcoset{G}{\tilde{H}})$ as commutative $G$-algebras, for some closed subgroup $\tilde{H}$ of $G$ such that $\lcoset{G}{\tilde{H}}$ is an affine variety. Then there exists an isomorphism of affine $G$-varieties $\varphi:\lcoset{G}{\tilde{H}}\to \lcoset{G}{H}$, see Remark \ref{rem:hilbert_nullstellensatz}.
	In particular, there exists $u\in G$ such that $\varphi(eH)=u\tilde{H}$, where $e$ is the unit of $G$. Note also that $H=\{g\in G\mid g\cdot eH=eH\}$. Hence, we have that
	\begin{align}
		g\in H 
			&\,\,\,\Leftrightarrow\,\,\,
		g\cdot eH=eH  
			\,\,\,\Leftrightarrow\,\,\,
		\varphi(g\cdot eH)=\varphi(eH) \\
			&\,\,\,\Leftrightarrow\,\,\,
		g\cdot \varphi( eH)=\varphi(eH) 
			\,\,\,\Leftrightarrow\,\,\,
		g\cdot u\tilde{H}=u\tilde{H} \\
			&\,\,\,\Leftrightarrow\,\,\,
		u^{-1}gu\tilde{H}=e\tilde{H} 
			\,\,\,\Leftrightarrow\,\,\,
		u^{-1}gu\in \tilde{H} \\
			&\,\,\,\Leftrightarrow\,\,\,
		g\in u\tilde{H}u^{-1} 
	\end{align}
	that is $H$ and $\tilde{H}$ are in the same conjugacy class.
	Vice versa, suppose that $\tilde{H}$ is a closed subgroup of $G$ in the same conjugacy class of $H$ and such that $\lcoset{G}{\tilde{H}}$ is an affine $G$-variety. Then $H=u\tilde{H}u^{-1}$ for some $u\in G$ and one can check that the map
	\begin{equation}
		\phi:\mathcal{O}(G)^H\longrightarrow \mathcal{O}(G)^{\tilde{H}}
		\,,\qquad
		f\longmapsto r(u^{-1})f
	\end{equation}
	is a $G$-algebra isomorphism.
	In conclusion, we get the desired result from Matsushima's criterion, see e.g.\ \cite{Arz08} and references therein: the $G$-homogeneous space $\lcoset{G}{H}$ is an affine variety if and only if $H$ is reductive. 
\end{proof}

Combining Theorem \ref{theo:A_is_fnitely_generated} with Theorem \ref{theo:A_is_subalgebra_of_O(G)}, we get the following important corollary:

\begin{cor}  \label{cor:simple_G-algebras_countable_dimension}
	Let $G$ be a reductive affine algebraic group. The function $H\longmapsto \mathcal{O}(G)^H$ gives a one-to-one correspondence between conjugacy classes of reductive closed subgroups of $G$ and isomorphism classes of simple commutative $G$-algebras with countable complex vector space dimension.
\end{cor}

Now, the final act of this section is to translate Corollary \ref{cor:simple_G-algebras_countable_dimension} into a categorical language. To this aim, recall notations and results from Section \ref{subsec:tensor_categories}. 
For a reductive affine algebraic group $G$, we denote by $\catc_G$ its category of finite-dimensional rational representations. 
This is a positive symmetric ribbon tensor category with simple tensor unit, which is also fusion whenever $G$ is finite, cf.\ \cite[Example 9.9.1(2)]{EGNO15}.  
In particular, the tensor structure is given by the usual direct sum and tensor product of representations. Then the tensor unit is the trivial representation and the braiding is simply the flip. Moreover, the twist is trivial and the categorical dimension is the usual vector space dimension of representations. The rigidity is given by the duality of representations. 
We also denote by $\Ind(\catc_G)$ the direct sum completion of $\catc_G$, see Remark \ref{rem:ind-completions}.
We point out that $\catc_G$ is ribbon tensor equivalent to the category $\catc_K$ of finite-dimensional continuous representations of one of its maximal compact subgroups $K$, see Remark \ref{rem:categories_compact_lie_groups}.
From this point of view, $\catc_G$ can be considered as a compact group dual in the sense of \cite{DR89, DR89b}, cf.\ also \cite{Del90, Del02}.

\begin{rem} \label{rem:correspondence_commut_G-algebras_commut_algebras_ind}
	Let $G$ be a reductive affine algebraic group.
	Note that there is a one-to-one correspondence between (simple) commutative algebras in $\Ind(\catc_G)$ in the sense of Definition \ref{defin:algebras_in_catc} and (simple) commutative $G$-algebras in the sense of Definition \ref{defin:G-simple_algebras}. 
	Moreover, the corresponding automorphism groups, see Definition \ref{defin:aut_group_commutative_algebras_in_C} and Definition \ref{defin:aut_group_commutative_G-algebras} respectively, coincide.
	And more in general, the corresponding notions of isomorphism also coincide.
	Therefore, with an abuse of notation, we denote the commutative algebras in the two contexts by the same symbols. 
	In particular, this means that we are implicitly using the same symbol for a commutative algebra in $\Ind(\catc_G)$ and the underlying corresponding object in the same category.
	Accordingly, we refer to the simple commutative algebra $\mathcal{O}(G)$ in $\Ind(\catc_G)$ as the \textit{regular algebra} in $\Ind(\catc_G)$, cf.\ \cite[Example 8.8.9$(i)$]{EGNO15}. 
	Note also that for every closed subgroup $H$ of $G$, the commutative algebras $\mathcal{O}(\lcoset{G}{H})$, $\mathcal{O}(\rcoset{G}{H})$ and $\mathcal{O}(G)^H$ are all isomorphic in $\Ind(\catc_G)$.
\end{rem}

Remark \ref{rem:correspondence_commut_G-algebras_commut_algebras_ind} allows us to translate in categorical language the result of Corollary \ref{cor:simple_G-algebras_countable_dimension}:

\begin{theo}  \label{theo:simple_commutative_algebras_in_Ind(C_G)}
	Let $G$ be a reductive affine algebraic group and $A=(X,m,\iota)$ be a simple commutative algebra in $\Ind(\catc_G)$. 
	Suppose that the ind-object $X$ is a direct sum of at most countably many simple objects in $\catc_G$. 
	Then $A$ is isomorphic to the simple commutative algebra $\mathcal{O}(G)^H\cong \mathcal{O}(\lcoset{G}{H})$ for a unique, up to conjugacy, reductive closed subgroup $H$ of $G$.
\end{theo}

\begin{rem}  \label{rem:generalized_finite_case}
	Let $G$ be a reductive affine algebraic group.
	We note that Theorem \ref{theo:simple_commutative_algebras_in_Ind(C_G)} generalizes \cite[Theorem 2.2]{KO02}, which proves that any rigid haploid commutative algebra in $\catc_G$ is isomorphic to $\mathcal{O}(\lcoset{G}{H})$ for some closed subgroup $H$ of $G$ such that $\lcoset{G}{H}$ is a finite set.
	Indeed, if $\lcoset{G}{H}$ is not a finite set, then $\mathcal{O}(\lcoset{G}{H})$ is not in $\catc_G$, but it is in $\Ind(\catc_G)$. We also point out that under the haploidness condition, a commutative algebra in $\catc_G$ is rigid if and only if it is simple, see \cite[Lemma 1.20]{KO02}.
\end{rem}

\begin{rem}  \label{rem:decomposition_O(G)_category}
	Let $G$ be a reductive affine algebraic group and recall the notations settled in Remark \ref{rem:correspondence_commut_G-algebras_commut_algebras_ind}.
	From Remark \ref{rem:decomposition_O(G)}, we get the following isomorphism of objects in $\Ind(\catc_G)$: 
	\begin{equation} \label{eq:categorical_decomposition_regular_algebra}
		\mathcal{O}(G) 
		\cong 
		\bigoplus_{\omega\in\widehat{G}} M_\omega'\boxtimes (\C\one_G)^{\dim(M_\omega)} \,.
	\end{equation}
	It also follows that $G\cong r(G)$ is isomorphic to a subgroup of invertible arrows in
	\begin{equation} \label{eq:categorical_decomposition_right_G-action}
		\prod_{\omega\in\widehat{G}}\{\one_{M_\omega'}\}\boxtimes
		\Hom\left((\C\one_G)^{\dim(M_\omega)},(\C\one_G)^{\dim(M_\omega)}\right) \,.
	\end{equation}
	This also implies the following isomorphisms of objects in $\Ind(\catc_G)$:
	for every closed subgroup $H$ of $G$, we have that
	\begin{equation}  \label{eq:categorical_decomposition_fixed-point_subalgebras}
		\mathcal{O}(\lcoset{G}{H})\cong 
		\mathcal{O}(G)^H \cong
		\bigoplus_{\omega\in\widehat{G}} M_\omega'\boxtimes (\C\one_G)^{\dim(M_\omega^H)}
		\,.
	\end{equation}
	where $M_\omega^H:=\{v\in M_\omega\mid \forall h(h\in H) \,\,\,\Rightarrow\,\,\, h\cdot v=v \}$.
\end{rem}

\begin{rem} \label{rem:O(G)_is_haploid_with_trivial_twist}
	Let $G$ be a reductive affine algebraic group and recall the notations settled in Remark \ref{rem:correspondence_commut_G-algebras_commut_algebras_ind}.
	Note that \eqref{eq:categorical_decomposition_regular_algebra} implies that the regular algebra $\mathcal{O}(G)$ is haploid as it contains the trivial representation $\C\one_G$ of $G$ with multiplicity one. Moreover, the same is true for the commutative algebra $\mathcal{O}(G)^H$ for all closed subgroup $H\subseteq G$ due to \eqref{eq:categorical_decomposition_fixed-point_subalgebras}.
	Note also that for all closed subgroup $H\subseteq G$, the commutative algebra $\mathcal{O}(G)^H$ has trivial twist as this is trivial in $\catc_G$ and thus in $\Ind(\catc_G)$ too.
\end{rem}

\section{Doplicher-Roberts cross product and related VOA extensions}
\label{sec:DR_cross_product_voas}

The main goal of this section is to obtain a Doplicher-Roberts type cross product construction by compact Lie group duals, see \cite{DR89,DR89b, DR90}, for VOAs, and to study some of its properties. 
The core idea behind this construction is the following. 
Let $V$ be a simple VOA with a ribbon tensor category $\catc_V$ of $V$-modules satisfying suitable conditions, see e.g.\ Subsection \ref{subsubsec:ind-categories}, so that the direct limit completion $\Ind(\catc_V)$ of $\catc_V$ has a vertex tensor category structure.
Suppose that $\catc_V$ is ribbon tensor equivalent to the category $\catc_G$ of finite-dimensional rational representations of a reductive affine algebraic group $G$, as introduced at the end of Section \ref{sec:simple_G-algebras}. Then we can construct a simple VOA extension $W$ of $V$ in the category $\Ind(\catc_V)$, and a suitable action of $G$ on $W$ by VOA automorphisms such that $W^G=V$. Moreover, $W$ is isomorphic to any other simple VOA extension $\tilde{W}$ of $V$ in $\Ind(\catc_V)$ that is equipped with a suitable action of $G$ by VOA automorphisms giving $\tilde{W}^G=V$.   
This is achieved in Theorem \ref{theo:DR_cross_product_VOAs} by using the main results of Section \ref{sec:simple_G-algebras}.
Under suitable conditions, see Remark \ref{rem:decomposition_voaxgroup_compact_case_and_vtc_criteria}, the converse statement is also true: if there exists a simple VOA extension $W$ of a simple VOA $V$ such that $W^G=V$, for a reductive affine algebraic group $G$, then there exists a ribbon tensor category $\catc_V$ of $V$-modules, ribbon tensor equivalent to $\catc_G$, such that $W$ lies in $\Ind(\catc_V)$.
Accordingly, we will call $W$ the \textit{Doplicher-Roberts (DR) VOA extension} $V\rtimes \catc_V$ of $V$ by $\catc_V$, see Theorem \ref{theo:DR_cross_product_VOAs}.
As an application of this construction, we also give in Corollary \ref{cor:galois_correspondence} a Galois correspondence for reductive affine algebraic groups acting on VOAs.

However, we note that the results in \cite{DR89, DR89b, DR90}, as well as the ones in Remark \ref{rem:decomposition_voaxgroup_compact_case_and_vtc_criteria}, are given in terms of compact Lie groups acting on a VOA. 
Nevertheless, this is perfectly compatible with our framework as we are going to show in Proposition \ref{prop:actions_compact_lie_group_and_its_complexification}, see also Remark \ref{rem:categories_compact_lie_groups}. 
To this aim, we will rely on some basic notions about the representation theory of compact Lie groups, see e.g.\ \cite{BtD95, Pro07}.

\begin{defin}
	Let $V$ be a VOA. A reductive affine algebraic group $G$ \textit{acts rationally} on $V$ if there exists a faithful rational representation of $G$ in $\GL(V)$ given by VOA automorphisms. Similarly, a compact Lie group $K$ \textit{acts continuously} on $V$ if there exists a faithful continuous representation of $G$ in $\GL(V)$ given by VOA automorphisms.
\end{defin}

\begin{rem}  \label{rem:categories_compact_lie_groups}
	The complexification $G:=K_\C$ of a compact Lie group $K$, which is naturally a complex analytic Lie group, is a reductive affine algebraic group. Moreover, $K$ is a \textit{maximal compact subgroup} of $G$ with respect to the Lie group topology of $G$. Vice versa, any reductive affine algebraic group is the complexification of one of its maximal compact subgroups, which are all conjugate to each other. 
	See e.g.\ \cite[Section 8]{BtD95},  \cite[Chapter 8, Section 6.2, Theorem 3]{Pro07} and  \cite[Chapter 8, Section 7.2, Theorem 3]{Pro07}.
	We will use $\catc_K$ to denote the category of finite-dimensional continuous representations of $K$.
	This is a positive symmetric ribbon tensor category with simple tensor unit.  
	It is known that $\catc_K$ is ribbon tensor equivalent to the category $\catc_G$ of finite-dimensional rational representations of $G$, see e.g.\ \cite[Chapter 8, Sections 7.1, Corollary]{Pro07}.
	It follows that the direct sum completion $\Ind(\catc_K)$ of $\catc_K$ is balanced tensor equivalent to $\Ind(\catc_G)$. Let $T:\Ind(\catc_G)\to\Ind(\catc_K)$ be a functor realizing this equivalence.
	Since the commutative algebra $\mathcal{O}(G)$ of regular functions on $G$ is isomorphic to the commutative algebra $F(K)$ of \textit{representative functions} (see e.g.\ \cite[Chapter III, Section 1]{BtD95}) on $K$, see \cite[Chapter 8, Section 7.1, Proposition]{Pro07}, we have that $T(\mathcal{O}(G))$ and $F(K)$ are isomorphic simple commutative algebras in $\Ind(\catc_K)$, see Remark \ref{rem:equivalence_categories_algebras}.
\end{rem}

The following proposition shows that considering the rational action of a reductive affine algebraic group $G$ on a VOA is basically the same thing of considering the continuous action of a compact Lie group $K$ on it.

\begin{prop}  \label{prop:actions_compact_lie_group_and_its_complexification}
	Let $V$ be a VOA. Then the following holds:
	\begin{itemize}
		\item[$(i)$] if a compact Lie group $K$ acts continuously on $V$, then there exists a rational action of its complexification $G:=K_\C$ on $V$;
		
		\item[$(ii)$] if a reductive affine algebraic group $G$ acts rationally on $V$, then $V^G$ is equal to $V^K$ for any maximal compact subgroup $K$ of $G$.
	\end{itemize}
\end{prop}

\begin{proof}
	We start to prove $(i)$.
	For the universal property of the complexification $G$ of $K$, see e.g.\ \cite[Proposition 8.6]{BtD95}, every finite-dimensional continuous representation $\pi:K\to \GL(M)$ lifts to a holomorphic one $\rho:G\to \GL(M)$. The latter is in particular also rational and if $\pi$ is faithful, then $\rho$ is too, see e.g.\ \cite[Chapter 8, Section 7.2, Theorem 3]{Pro07}.
	By definition of VOA automorphism, the action of $K$ on $V$ preserves the conformal vector and thus the corresponding $L_0$-eigenspaces $V_n$ in $V$, which are finite-dimensional. Therefore, we can lift the faithful continuous representation $\pi$ of $K$ in $\Aut(V)\subseteq\prod_{n\in\Z}\GL(V_n)$ to a faithful rational representation $\rho$ of $G$ in $\prod_{n\in\Z}\GL(V_n)$.
	To show that $\rho$ is actually a representation in $\Aut(V)$, we use the fact that $G$ has a Cartan decomposition $G=K\exp(i\mathfrak{k})$, where $\mathfrak{k}$ is the Lie algebra of $K$, see e.g.\ \cite[Chapter 8, Section 7.2, Corollary]{Pro07}.
	Indeed, let $\C\ni z\mapsto \phi_z:=\rho(\exp(zv))$, for $v\in\mathfrak{k}$, be any one-parameter subgroup of $G$.
	Then for all homogeneous vectors $a,b\in V$ and all $n\in\Z$, we have an holomorphic function 
	\begin{equation}
		f: \C\ni z\longmapsto \left[\phi_z(a_nb)-\phi_z(a)_n\phi_z(b)\right]\in V_{d_b-n} \,.
	\end{equation}
	For all $z\in\R$, $\phi_z\in \Aut(V)$, so that $f(z)=0$. Then for all $z\in \C$, $f(z)=0$ by the identity theorem for holomorphic functions. It follows that $\rho(G)\subseteq\Aut(V)$.
	
	To prove $(ii)$, let $K$ be any maximal compact subgroup of the reductive affine algebraic group $G$. Then we can proceed similarly to the proof of $(i)$, using the Cartan decomposition of $G$ and the holomorphicity of the following functions: for all homogeneous vectors $a\in V$ and all $v\in\mathfrak{k}$, set 
	\begin{equation}
		\tilde{f}: z\in\C\longmapsto \left[\rho(\exp(zv))(a)-a\right]\in V_{d_a} \,.
	\end{equation}
\end{proof}

Consider a simple VOA $V$ with a ribbon tensor category $\catc_V$ of $V$-modules, satisfying conditions $(i)$--$(vi)$ in Subsection \ref{subsubsec:ind-categories}.
This implies that the direct limit completion $\Ind(\catc_V)$ of $\catc_V$, see also Remark \ref{rem:ind-completions}, has a vertex tensor category structure, extending the one of $\catc_V$ and making it a balanced tensor category.
Assume that $\catc_V$ is ribbon tensor equivalent to $\catc_G$ for some reductive affine algebraic group $G$. Then we have a balanced tensor equivalence between $\Ind(\catc_V)$ and $\Ind(\catc_G)$. This also gives a one-to-one correspondence between the isomorphism classes of simple commutative algebras in the two categories, see Remark \ref{rem:equivalence_categories_algebras}. 
Recalling definitions and properties in Remark \ref{rem:correspondence_commut_G-algebras_commut_algebras_ind} and in Remark \ref{rem:O(G)_is_haploid_with_trivial_twist}, we give the following definition (cf.\ \cite[Definition 9.9.19 and Exercise 9.9.21]{EGNO15}):

\begin{defin}  \label{defin:regular_algebras_IndC_V}
	Let $V$ be a simple VOA with a ribbon tensor category $\catc_V$ of $V$-modules, satisfying the conditions $(i)$--$(vi)$ in Subsection \ref{subsubsec:ind-categories}.
	Assume that $\catc_V$ is ribbon tensor equivalent to $\catc_G$ for some reductive affine algebraic group $G$ and consider the induced balanced tensor equivalence $T:\Ind(\catc_G)\to\Ind(\catc_V)$.
	Then the simple haploid commutative algebra with trivial twist $\mathcal{O}(\catc_V):=T(\mathcal{O}(G))$ is the \textit{regular algebra} in $\Ind(\catc_V)$.
\end{defin}

The following remark gives examples of how regular algebras in Definition \ref{defin:regular_algebras_IndC_V} can arise.

\begin{rem}  \label{rem:decomposition_voaxgroup_compact_case_and_vtc_criteria}
	Let $W$ be a simple VOA and let $K$ be a compact Lie group acting continuously on it. By \cite[Theorem 2.4]{DLM96}, we have the following decomposition of $W$ into $(W^K\times K)$-modules:
	\begin{equation}  \label{eq:decomposition_voaxcompact_group-modules}
		W\cong \bigoplus_{\omega\in\widehat{K}} W^{\omega'}\otimes N_\omega
	\end{equation}
	where for all $\omega\in\widehat{K}$, that is the set of equivalence classes of irreducible continuous representations of $K$, $W^{\omega'}$ is an irreducible $W^K$-module and $N_\omega$ is an irreducible continuous representation of $K$ in the isomorphism class $\omega$. 
	Moreover, for all $\omega,\eta\in\widehat{K}$, $W^{\omega'}$ and $W^{\eta'}$ are non-isomorphic whenever $\omega\not=\eta$.
	In particular, $V:=W^{\omega_0}=W^K$ is a simple VOA, where $\omega_0$ is the class of the trivial representation. 
	Now, consider the semisimple abelian full subcategory $\catc_V$ of $\Rep(V)$ generated by the $V$-module appearing in \eqref{eq:decomposition_voaxcompact_group-modules}. 
	In \cite[Theorem 4.1 and Corollary 4.8]{McR20}, the authors gives two criteria for $\catc_V$ to have a vertex tensor category structure, whose underlying ribbon one is ribbon tensor equivalent to $\catc_K$. In the latter case, $\catc_V$ will be ribbon tensor equivalent also to $\catc_G$ by Remark \ref{rem:categories_compact_lie_groups}.
	The first criteria assures the equivalence if the spaces of intertwining operators among the $W^\omega$'s are isomorphic to the corresponding spaces of intertwining operators among the $N_\omega$'s.
	Instead, the second criteria says that the equivalence holds if $\catc_V$ is already contained in a vertex tensor category.
	In particular, this second criteria can be used when $\catc_V$ satisfies condition $(i)$--$(vi)$ in Subsection \ref{subsubsec:ind-categories}. In this case, we can also consider the regular algebra $\mathcal{O}(\catc_V)$ in $\Ind(\catc_V)$, as introduced in Definition \ref{defin:regular_algebras_IndC_V}.
\end{rem}

Now, we are ready to prove the announced main result of this section, that is a Doplicher-Roberts cross product construction, see \cite{DR89,DR89b, DR90}, for VOAs.

\begin{theo}   \label{theo:DR_cross_product_VOAs}
	Let $V$ be a simple VOA and let $\catc_V$ be a ribbon tensor category of $V$-modules, satisfying the conditions $(i)$--$(vi)$ in Subsection \ref{subsubsec:ind-categories}. 
	Assume that $\catc_V$ is ribbon tensor equivalent to the category $\catc_G$ for some reductive affine algebraic group $G$. 
	Further assume that the weak $V$-module $W$ underlying the regular algebra $\mathcal{O}(\catc_V)$ in $\Ind(\catc_V)$ is a $V$-module.
	Then the following hold:
	\begin{itemize}
		\item[$(i)$] $W$ is a simple VOA extension of $V$ and there exists a rational action of $G$ on $W$ such that $V=W^G$.
		
		\item[$(ii)$] The map $H\mapsto W^H$ gives a one-to-one correspondence between reductive closed subgroups of $G$ and simple vertex subalgebras of $W$ containing $V$. Furthermore, any simple VOA extension $U$ of $V$ such that $U\in\Ind(\catc_V)$ is isomorphic to $W^H$ for a unique, up to conjugacy, reductive closed subgroup $H$ of $G$.
		
		\item[$(iii)$] Let $U$ be a simple VOA extension of $V$ such that $U\in\Ind(\catc_V)$ and that $G$ acts rationally on it so that $U^G=V$. Then $U$ is isomorphic to $W$.
	\end{itemize}
	Accordingly, $W$ is called the Doplicher-Roberts (DR) VOA extension $V\rtimes\catc_V$ of $V$ by $\catc_V$.
\end{theo}

\begin{proof}
	The plan of the proof is to demonstrate $(i)$ and $(ii)$ at once and then moving to prove $(iii)$.
	Recall that we are using notations and results from Section \ref{subsec:voas_and_modules} and from Section \ref{sec:simple_G-algebras}. 
	
	The fact that $V$ satisfies the conditions $(i)$--$(vi)$ in Subsection \ref{subsubsec:ind-categories} assures us that the direct limit completion $\Ind(\catc_V)$ of $\catc_V$ has a vertex tensor category structure, extending the ones of $\catc_V$ and making it a balanced tensor category.
	Recall also that there is a balanced tensor equivalence $T: \Ind(\catc_G)\to \Ind(\catc_V)$ as an effect of the equivalence between the respective subcategories $\catc_G$ and $\catc_V$.
	Crucially, this gives a one-to-one correspondence between the isomorphism classes of (simple) commutative algebras in the two categories $\Ind(\catc_G)$ and $\Ind(\catc_V)$, see Remark \ref{rem:equivalence_categories_algebras}.  
	By hypotheses, $\mathcal{O}(\catc_V):=T(\mathcal{O}(G))$, see Remark \ref{rem:correspondence_commut_G-algebras_commut_algebras_ind}, is a simple haploid commutative algebra with trivial twist in $\Ind(\catc_V)$ such that the underlying weak $V$-module $W$ is a $V$-module. Therefore, $W$ is a simple VOA extension of $V$.  
	By \eqref{eq:categorical_decomposition_regular_algebra}, we have the following isomorphism of $V$-modules:
	\begin{equation} \label{eq:voa-module_decomposition}
		W\cong \bigoplus_{\omega\in\widehat{G}} W^{\omega'}\boxtimes V^{\dim(M_\omega)}
	\end{equation}
	where for every $\omega\in\widehat{G}$, $W^{\omega'}$ is a simple object in $\catc_V$ chosen from the isomorphism class corresponding, under the ribbon tensor equivalence, to $M_\omega'$. 
	By Remark \ref{rem:equivalence_categories_algebras}, the right regular representation $r$ of $G$ on $\mathcal{O}(G)$ gives a faithful representation $r_T$ of $G$ in $\Aut(\mathcal{O}(\catc_V))$, given by $g\mapsto T(r(g))$. On the other hand, $\Aut(\mathcal{O}(\catc_V))$ is the subgroup of those automorphisms in $\Aut(W)$ fixing $V$, so giving an action of $G$ on $W$.
	As depicted in \eqref{eq:categorical_decomposition_right_G-action}, the right action $r$ of $G$ factors through actions on multiplicity spaces $\Hom((\C\one_G)^{\dim(M_\omega)},(\C\one_G)^{\dim(M_\omega)})$. 
	Moreover, $T$ induces linear isomorphisms between the latter and the multiplicity spaces $\Hom(V^{\dim(M_\omega)},V^{\dim(M_\omega)})$.
	This promotes the action of $G$ on $W$ by $r_T$ to a rational one, which is also equivalent to the right action on $\mathcal{O}(G)$.
	To sum up, we have the following decomposition of $W$ as $(V\times G)$-module, cf.\ \eqref{eq:decomposition_voaxcompact_group-modules}:
	\begin{equation}  \label{eq:voaxgroup-module_decomposition}
		W\cong \bigoplus_{\omega\in\widehat{G}} W^{\omega'}\otimes M_\omega \,.
	\end{equation}
		
	Now, note that using Theorem \ref{theo:A_is_fnitely_generated}, $\mathcal{O}(G)^H\mapsto T(\mathcal{O}(G)^H)$ induces a one-to-one correspondence between simple $G$-subalgebras of $\mathcal{O}(G)$ in the sense of Definition \ref{defin:G-subalgebras} and simple vertex subalgebras of $W$.
	This immediately implies that $W^G=V$, concluding the proof of $(i)$.
	Moreover, this also proves the first statement of $(ii)$.
	The correspondence can be justified as follows.
	As $\catc_G$ is a semisimple category, every (simple) $G$-subalgebra $A$ of $\mathcal{O}(G)$ is equivalent to consider an idempotent morphism $p_A\in \Hom(\mathcal{O}(G),\mathcal{O}(G))$ such that $p_Am=m(p_A\boxtimes p_A)$ and $p_A\iota=\iota$, where $m$ and $\iota$ are the multiplication and the unit morphisms of $\mathcal{O}(G)$ respectively.
	Furthermore, $A=\mathcal{O}(G)^H$ for a (reductive) closed subgroup $H$ of $G$ if and only if for all $h\in H$, $p_Ar(h)=r(h)=r(h)p_A$.
	The same is true for (simple) vertex subalgebras $U$ of $W$ in $\catc_V$ and its fixed-point ones $W^H$ for (reductive) closed subgroups $H$ of $G$.
	Therefore, the correspondence follows from Theorem \ref{theo:A_is_fnitely_generated} and the fact that $G$ acts on $W$ by $r_T$.
	
	More in general, let $U$ be any VOA extension of $V$ such that $U\in\Ind(\catc_V)$. Then its commutative algebra $B_U$ in $\Ind(\catc_V)$ is isomorphic to $T(A_U)$ for some commutative algebra $A_U$ in $\Ind(\catc_G)$, see again Remark \ref{rem:equivalence_categories_algebras}. 
	As vector space, $U$ decomposes into a direct sum of a countable number of finite-dimensional eigenspaces of its conformal Hamiltonian. Therefore, the ind-object underlying $A_U$ must be a direct sum of at most countably many simple objects in $\catc_G$. If $U$ is simple, then both $B_U$ and $A_U$ are also simple. According to Theorem \ref{theo:simple_commutative_algebras_in_Ind(C_G)}, $A_U$ must be isomorphic to $\mathcal{O}(G)^H$ for some reductive closed subgroup $H$ of $G$, unique up to conjugacy. It follows that $U$ is isomorphic to $W^H$, so concluding the proof of $(ii)$. 
	
	Finally, suppose that $U$ is a simple VOA extension of $V$ such that $U\in\Ind(\catc_V)$ and that $G$ acts rationally on it so that $U^G=V$. 
	By $(ii)$ of Proposition \ref{prop:actions_compact_lie_group_and_its_complexification} and \eqref{eq:decomposition_voaxcompact_group-modules}, $U\in\Ind(\catc_V)$ decomposes into the direct sum of irreducible $V$-modules in $\catc_V$. 
	On the other hand, $U$ is a simple VOA extension of $V$ and thus it must be isomorphic to a fixed-point vertex subalgebra $W^H$ of $W$ for some reductive closed subgroup $H$ of $G$, thanks to part $(ii)$. Comparing \eqref{eq:decomposition_voaxcompact_group-modules} with \eqref{eq:voa-module_decomposition} or \eqref{eq:voaxgroup-module_decomposition}, we get that $U$ must be isomorphic to $W$. This proves $(iii)$ and concludes the proof.
\end{proof}

\begin{rem} \label{rem:uniqueness_structure_DR_VOA_extension}
	Note that Theorem \ref{theo:DR_cross_product_VOAs} implies that any $V$-module having a decomposition into irreducible $V$-modules as the one of the DR VOA extension $V\rtimes \catc_V$, see \eqref{eq:voa-module_decomposition}, or as the one of its fixed-point vertex subalgebras, has a unique structure of simple VOA extension of $V$, cf.\ \cite[Conjecture 13.13 and Theorem 13.16]{CC}.
\end{rem}

\begin{rem}	
	Let $V$, $\catc_V$ and $\mathcal{O}(\catc_V)$ satisfy the hypotheses of Theorem \ref{theo:DR_cross_product_VOAs}.
	Further assume that $V$ is of CFT type and that all $V$-modules in $\catc_V$, different from the adjoint one, have positive conformal weight. Then it follows from the haploidness of $\mathcal{O}(\catc_V)$, see Remark \ref{rem:O(G)_is_haploid_with_trivial_twist}, that the DR VOA extension $V\rtimes\catc_V$ of $V$ is also of CFT type. 
\end{rem}

\begin{rem}  \label{rem:DR_VOA_extension_strongly_rational}
	Let $V$ and $\catc_V$ be as in Theorem \ref{theo:DR_cross_product_VOAs}. Assume also that $G$ is finite. In this case, $\mathcal{O}(\catc_V)$ and $\mathcal{O}(G)$ are simple commutative algebras in $\catc_V$ and in $\catc_G$ respectively. Therefore, $W$ is automatically a $V$-module. 
	Note that $G$ turns out to be finite if $V$ is strongly rational. Indeed, in this case, $\Rep(V)$ is a modular tensor category, so that $\catc_V$ and $\catc_G$ are forced to be symmetric ribbon fusion categories, implying that $G$ is finite.
	If we assume that $V$ is strongly rational and that $W_0=\C\Omega$, then $V\rtimes\catc_V$ is of CFT-type by \cite[Remark 4.5]{CKLW18} and its strongly rationality follows from \cite[Theorem 4.14]{McR20} or \cite[Theorem 4.10]{CMSY}.
\end{rem}

\begin{rem}
	Let $V$ and $\catc_V$ be as in Theorem \ref{theo:DR_cross_product_VOAs}. Note that the weak $V$-module $W$ underlying the regular algebra $\mathcal{O}(\catc_V)$ in $\Ind(\catc_V)$ is a $V$-module if and only if for all $t\in\R$ the set $\{\omega\in\widehat{G}\mid \lambda_{W^\omega}<t\}$ is finite, see \eqref{eq:voa-module_decomposition}.
	If this is not the case, then $W$ has a structure of conformal vertex algebra (that is a VOA where the conformal Hamiltonian is allowed to have infinite dimensional eigenspaces and arbitrary negative eigenvalues) containing $V$, see references in Section \ref{subsec:voas_and_modules}.
	Accordingly, $W$ can be generically called the Doplicher-Roberts (DR) extension $V\rtimes \catc_V$ of $V$ by $\catc_V$. 
\end{rem}

\begin{ex}
	Examples of DR VOA extensions are given by strongly rational VOAs with trivial representation theory, called \textit{holomorphic}, see \cite{DM04b}. Indeed, we can interpret all the holomorphic VOAs constructed in \cite{GK19, MS23} and in \cite{GK21}, by mean of the cyclic orbifold theory \cite{EMS20} and the non-abelian one \cite{EG22} respectively, as DR VOA extensions, where the group $G$ in Theorem \ref{theo:DR_cross_product_VOAs} is finite and solvable.
\end{ex}

\begin{ex}  \label{ex:simple_current_exts_as_DR_VOA_exts}
	Let $V$ be a simple VOA and let $G$ be an abelian reductive affine algebraic group. 
	Then $\widehat{G}$ has a structure of a countable abelian group. 
	A VOA $U$ is said to be a \textit{$\widehat{G}$-graded simple current extension} of $V$ if it has the following semisimple decomposition in some vertex tensor category of $V$-modules:
	\begin{equation}
		U\cong\bigoplus_{\omega\in\widehat{G}} V^\omega
		\,,\qquad
		V^{\omega_0}=V
		\,,\qquad
		\forall \omega,\eta\in \widehat{G}
		\quad V^\omega\boxtimes V^\eta\cong V^{\omega\eta}
	\end{equation}
	where $\omega_0$ is the identity element of $\widehat{G}$ and for all $\omega,\eta\in \widehat{G}$ such that $\omega\not=\eta$, $V^\omega$ and $V^\eta$ are non-isomorphic irreducible $V$-modules. In particular, note that for all $\omega\in \widehat{G}$, $V^\omega$ is an invertible object with trivial twist.
	Denote by $\catc_V$ the semisimple category generated by the $V^\omega$'s. Then $\catc_V$ is ribbon tensor equivalent to $\catc_G$ and it satisfies the hypotheses to invoke Theorem \ref{theo:DR_cross_product_VOAs}. Accordingly, the DR VOA extension $V\rtimes \catc_V$ is a $\widehat{G}$-graded simple current extension of $V$ and it is isomorphic to $U$. Cf.\ \cite[Section 5]{DM04} and the beginning of \cite[Section 3.1]{CKLR19}.
\end{ex}

As an application of Theorem \ref{theo:DR_cross_product_VOAs}, we have the following Galois correspondence:

\begin{cor}  \label{cor:galois_correspondence}
	Let $W$ be a simple VOA and let $G$ be a reductive affine algebraic group acting rationally on it. Set $V:=W^G$.
	Then $W$ is a completely reducible $V$-module. 
	Accordingly, denote by $\catc_V$ the semisimple abelian full subcategory of $\Rep(V)$ generated by the $V$-modules appearing in the decomposition of $W$ into $V$-modules.
	Suppose that $\catc_V$ satisfies the conditions $(i)$--$(vi)$ in Subsection \ref{subsubsec:ind-categories}. 
	Then $W$ is isomorphic to the DR VOA extension $V\rtimes \catc_V$ of $V$ by $\catc_V$. In particular, $H\mapsto W^H$ realizes a one-to-one correspondence between reductive closed subgroup of $G$ and simple vertex subalgebras of $W$ containing $V$. 
\end{cor}

\begin{proof}
	By (ii) of Proposition \ref{prop:actions_compact_lie_group_and_its_complexification}, $V=W^G=W^K$ for any maximal compact subgroup $K$ of $G$. 
	By Remark \ref{rem:decomposition_voaxgroup_compact_case_and_vtc_criteria}, $W$ is a completely reducible $V$-module and by the hypotheses on $\catc_V$, this is ribbon tensor equivalent to $\catc_G$. Therefore, we can construct the DR VOA extension $V\rtimes \catc_V$ of $V$ by $\catc_V$ by Theorem \ref{theo:DR_cross_product_VOAs}. Note that $W$ is in $\Ind(\catc_V)$ by construction, so that it must be isomorphic to $V\rtimes \catc_V$ by $(iii)$ of Theorem \ref{theo:DR_cross_product_VOAs}. Then the result follows from $(ii)$ of Theorem \ref{theo:DR_cross_product_VOAs}.
\end{proof}

\begin{rem}  \label{rem:hypothesis_image_intertwining_operators}
	Note that in the proof of Theorem \ref{theo:DR_cross_product_VOAs} and in the one of Corollary \ref{cor:galois_correspondence}, the fact that $V$ satisfies the condition $(vi)$ in Subsection \ref{subsubsec:ind-categories} is only used to deduce that the direct sum completion $\Ind(\catc_V)$ of $\catc_V$ has vertex tensor category structure extending the one of $\catc_V$. Accordingly, we can avoid this hypothesis if we already know that $\Ind(\catc_V)$ has such vertex tensor category structure.
\end{rem}

We end this section linking the Galois correspondence established by Corollary \ref{cor:galois_correspondence} with the one in \cite[Theorem 3]{DM99}. In the latter, the authors consider a continuous action of a compact Lie group $K$ on a simple VOA $W$.
Recall from Remark \ref{rem:decomposition_voaxgroup_compact_case_and_vtc_criteria} that in this case $W$ is a completely reducible $(V\times K)$-module, where $V:=W^K$.
They show that the map $H\mapsto W^H$ gives a one-to-one correspondence between closed subgroups $H$ of $K$ and simple orthogonally complemented vertex subalgebras of $W$ containing $V$.
To understand the condition of being orthogonally complemented, consider the decomposition \eqref{eq:decomposition_voaxcompact_group-modules} of $W$ into $(V\times K)$-modules.
For all $\omega\in\widehat{K}$, choose a $K$-invariant scalar product $\scalarang_\omega$ on the representation $N_\omega$, as appearing in \eqref{eq:decomposition_voaxcompact_group-modules}.
For a vertex subalgebra $U$ of $W$ containing $V$, we have the following decomposition into $(V\times K)$-modules:
\begin{equation}
	U\cong \bigoplus_{\omega\in\widehat{K}} W^{\omega'}\otimes R_\omega 
\end{equation}
for some vector subspaces $R_\omega$ of $N_\omega$.
Then the vertex subalgebra $U$ of $W$ containing $V$ is said to be \textit{orthogonally complemented} with respect to $K$, see \cite[Definition 2]{DM99}, if and only if the vector space defined by
\begin{equation} \label{eq:voa_orthogonal_complement}
	U^\perp:=\bigoplus_{\omega\in\widehat{K}} W^{\omega'}\otimes R_\omega^\perp
\end{equation}
has a $U$-module structure, where for all $\omega\in\widehat{K}$, $R_\omega^\perp$ is the orthogonal complement of $R_\omega$ in $N_\omega$ with respect to $\scalarang_\omega$.

\begin{rem}  \label{rem:star-structure_from_maxcompact_subgroup}
	Let $G$ be a reductive affine algebraic group and let $K$ be one of its maximal compact subgroup. 
	Consider commutative algebras in the sense of Section \ref{sec:simple_G-algebras}. Recall from Remark \ref{rem:categories_compact_lie_groups} and references therein that any representative function $f$ in the commutative algebra $F(K)$ has a unique holomorphic extension $f_\C$ living in $\mathcal{O}(G)$. 
	Then $K$ defines a unique $\ast$-involution on the commutative algebra $\mathcal{O}(G)$ by:
	\begin{equation}  \label{eq:star-involution_from_maxcompact_subgroup}
		\forall f\in\mathcal{O}(G) 
		\qquad
		f^*:= (\overline{f\restriction_K})_\C
	\end{equation}
	where $\overline{\cdot}$ denotes the usual complex conjugate of functions from $K$ to $\C$.
	Furthermore, there is a scalar product on $\mathcal{O}(G)$ defined by
	\begin{equation}  \label{eq:scalar_product_from_maxcompact_subgroup}
		\forall f,s\in\mathcal{O}(G)\qquad
		\langle f|s\rangle:=\int_K \overline{f(k)}s(k)\mathrm{d}k
	\end{equation} 
	where $\mathrm{d}k$ is the unique (normalized) Haar measure on $K$. Note that this scalar product is invariant with respect to the action of $K$ on $\mathcal{O}(G)$ through the left and the regular representations.
\end{rem}

For all closed subgroup $H$ of a reductive affine algebraic group $G$, note that $\mathcal{O}(G)$ is an algebra module for its subalgebra $\mathcal{O}(G)^H$ of functions invariant by right translations. Then we have the following. 

\begin{theo}  \label{theo:orthogonally_complemented}
	Let $G$ be a reductive affine algebraic group and let $H$ be a reductive closed subgroup of it. Consider any maximal compact subgroup $K$ of $G$ with the induced $\ast$-involution and the induced scalar product $\scalarang$ as defined in Remark \ref{rem:star-structure_from_maxcompact_subgroup}. Then the orthogonal complement $\mathcal{O}(G)^{H,\perp}$ of $\mathcal{O}(G)^H$ in $\mathcal{O}(G)$ with respect to $\scalarang$ is an algebra module for $\mathcal{O}(G)^H$ if and only if $\mathcal{O}(G)^H$ is a $\ast$-subalgebra of $\mathcal{O}(G)$ if and only if $K\cap H$ is a maximal compact subgroup of $H$. 
\end{theo}

\begin{proof}
	We start proving the first statement.
	By the definition of the $\ast$-involution, we have that for all $f\in\mathcal{O}(G)$ and all $k\in K$, $f^*(k)=\overline{f(k)}$. Hence, it is not difficult to see that for all $f,s,t\in\mathcal{O}(G)$, $\langle fs|t\rangle= \langle s|f^*t\rangle$.
	
	Suppose that $\mathcal{O}(G)^{H,\perp}$ is an algebra module for $\mathcal{O}(G)^H$. Then for all $f,t\in\mathcal{O}(G)^H$ and all $s\in\mathcal{O}(G)^{H,\perp}$, we have that $\langle s|f^*t\rangle=\langle fs|t\rangle= 0$. 
	This implies that $f^*t$ is in the orthogonal complement $\mathcal{O}(G)^{H,\perp,\perp}$  of $\mathcal{O}(G)^{H,\perp}$ in $\mathcal{O}(G)$ with respect to $\scalarang$, which in turn is equal to $\mathcal{O}(G)^H$. Accordingly, choosing $t=\one_G$, we get that $f^*\in \mathcal{O}(G)^H$, that is $\mathcal{O}(G)^H$ is a $\ast$-subalgebra of $\mathcal{O}(G)$.
	Vice versa, suppose that $\mathcal{O}(G)^H$ is a $\ast$-subalgebra of $\mathcal{O}(G)$. Then for all $f,s\in\mathcal{O}(G)^H$, $f^*s\in\mathcal{O}(G)^H$.
	Consequently, for all $t\in\mathcal{O}(G)^{H,\perp}$, $\langle s|ft\rangle=\langle f^*s|t\rangle= 0$, implying that $ft\in\mathcal{O}(G)^{H,\perp}$. This proves that $\mathcal{O}(G)^{H,\perp}$ is an algebra module for $\mathcal{O}(G)^H$ and the proof of $(i)$ is done.
	
	Now, we move to prove the second part of the theorem.
	The $\ast$-involution and the scalar product $\scalarang$ induces a $C^*$-norm on $\mathcal{O}(G)$ via $\norm{f}:= \sup_{k \in K} \abs{f(k)}$. 
	The completion of $\mathcal{O}(G)$ with respect to this norm yields a commutative $C^*$-algebra $A$, which is naturally isomorphic to the commutative $C^*$-algebra $C(K)$ of complex-valued continuous functions on $K$ by the Stone--Weierstrass Theorem. 
	Without lost of generality, we identify $A$ with $C(K)$.
	Note also that $K$ acts by (automatically continuous) $\ast$-automorphisms on $A$ via the usual left and right translations of functions. In particular, these actions are ergodic, that is ${^K}A=\C\one_K=A^K$, where $\one_K$ is the constant function on $K$, ${^K}A$ and $A^K$ are the subalgebras of $A$ of invariant functions by left and right translations respectively.
	
	Assume $\mathcal{O}(G)^H$ is a $\ast$-subalgebra of $\mathcal{O}(G)$. This ensures that its closure in $A$ is a unital $C^*$-subalgebra $B \subseteq A$. Moreover, the ergodic action of $K$ on $A$ by left translations leaves $B$ globally invariant, so that $K$ acts ergodically also on $B$.
	Trivially, $\mathcal{O}(G)^H$ is contained in $B\cap\mathcal{O}(G)$. 
	On the other hand, from the Peter--Weyl theorem, the action of $K$ on $B\cap\mathcal{O}(G)$ induces the following orthogonal algebraic direct sum with respect to $\scalarang$  
	\begin{equation}
		B\cap\mathcal{O}(G)\cong\bigoplus_{\omega\in\widehat{K}} M_\omega
	\end{equation}
	where for all $\omega\in\widehat{K}$, $M_\omega$ is the isotypic component of type $\omega$ in $B\cap\mathcal{O}(G)$, see the notation in Remark \ref{rem:decomposition_voaxgroup_compact_case_and_vtc_criteria}.
	$\mathcal{O}(G)^H$ also decomposes as an orthogonal algebraic direct sum with respect to $\scalarang$ into vector subspaces $N_\omega$ of $M_\omega$ for all $\omega\in\widehat{K}$. If $f\in M_\omega$ is orthogonal to $N_\omega$ for some $\omega\in\widehat{K}$, then $f$ must be orthogonal to $\mathcal{O}(G)^H$. Hence, $f=0$ as $\mathcal{O}(G)^H$ is norm dense in $B$. Since $\omega\in\widehat{K}$ was arbitrary, it follows that $\mathcal{O}(G)^H=B\cap\mathcal{O}(G)$. 
	
	As usual, for any unital $C^*$-algebra, define its spectrum as the set of its characters or equivalently as the set of its maximal ideals. 
	Using the Gelfand transform, see e.g.\ \cite[Section 4.2]{Ped89}, the ergodic action of $K$ on $B$ given by left translations induces a topologically ergodic action on the spectrum $X_B$ of $B$, meaning that all $K$-orbits are dense in $X_B$.
	On the other hand, $K$ is compact and its action on $B$ is continuous, so that there is a unique closed $K$-orbit coinciding with the whole spectrum $X_B$ of $B$. Therefore, the action of $K$ on $X_B$ is transitive, so that $X_B$ is a $K$-homogeneous space.
	In particular, we have a $K$-equivariant surjective continuous function $\rho_B:K\to X_B$ given by the restriction of characters.
	Then a standard argument, cf.\ the end of the proof of Theorem \ref{theo:A_is_fnitely_generated}, shows that $B = A^C$ for some closed subgroup $C\subseteq K$. 
	It follows that
	\begin{equation}
	\mathcal{O}(G)^H= A^C\cap \mathcal{O}(G)= \mathcal{O}(G)^C= \mathcal{O}(G)^{C_\C}
	\end{equation}
	where $C_\C$ is the complexification of $C$, which is a reductive closed subgroup of $G$, cf.\ the proof of Porposition \ref{prop:actions_compact_lie_group_and_its_complexification} for the righter equality. As the map $H \mapsto \mathcal{O}(G)^H$ is one-to-one, see Theorem \ref{theo:A_is_fnitely_generated}, we must have $C_\C= H$. 
	By Remark \ref{rem:categories_compact_lie_groups}, $C$ is a maximal compact subgroup of $C_\C$. Since we evidently have $C \subseteq K\cap H$, and $K\cap H$ is a compact Lie subgroup of $H$, the maximality of $C$ forces $C=K\cap H$. Thus, $K\cap H$ is a maximal compact subgroup of $H$.
	
	Conversely, if $K\cap H$ is a maximal compact subgroup of $H$, then $H=(K\cap H)_\C$, so that $\mathcal{O}(G)^H=\mathcal{O}(G)^{K\cap H}$. Since $K$ acts by $\ast$-automorphisms of $\mathcal{O}(G)$, $\mathcal{O}(G)^H$ is a $\ast$-subalgebra of $\mathcal{O}(G)$. This concludes the proof.
\end{proof}

Theorem \ref{theo:orthogonally_complemented} gives the desired link between the Galois correspondence in Corollary \ref{cor:galois_correspondence} and the one in \cite[Theorem 3]{DM99}:

\begin{cor}  \label{cor:orthogonally_complemented}
	In the framework of Corollary \ref{cor:galois_correspondence}, let $K$ be any maximal compact subgroup of $G$. Then $W^H$ is orthogonally complemented with respect to $K$ in $W$ if and only if $K\cap H$ is a maximal compact subgroup of $H$, in which case $W^H=W^{K\cap H}$.
\end{cor}

\begin{proof}
	In Corollary \ref{cor:galois_correspondence}, we have showed that $W$ is isomorphic to the DR VOA extension $V\rtimes\catc_V$ of $V$ by $\catc_V$. Accordingly, without lost of generality, we identify $W$ and $V\rtimes\catc_V$.
	In the proof of Theorem \ref{theo:DR_cross_product_VOAs}, it is proved that the balanced tensor equivalence $T$ from $\Ind(\catc_G)$ to $\Ind(\catc_V)$ induces a one-to-one correspondence between simple subalgebras of $\mathcal{O}(G)$ and simple vertex subalgebras of $W$. This correspondence is obtained associating to $\mathcal{O}(G)^H$ the object underlying the commutative algebra $T(\mathcal{O}(G)^H)$ in $\Ind(\catc_V)$, which is exactly $W^H$. Recall from Subsection \ref{subsubsec:ind-categories} that the balanced tensor category $\Rep_{\Ind(\catc_V)}(W^H)$ of $W^H$-modules which are also objects of $\Ind(\catc_V)$, when seen as $V$-modules, is balanced tensor equivalent to the category $\Rep^0(T(\mathcal{O}(G)^H))$ of local $T(\mathcal{O}(G)^H)$-modules. On the other hand, we know from Remark \ref{rem:equivalence_categories_algebras} that the latter is balanced tensor equivalent to the category $\Rep^0(\mathcal{O}(G)^H)$ of local $\mathcal{O}(G)^H$-modules.
	Recall from Proposition \ref{prop:actions_compact_lie_group_and_its_complexification} that $W^G=W^K$.
	Note that if $(W^H)^\perp$ as in \eqref{eq:voa_orthogonal_complement} is a $W^H$-module, then it corresponds, under the equivalence between $\Rep^0(T(\mathcal{O}(G)^H))$ and $\Rep^0(\mathcal{O}(G)^H)$, to $\mathcal{O}(G)^{H,\perp}$.
	This follows from the fact that $K$-invariant scalar products on any irreducible continuous representation of $K$ is unique, up to a positive constant factor. 
	In this case, $\mathcal{O}(G)^{H,\perp}$ turns to be an algebra module for $\mathcal{O}(G)^H$.
	Then we have the following series of equivalent statements.
	$W^H$ is orthogonally complemented in $W$ with respect to $K$ if and only if $(W^H)^\perp$ is a $W^H$-module in $\Rep_{\Ind(\catc_V)}(W^H)$, if and only if $\mathcal{O}(G)^{H,\perp}$ is an object in $\Rep^0(\mathcal{O}(G)^H)$ if and only if, thanks to Theorem \ref{theo:orthogonally_complemented}, $K\cap H$ is a maximal compact subgroup of $H$. Finally, $(ii)$ of Proposition \ref{prop:actions_compact_lie_group_and_its_complexification} concludes the proof. 
\end{proof}

\section{\texorpdfstring{Classification results in the case of $c=1$}{Classification results in the case of c=1}}
\label{sec:classification_c=1}

Theorem \ref{theo:classification_preunitary_spectrum_condition} in Section \ref{subsec:preunitary_spectrum_condition} classifies the simple CFT type preunitary VOA extensions of the simple $c=1$ Virasoro VOA $L(1,0)$ satisfying the spectrum condition, see Definition \ref{defin:extensions_of_L(1,0)}. 
The crucial ideal is that a VOA extension $U$ of this type contains a simple discrete series VOA extension $W$ of $L(1,0)$, see Proposition \ref{prop:preunitary_exts_contain_discrete_series_ones}. 
Moreover, if $W=L(1,0)$, then $U=L(1,0)$.
Accordingly, in order to classify the simple CFT type preunitary VOA extensions of $L(1,0)$, it is enough to classify the simple discrete series ones, different from $L(1,0)$, and their VOA extensions.
Indeed, this is the VOA analogue of the strategy used for the proof of the classification result \cite[Theorem 4.6]{Xu05} in the conformal net setting.  

We achieve our goal by three intermediate classification results for VOA extensions of $L(1,0)$: simple discrete series ones in Theorem \ref{theo:discrete_series_extensions}, Section \ref{subsec:discrete_VOA_extensions}; simple CFT type preunitary ones containing a vector of conformal weight $1$ in Theorem \ref{theo:characterization_preunitary_with_currents}, Section \ref{subsec:preunitary_currents}; simple CFT type preunitary ones containing a primary vector of conformal weight $4$ in Theorem \ref{theo:characterization_preunitary_exts_primary_cw4}, Section \ref{subsec:preunitary_cw4}.
These results are used in Section \ref{subsec:strongly_rational_c=1} to present some classification theorems for strongly rational VOAs with central charge $c=1$. And they are applied to vertex subalgebras of the rank-one lattice VOAs in Section \ref{subsec:vertex_subalgebras_generalized_symmetries}.
We strongly rely on the DR cross product construction for VOAs and its consequences, realized in Section \ref{sec:DR_cross_product_voas}. 

\subsection{Discrete series VOA extensions} \label{subsec:discrete_VOA_extensions}

For all $c\in\C$, we denote by $V(c,0)$ the \textit{universal Virasoro VOA} with central charge $c$, see e.g.\ \cite[Section 6.1]{LL04}, \cite{KRR13}. In general, $V(c,0)$ is of CFT type, but it may be not simple. In the latter case, it has a maximal ideal $I(c,0)$, so that $L(c,0):=\lcoset{V(c,0)}{I(c,0)}$ is a simple VOA of CFT type, called the \textit{Virasoro VOA} with central charge $c\in\C$. In particular, $L(1,0)=V(1,0)$. Moreover, all the isomorphism classes of irreducible $L(1,0)$-modules are given by the irreducible highest weight modules of the underline Virasoro algebra $\Vir_1$, which we denote by $L(1,h)$ for all highest weights $h\in\C$.
For further use we recall some of the known fusion rules among the irreducible $L(1,0)$-modules. By \cite[Proposition 3.2 and Theorem 3.3]{Mil02}, we have that for all $j,k,p\in \Zpluseq$,
\begin{equation} \label{eq:fusion_rules_discrete_series}
	\dim\binom{L(1, p)}{L(1,\frac{j^2}{4})\,\, L(1,\frac{k^2}{4})}
		= 
	\left\{
	\begin{array}{lc}
		1 & p=\frac{\ell^2}{4}\,,\,\,\, \ell\in\left\{j+k,\dots, \abs{j-k}\right\}\subseteq\Zpluseq \\
		0 & \mbox{ otherwise }
	\end{array}
	\right.  \,.
\end{equation}
By \cite[Theorem 4.7]{DJ10}, we have that for all $j,k,p\in \Zpluseq$ such that $k\not=\ell^2$ for all $\ell\in\Zplus$,
\begin{equation}  \label{eq:fusion_rules_preunitary} 
	\dim\binom{L(1, p)}{L(1,j^2)\,\, L(1,k)}
		= 
	\left\{
	\begin{array}{lc}
		1 & p=k \\
		0 & \mbox{ otherwise }
	\end{array}
	\right.  \,.
\end{equation}

\begin{defin}
	The set of irreducible highest weight $L(1,0)$-modules $L(1,h^2)$ for all $h\in\half\Zpluseq$ is called the \textit{discrete series} of $L(1,0)$.
\end{defin}

For the remaining of this paper, our focus is on the VOA extensions of $L(1,0)$. 

\begin{defin}  \label{defin:extensions_of_L(1,0)}
	A VOA extension $U$ of $L(1,0)$ is said to be \textit{preunitary} if, as $L(1,0)$-module,
	\begin{equation}  \label{eq:preunitary_extensions}
		U\cong \bigoplus_{h\in\Zpluseq} \alpha_h L(1,h)
		\,,\qquad \alpha_h\in\Zpluseq \,.
	\end{equation}  
	In particular, $U$ is said to satisfy the \textit{spectrum condition} if it is isomorphic to $L(1,0)$ or there exists $j\in\Zplus$ such that $\alpha_{j^2}\not=0$.
	Instead, $U$ is a \textit{discrete series} VOA extension of $L(1,0)$ if $\alpha_h=0$ whenever $h\not= j^2$ for every $j\in\Zpluseq$.
\end{defin}

\begin{rem}  \label{rem:terminology_vir_extensions}
	In Definition \ref{defin:extensions_of_L(1,0)}, the term ``preunitary'' for a VOA extension $U$ of $L(1,0)$ is justified by the fact that the defining decomposition \eqref{eq:preunitary_extensions} of $U$ is equivalent to require that $U$ is a \textit{unitary} $L(1,0)$-module, see \cite[Section 4.1]{DL14} and \cite[Proposition 12.15]{CC}. 
	Nevertheless, the unitarity of VOAs and their modules does not play any role in the results on the classification of these extensions we are going to present in the current and in the following sections. 
	Instead, the term ``spectrum condition'' has been introduced in \cite[Section 4]{Xu05} in the context of the operator algebraic approach to conformal field theory via conformal nets of von Neumann algebras. Indeed, the category of unitary VOAs and the one of conformal nets are conjectured to be equivalent, see \cite{CKLW18, CGH, HT}.
\end{rem}

Definition \ref{defin:extensions_of_L(1,0)} is motivated by the following lemma.

\begin{lem}  \label{lem:direct_sum_hw_L(1,0)_modules}
	Let $U$ be a self-contragredient VOA with central charge $c=1$. Suppose that $U$ is a direct sum of highest weight $L(1,0)$-modules. Then $U$ is a preunitary VOA extension of $L(1,0)$. 
\end{lem}

\begin{proof}
	Recall that the self-contragredient condition on a VOA is equivalent to the existence of a non-degenerate invariant bilinear form on it, see Remark \ref{rem:invariant_bilinear_form}.
	Then the fact that $U$ is a completely reducible $L(1,0)$-module is proved in \cite[Lemma 5.9]{DJ10}. The remaining part follows as no $L(1,0)$-module $L(1,h)$ for some $h\in\C\backslash\Z$ can appear as $L(1,0)$-submodule of any VOA extension of $L(1,0)$.
\end{proof}

For all $n\in\Zplus$, we denote by $V_{\latt_{2n}}$ the \textit{rank-one lattice VOA} with central charge $c=1$ associated to the even rank-one positive-definite integral lattice $\latt_{2n}:=\sqrt{2n}\Z$, equipped with the bilinear form $B(\sqrt{2n}p,\sqrt{2n}q)=2npq$ for all $p,q\in\Z$, see e.g.\ \cite[Sections 6.4--6.5]{LL04} for references. 
A rank-one lattice VOA $V_{\latt_{2n}}$ is an example of preunitary VOA extension of $L(1,0)$, which is actually a discrete series one whenever $n$ is a perfect square (that is $n$ is the square of a positive integer). Indeed, by \cite[Proposition 2.2]{DG98}, we have the following decompositions into $L(1,0)$-modules:
\begin{equation}  \label{eq:decomposition_L_2n_not_perfect_square}
	V_{\latt_{2n}}
	\cong
	\bigoplus_{p\geq 0}L(1,p^2) \oplus\bigoplus_{m>0} 2L(1, nm^2)   
\end{equation}
if $n\in\Zplus$ is not a perfect square; whereas
\begin{equation} \label{eq:decomposition_L_2n_perfect_square}
	V_{\latt_{2n}}\cong\bigoplus_{m\geq0}\bigoplus_{p=0}^{k-1}(2m+1)L(1, (mk+p)^2) 
\end{equation}
if $n=k^2$ for some $k\in\Zplus$.
In particular, we have that
\begin{equation}  \label{eq:decomposition_L_2}
	V_{\latt_2}\cong\bigoplus_{m\geq0}(2m+1)L(1, m^2)
\end{equation}

We can obtain other examples by considering fixed-point vertex subalgebras of rank-one lattice VOAs.  
The automorphism groups $\Aut(V_{\latt_{2n}})$ of all rank-one lattice VOAs $V_{\latt_{2n}}$ are well known, see \cite[Theorem 2.1]{DN99}, and their fixed-point subalgebras by finite groups of automorphisms are classified by \cite{DG98, DGR99}. 
Indeed, we have that
\begin{equation}  \label{eq:automorphism_groups_rank-one_lattice_VOAs}
	\Aut(V_{\latt_{2n}})
	\cong\left\{
	\begin{array}{lr}
		\C^*\rtimes\Z_2 &
		\mbox{ if } n>1 \\
		\operatorname{PSL}(2,\C) & \mbox{ if } n=1
	\end{array}
	\right.  
\end{equation}
where $\C^*:=\C\backslash\{0\}$ is the multiplicative group of non-zero complex numbers.
Accordingly, we have the following list of conjugacy classes of compact Lie subgroups, or equivalently the ones of reductive closed subgroups by $(ii)$ of Proposition \ref{prop:actions_compact_lie_group_and_its_complexification}, of $\Aut(V_{\latt_{2n}})$ with the corresponding fixed-point vertex subalgebras of $V_{\latt_{2n}}$.
For all $n\in\Zplus$:
\begin{itemize}
	\item a cyclic group of order two, corresponding to the subgroup $\{1\}\rtimes\Z_2$ of $\C^*\rtimes\Z_2$, generated by the lift to a VOA automorphism $\phi_n$ of $V_{\latt_{2n}}$ of order two of the isometry of the lattice $\latt_{2n}$ given by $\sqrt{2n}p\mapsto -\sqrt{2n}p$, $p\in\Z$, whose corresponding fixed-point vertex subalgebra is denoted by $V_{\latt_{2n}}^+$;
	
	\item the one-dimensional compact torus $\mathbb{T}\subset \C^*$, where $V_{\latt_{2n}}^\mathbb{T}= V_{\latt_{2n}}^{\C^*}$ is isomorphic to the \textit{Heisenberg VOA} $M(1)$ with central charge $c=1$, see e.g.\ \cite[Section 6.3]{LL04};
	
	\item the continuous dihedral group $\operatorname{D}_\infty:=\mathbb{T}\rtimes\Z_2\cong \operatorname{O}(2)$, where $V_{\latt_{2n}}^{\operatorname{D}_\infty}= V_{\latt_{2n}}^{\C^*\rtimes\Z_2}$ is isomorphic to the fixed-point vertex subalgebra $M(1)^+$ of $M(1)$, obtained from the restriction $\phi$ of $\phi_n$ to $M(1)$;
	
	\item cyclic subgroups $\Z_k$ of $\T$ of order $k\in\Zplus$, where $V_{\latt_{2n}}^{\Z_k}$ is isomorphic to $V_{\latt_{2nk^2}}$;
	
	\item dihedral subgroups $\operatorname{D}_k:=\Z_k\rtimes \Z_2$ of $\T\rtimes\Z_2$ of order $k\in\Zplus$, where $V_{\latt_{2n}}^{\operatorname{D}_k}$ is isomorphic to $V_{\latt_{2nk^2}}^+$;  
	in particular, $V_{\latt_2}^+$ is isomorphic to $V_{\latt_8}$ and $V_{\latt_4}^+$ is isomorphic to the VOA tensor product $L(\half,0)\otimes L(\half,0)$, see \cite[Lemma 3.1]{DGH98}.
\end{itemize}
And for $n=1$ only:
\begin{itemize}
	\item the maximal compact subgroup $\SO(3)$ of $\operatorname{PSL}(2,\C)$, where $V_{\latt_2}^{\SO(3)}= V_{\latt_2}^{\operatorname{PSL}(2,\C)}$ is isomorphic to $L(1,0)$;
	
	\item the exceptional finite subgroups $\operatorname{S}_4$, $\operatorname{A}_4$ and $\operatorname{A}_5$ of $\SO(3)$.  
\end{itemize}

\begin{rem}  \label{rem:strongly_rationality_lattice_orbifolds}
	All rank-one lattice VOAs are strongly rational, see \cite[Theorem 3.16]{DLM97}. Moreover, they satisfy the positivity condition on modules (so that their effective central charge $\tilde{c}=1$) and their categories of representations are given by positive pointed modular tensor categories, see \cite{DL93} for details.
	For all $n\in\Zplus$, the strong rationality of the VOAs $V_{\latt_{2n}}^+$ were given in \cite{DN99b, Yam04, Abe05}, along with a description of their irreducible VOA modules, showing that they satisfy the positivity condition on modules, so that $\tilde{c}=1$ also in these cases. 
	Similar results stand for $V_{\latt_2}^{\operatorname{A}_4}$ and $V_{\latt_2}^{\operatorname{S}_4}$, see \cite{DJ13b} and \cite{Lin15} respectively, cf.\ also \cite{CM} and \cite[Corollary 4.23]{McR21}. 
	On the other hand, $M(1)$ and $M(1)^+$ are neither rational nor $C_2$-cofinite, see \cite[Section 6.3]{LL04} and \cite{DN99c} respectively for complete classifications of their irreducible VOA modules.
\end{rem}

As an application of the DR cross product construction for VOAs introduced in Section \ref{sec:DR_cross_product_voas}, we present a complete classification result for discrete series VOA extensions of $L(1,0)$:

\begin{theo}  \label{theo:discrete_series_extensions}
	Let $U$ be a simple discrete series VOA extension of $L(1,0)$. Then $U$ is isomorphic to $V_{\latt_2}^H$ for some closed subgroup $H$ of $\Aut(V_{\latt_2})$ isomorphic to a closed subgroup of $\SO(3)$.
\end{theo}

\begin{proof}
	We use the notation and the results settled in Section \ref{sec:DR_cross_product_voas}.
	Consider the full semisimple subcategory $\catc_{L(1,0)}$ of $\Rep(L(1,0))$ generated by the irreducible $L(1,0)$-modules $L(1,j^2)$ for all $j\in\Zpluseq$. 
	Using fusion rules \eqref{eq:fusion_rules_discrete_series}, it is shown in \cite[Example 4.12]{McR20} that $\catc_{L(1,0)}$ has a vertex tensor category structure in the sense of \cite{HLZ14}--\cite{HLZj}, making it a ribbon tensor category. Moreover, it is ribbon tensor equivalent to $\catc_{\SO(3)}$, that is the positive symmetric semisimple ribbon tensor category with simple tensor unit of finite-dimensional representations of the compact Lie group $\SO(3)$. 
	Recall from Remark \ref{rem:categories_compact_lie_groups} that $\catc_{\SO(3)}$ is ribbon tensor equivalent to $\catc_{\operatorname{PSL}(2,\C)}$.
	Therefore, $\catc_{L(1,0)}$ satisfies the hypotheses $(i)$--$(v)$ in Subsection \ref{subsubsec:ind-categories} and we can consider its direct limit completion $\Ind(\catc_{L(1,0)})$.
	To prove that $\catc_{L(1,0)}$ satisfies condition $(vi)$ in Subsection \ref{subsubsec:ind-categories} too, we can follow the argument in \cite[Example 7.2]{CMY22}. 
	$\catc_{L(1,0)}$ is a semisimple full subcategory of the category $\catc_1(L(1,0))$ of $C_1$-cofinite grading-restricted generalized $L(1,0)$-modules, which is proved in \cite[Theorem 3.1.4]{CJORY21} to be equivalent to the category of finite-length generalized $L(1,0)$-modules whose composition factors are in the discrete series of $L(1,0)$, see Remark \ref{rem:generalized_modules} and Remark \ref{rem:finite-length_modules} for the definitions.
	By Remark \ref{rem:finite-length_modules} and the fact that the $L(1,0)$-modules in the discrete series are self-contragredient, we deduce that $\catc_1(L(1,0))$ is closed by taking contragredient modules. 
	Therefore, by \cite[Theorem 7.1]{CMY22},  $\catc_1(L(1,0))$ satisfies the conditions to invoke \cite[Theorem 1.1]{CMY22}. 
	In particular, this implies that every weak $L(1,0)$-module $\operatorname{Im}\Y$ in $\Ind(\catc_{L(1,0)})$ as in $(vi)$ in Subsection \ref{subsubsec:ind-categories} is in $\catc_1(L(1,0))$. Then $\operatorname{Im}\Y$ must be in $\catc_{L(1,0)}$ too.
	This means that $\catc_{L(1,0)}$
	satisfies the conditions $(i)$--$(vi)$ in Subsection \ref{subsubsec:ind-categories}, so that $\Ind(\catc_{L(1,0)})$ has a vertex tensor category structure extending the one of $\catc_{L(1,0)}$. Recall that objects in $\Ind(\catc_{L(1,0)})$ are possibly infinite direct sum of $L(1,0)$-modules in $\catc_{L(1,0)}$ and note that discrete series VOA extensions of $L(1,0)$ naturally live in $\Ind(\catc_{L(1,0)})$.
	By $(iii)$ of Theorem \ref{theo:DR_cross_product_VOAs} and \eqref{eq:decomposition_L_2}, $V_{\latt_2}$ is isomorphic to the DR VOA extension $L(1,0)\rtimes\catc_{L(1,0)}$ of $L(1,0)$ and $U$ must be isomorphic to $V_{\latt_2}^H$ for some subgroup $H$ of $\Aut(V_{\latt_2})$ isomorphic to a reductive closed subgroup of $\operatorname{PSL}(2,\C)$, see part $(ii)$ of the same theorem. Then the theorem follows by replacing $H$ with one of its maximal compact subgroups, see Section \ref{subsec:discrete_VOA_extensions} or $(ii)$ of Proposition \ref{prop:actions_compact_lie_group_and_its_complexification}.
\end{proof}

\begin{rem}
	Theorem \ref{theo:discrete_series_extensions} gives alternative characterizations of the fixed-point vertex subalgebras $V_{\latt_{2k^2}}^+$ for all $k\in \Zplus$, $V_{\latt_2}^{\operatorname{A}_4}$ and $V_{\latt_2}^{\operatorname{S}_4}$ given in \cite{DJ13, DJ14, Lin17} respectively. Note that in Theorem \ref{theo:discrete_series_extensions}, neither the strong rationality nor the condition that the effective central charge $\tilde{c}$ is equal to $1$ are assumed. 
\end{rem}

\begin{rem}  
	Note that Theorem \ref{theo:discrete_series_extensions} see also Remark \ref{rem:uniqueness_structure_DR_VOA_extension}, gives an alternative proof to \cite[Theorem B.1]{MY23} on the uniqueness of the VOA structure on the $L(1,0)$-module \eqref{eq:decomposition_L_2} underline $V_{\latt_2}$.
\end{rem}

\subsection{\texorpdfstring{Preunitary VOA extensions containing $M(1)$}{Preunitary VOA extensions containing M(1)}}
\label{subsec:preunitary_currents}

Let $V(0,0)$ be the Virasoro VOA with central charge $0$, see Section \ref{subsec:discrete_VOA_extensions} and reference therein.
When no confusion can arise, we denote by the same symbols the vacuum and the conformal vectors of $V(0,0)$ and $L(1,0)$.
$V(0,0)$ is not simple and it is generated by the vacuum vector $\Omega$ and the conformal one $L_{-2}\Omega$, which is a primary vector.
Indeed, the maximal ideal $I(0,0)$ is generated by $L_{-2}\Omega$, so that $L(0,0):=\lcoset{V(0,0)}{I(0,0)}$ is isomorphic to $\C\Omega$.
The VOA tensor product $M:=V(0,0)\otimes L(1,0)$ is a VOA of CFT type with central charge $1$, where the conformal vector is given by $\nu\otimes \Omega + \Omega\otimes \nu$, see e.g.\ \cite[Section 3.12]{LL04}. 
Note that $M$ has also a natural structure of $L(1,0)$-module, where $Y^M(a,z):=1_{V(0,0)}\otimes Y(a,z)$ for all $a\in L(1,0)$. Then we have the following result:

\begin{lem}  \label{lem:ideals_of_V(0,0)xL(1,0)}
	Every ideal of the VOA tensor product $M:=V(0,0)\otimes L(1,0)$ is of the form $J\otimes L(1,0)$ for some ideal $J$ of $V(0,0)$. 
	Furthermore, every quotient VOA $M^J:=\lcoset{V(0,0)}{J}\otimes L(1,0)$ of $M$ is a completely reducible $L(1,0)$-submodule of $M$ if and only if $J=I(0,0)$. In the latter case, $M^J= \C\Omega\otimes L(1,0)\cong L(1,0)$ is irreducible.
\end{lem}

\begin{proof}
	Note that 
	\begin{equation}
		\forall a\in L(1,0) \,\,\,
		\forall b\in M \,\,\,
		\forall n\in\Z
		\qquad
		(\Omega\otimes a)_{(n)}b=a^M_{(n)}b \,.
	\end{equation}
	This implies that every ideal $I$ of $M$ is also a $L(1,0)$-submodule of $M$, when considered as $L(1,0)$-module.
	On the other hand, the $L(1,0)$-module $M$ is a multiple of the adjoint module of $L(1,0)$. Therefore, $I$ must be equal to $J\otimes L(1,0)$ as a vector space, so that $J$ must be an ideal of $V(0,0)$.
	
	For the second statement, the only non-trivial part is the ``only if'' case. Assume by contradiction that $J\subsetneq I(0,0)$ and that $M^J$ is a completely reducible $L(1,0)$-submodule of $M$.
	Set $N^J:=\lcoset{I(0,0)}{J}\otimes L(1,0)$, which is a $L(1,0)$-submodule of $M^J$ such that $\lcoset{M^J}{N^J}$ is isomorphic to $\lcoset{V(0,0)}{I(0,0)}\otimes L(1,0)\cong L(1,0)$ as $L(1,0)$-module.
	As $M^J$ is assumed to be completely reducible, we have that it can be decomposed as $N^J\oplus S^J$, where $S^J$ is a $L(1,0)$-submodule of $M^J$ such that $S^J\cong\lcoset{M^J}{N^J}\cong L(1,0)$.
	On the other hand,
	\begin{equation}
		M^J_0=N^J_0\oplus S^J_0=S^J_0
		\quad\mbox{ as }\quad
		N^J_0=(\lcoset{I(0,0)}{J})_0\otimes L(1,0)_0=\{0\}\otimes L(1,0)_0 \,.
	\end{equation}
	Hence, $\Omega\otimes\Omega$ is in $S^J$, which must equal to $M^J$. 
	This implies that $N^J=\{0\}$, which is a contradiction as $\lcoset{I(0,0)}{J}$ is non-zero if $J\subsetneq I(0,0)$.
\end{proof}

We present a technical lemma on general preunitary VOA extensions of $L(1,0)$, which will be useful also in the following sections.

\begin{lem} \label{lem:preunitary_extensions_have_bilinear_forms}
	Let $U$ be a preunitary VOA extension of $L(1,0)$. Then $U$ has a non-zero invariant bilinear form $\bilinear$. Furthermore: 
	\begin{itemize}
		\item[$(i)$] if $U$ is of CFT type, then $\bilinear$ is unique up to normalization;
		
		\item[$(ii)$]  if $U$ is simple, then $\bilinear$ is non-degenerate and unique up to normalization;
		
		\item[$(iii)$] if $\bilinear$ is non-degenerate, then for all
		$h,k\in\Zpluseq$ such that $h\not=k$, $(U^h,U^k)=0$ whenever $U^h$ and $U^k$ are $L(1,0)$-submodules of $U$ isomorphic to $L(1,h)$ and $L(1,k)$ respectively.
	\end{itemize}
\end{lem}

\begin{proof}
	If $U$ decomposes as in \eqref{eq:preunitary_extensions}, then $U_1= \alpha_1L(1,1)_1$. On the other hand, vectors in $L(1,1)_1$ are primary and thus the non-zero invariant bilinear form $\bilinear$ exists, see Remark \ref{rem:invariant_bilinear_form}, and part $(i)$ also follows. 
	See again Remark \ref{rem:invariant_bilinear_form} for part $(ii)$.
	The proof of $(iii)$ goes as follows. 
	If $U\cong \alpha_h L(1,h)$ for some $h\in \Zpluseq$, then the statement is trivially proved. 
	Fix $h,k\in\Zpluseq$ such that $h\not=k$ and that $\alpha_h\not=0\not=\alpha_k$. 
	Choose two $L(1,0)$-submodules $U^h$ and $U^k$ of $U$ isomorphic to $L(1,h)$ and $L(1,k)$ respectively.
	Consider two non-zero vectors $a\in U^h$ and $b\in U^k$.
	If $a$ and $b$ are primary, then they have different conformal weights, so that $(a,b)=0$, see Remark \ref{rem:invariant_bilinear_form}. If $a$ is primary and $b=L_{-n_1}\cdots L_{-n_s}c$ for some primary vector $c\in U^k$, $s\in\Zplus$ and $n_1,\dots, n_s\in\Zplus$, then
	\begin{equation}
		(a,b)=(a,L_{-n_1}\cdots L_{-n_s}c)=(L_{n_s}\cdots L_{n_1}a,c)=0 
	\end{equation}
	thanks to the invariant property \eqref{eq:invariant_bilinear_form}.
	If also $a$ is not primary, then we can write $a=L_{-m_1}\cdots L_{-m_t}d$ for some primary vector $d\in U^h$, $t\in\Zplus$ and $m_1,\dots, m_t\in\Zplus$, to get that 
	\begin{equation}
		(a,b)=(d,L_{m_t}\cdots L_{m_1}L_{-n_1}\cdots L_{-n_s}c)=0 \,.
	\end{equation}
	Indeed, using the Virasoro algebra commutation relations \eqref{eq:virasoro_alg_crs}, $L_{m_t}\cdots L_{m_1}L_{-n_1}\cdots L_{-n_s}c$ is either $0$ or a linear combination of non-primary vectors in $U^k$, which we have already shown to give the desired conclusion. This concludes the proof.
\end{proof}

Therefore, we have the following:

\begin{prop}  \label{prop:preunitary_exts_currents_contains_heisenberg}
	Let $U$ be a simple CFT type preunitary VOA extension of $L(1,0)$ with $U_1\not=\{0\}$. 
	Let $\bilinear$ be the unique normalized non-degenerate invariant bilinear form on $U$ given by Lemma \ref{lem:preunitary_extensions_have_bilinear_forms}.
	Then there exists a vector $\cur\in U_1$ with $(\cur,\cur)=-1$ and such that the vertex subalgebra $U^\cur$ of $U$ generated by $\cur$ is isomorphic to $M(1)$. Furthermore, the conformal vector $\nu=L_{-2}\Omega$ of $U$ is equal to $\half \cur_{-1}\cur$.
\end{prop}

\begin{proof}
	Note that $\bilinear$ is non-degenerate on $U_1$, see Remark \ref{rem:invariant_bilinear_form}. Moreover, since $\bilinear$ is $\C$-bilinear, we can find a basis $\{e^j\}_{j=1}^k$, $k\in \Zplus$ of $U_1$, which is also orthonormal with respect to $\bilinear$. Set $\cur:=ie^1$. Note that $L_1\cur\in L_1V_1=\{0\}$ as $U$ is of CFT type and $\bilinear$ is non-zero, see Remark \ref{rem:invariant_bilinear_form}.
	From the invariant property \eqref{eq:invariant_bilinear_form}, we get that
	\begin{equation}
		(\cur_1\cur,\Omega)=-(\cur,\cur_{-1}\Omega)=-(\cur,\cur)(\Omega,\Omega)=(-(\cur,\cur)\Omega,\Omega)
	\end{equation}
	which implies that $\cur_1\cur=-(\cur,\cur)\Omega=\Omega$. By the skew-symmetry \eqref{eq:skew-symmetry} and the CFT typeness of $U$, we have that 
	\begin{equation}
		\cur_0\cur=-\cur_0\cur+L_{-1}\cur_1\cur=-\cur_0\cur+L_{-1}\Omega=-\cur_0\cur
	\end{equation}
	which implies that $\cur_0\cur=0$. 
	Applying the Borcherds commutator formula \eqref{eq:borcherds_commutator_formula} and using that $U$ is of CFT type, we find that
	\begin{equation}  \label{eq:heisenberg_crs}
		\forall n,m\in\Z
		\qquad
		[\cur_n,\cur_m]= (\cur_0\cur)_{m+n}+n(\cur_1\cur)_{n+m}=n\delta_{n+m,0} 1_U \,.
	\end{equation}
	It follows that the vertex subalgebra $U^\cur$ of $U$ generated by $\cur$ is isomorphic to $M(1)$.
	
	We also have that $\nu^\cur:=\half \cur_{-1}\cur$ is a conformal vector for $M(1)$ with central charge $1$, see \cite[Proposition 4.10(a)]{Kac01}. 
	Then $\hat{\nu}:=\nu-\nu^\cur$ is a Virasoro vector with central charge $0$. 
	Using the Borcherds commutator formula \eqref{eq:borcherds_commutator_formula}, the CFT typeness of $U$ and the $L_{-1}$-derivative property for $U$, one can show that $L_1\nu^\cur=0$.
	Then $Y(\hat{\nu},z)$ commutes with all vertex operators $Y(a,z)$ with $a\in U^\cur$, see \cite[Theorem 5.1]{FZ92} and \cite[Corollary 4.6(b)]{Kac01}.
	As a consequence, $\hat{\nu}$ and $\nu^\cur$ generate a vertex subalgebra of $U$, which must be isomorphic to a quotient VOA $M^J$ of $M:=V(0,0)\otimes L(1,0)$ as in Lemma \ref{lem:ideals_of_V(0,0)xL(1,0)}.
	As $U$ is a preunitary VOA extension of $L(1,0)$, it is completely reducible as $L(1,0)$. Then $M^J$ is also completely reducible as $L(1,0)$-module. By Lemma \ref{lem:ideals_of_V(0,0)xL(1,0)}, $J=I(0,0)$ so that $\hat{\nu}=0$. Hence, $\nu=\nu^\cur=\half \cur_{-1}\cur$ as desired.
\end{proof}

\begin{rem}
	Note that, thanks to Lemma \ref{lem:direct_sum_hw_L(1,0)_modules}, we could have weaken the hypothesis that $U$ was a preunitary VOA extension in Proposition \ref{prop:preunitary_exts_currents_contains_heisenberg}, only requiring that $U$ was a direct sum of highest weight $L(1,0)$-modules.
\end{rem}

\begin{lem}  \label{lem:properties_spectrum_current_j}
	Let $U$ be a simple CFT type preunitary VOA extension of $L(1,0)$ with $U_1\not=\{0\}$. 
	Let $\bilinear$ be the unique normalized non-degenerate invariant bilinear form on $U$ given by Lemma \ref{lem:preunitary_extensions_have_bilinear_forms}.
	Let $\cur\in U_1$ and $U^\cur\cong M(1)$ be as in Proposition \ref{prop:preunitary_exts_currents_contains_heisenberg} and denote by $\spec(\cur_0)$ the set of eigenvalues of $\cur_0$ on $U$. Then we have the following:
	\begin{itemize}
		\item[$(i)$] $\spec(\cur_0)\not=\emptyset$ and 
		\begin{equation}  \label{eq:decomposition_preunitary_heisenberg_modules}
			U=\bigoplus_{\lambda \in \spec(\cur_0)} U^\lambda
			\,,\quad
			U^\lambda:=\{a\in U\mid (\exists k\in\Zplus) \,\,\Rightarrow\,\, (\cur_0-\lambda 1_U)^ka=0  \} 
		\end{equation}
		where every $U^\lambda$ is a $U^\cur$-submodule of $U$;
		
		\item[$(ii)$] for all $\lambda,\mu\in\spec(\cur_0)$ such that  $\lambda\not=-\mu$, it holds that $(U^\lambda, U^\mu)=0$;
		
		\item[$(iii)$] either $\spec(\cur_0)=\{0\}$ or it is a rank-one lattice $\latt_{2n}:=\sqrt{2n}\Z$ for some $n\in\Zplus$.
	\end{itemize}
\end{lem}

\begin{proof}
	\textit{Proof of $(i)$}. 
	Note that for all $n\in \Zpluseq$, the $L_0$-eigenspace $U_n$ is preserved by $\cur_0$. 
	Then for all $n\in \Zpluseq$, we can consider the Jordan normal form of the operator $\cur_0\restriction_{U_n}$, which gives a decomposition
	\begin{equation}
		U_n=\bigoplus_{\lambda \in \spec(\cur_0\restriction_{U_n})} U_n^\lambda
		\,,\quad
		U_n^\lambda:=\{a\in U_n\mid (\exists k\in\Zplus) \,\,\Rightarrow\,\, (\cur_0\restriction_{U_n}-\lambda 1_{U_n})^ka=0  \} 
	\end{equation} 
	where $\spec(\cur_0\restriction_{U_n})$ is the set of eigenvalues of $\cur_0\restriction_{U_n}$ on $U_n$. 	
	Therefore, the decomposition \eqref{eq:decomposition_preunitary_heisenberg_modules} follows by setting
	\begin{equation}
		\spec(\cur_0):=\{\lambda\in\C\mid U^\lambda\not=\{0\}\}\not=\emptyset
	\end{equation}
	where 
	\begin{equation}
		\forall\lambda\in\C
		\qquad 
		U^\lambda:=\bigoplus_{n\in\Zpluseq} U_n^\lambda \,.
	\end{equation}
	By the canonical commutation relations on $M(1)$, see e.g.\ \eqref{eq:heisenberg_crs}, we know that $\cur_0$ commutes with $a_n$ for all $a\in U^\cur$ and all $n\in \Z$. Then it is easy to see that if $b\in U^\lambda$ for some $\lambda\in\spec(\cur_0)$, then for all $a\in U^\cur$ and all $n\in \Z$, $a_nb\in U^\lambda$. This means that all the $U^\lambda$'s are $U^\cur$-submodules of $U$.
	
	\textit{Proof of $(ii)$}.
	Let $\lambda,\mu\in\spec(\cur_0)$ such that $\lambda\not=-\mu$.
	By Remark \ref{rem:invariant_bilinear_form}, $(U^\lambda_n, U^\mu_m)=0$ if $n,m\in\Zpluseq$ are such that $n\not= m$. Accordingly, fix $n\in\Zpluseq$ and
	fix some $k\in\Zplus$ such that 
	\begin{equation}
		(\cur_0-\lambda 1_U)^k\restriction_{U^\lambda_n}
		=0=
		(\cur_0-\mu 1_U)^k\restriction_{U^\mu_n} \,.
	\end{equation}
	Then
	\begin{equation}
		\cur_0+\lambda 1_U=(\cur_0-\mu 1_U)+(\lambda+\mu) 1_U
		=(\lambda +\mu)\left( 1_U -\frac{\mu 1_U-\cur_0}{\lambda+\mu}\right) 
	\end{equation}
	that is
	\begin{equation}
		\frac{\cur_0+\lambda 1_U}{\lambda+\mu}=  1_U -\frac{\mu 1_U-\cur_0}{\lambda+\mu}
	\end{equation}
	Therefore, if $b\in U^\mu_n$, we have that
	\begin{align}
		\left(\frac{\cur_0+\lambda 1_U}{\lambda+\mu}\right)^k
		\left(\sum_{m=0}^{k-1}\left(\frac{\mu 1_U-\cur_0}{\lambda+\mu}\right)^m\right)^kb
			&= 
		\left[\left(1_U -\frac{\mu 1_U-\cur_0}{\lambda+\mu}\right)\sum_{m=0}^{k-1}\left(\frac{\mu 1_U-\cur_0}{\lambda+\mu}\right)^m\right]^kb \\
			&=
		\left[\sum_{m=0}^{k-1}\left(\frac{\mu 1_U-\cur_0}{\lambda+\mu}\right)^m-\sum_{m=1}^k\left(\frac{\mu 1_U-\cur_0}{\lambda+\mu}\right)^m\right]^kb \\
			&= 
		\left[ 1_U-\left(\frac{\mu 1_U-\cur_0}{\lambda+\mu}\right)^k\right]^kb  =  b \,.
	\end{align}
	In other words, $b=(\cur_0+\lambda 1_U)^kd$, where 
	\begin{equation}
		d:=\left(\sum_{m=0}^{k-1}\left(\frac{\mu 1_U-\cur_0}{\lambda+\mu}\right)^m\right)^k\frac{b}{(\lambda+\mu)^k} \in U^\mu_n \,.
	\end{equation}
	Then using the invariant property \eqref{eq:invariant_bilinear_form} of the bilinear form $\bilinear$ on $U$, we get that for all $c\in U^\lambda_n$
	\begin{equation}
		(b,c)=((\cur_0+\lambda 1_U)^kd,c)=(-1)^k(d,(\cur_0-\lambda 1_U)^kc)=0 
	\end{equation}
	as desired.
	
	\textit{Proof of $(iii)$}. By $(ii)$ and the non-degeneracy of $\bilinear$, we have a vector space isomorphism $T:U^{-\lambda}\to (U^\lambda)'$ defined by $[T(b)](c):=(b,c)$ for all $b\in U^{-\lambda}$ and all $c\in U^\lambda$.
	The fact that $T$ is a $U^\cur$-module isomorphism straightforwardly follows from the invariant property of the bilinear form and the definition  \eqref{eq:contragredient_module} of the contragredient module. As a consequence, for all $\lambda\in\spec(\cur_0)$, $U^{-\lambda}\not=\{0\}$, so that $-\lambda\in\spec(\cur_0)$.
	
	Now, we prove that for all $\lambda,\mu\in\spec(\cur_0)$, there exist non-zero vectors $a\in U^\lambda$, $b\in U^\mu$ and $n\in\Z$ such that $a_nb\not=0$.
	Suppose by contradiction that this is not the case. Then there exist non-zero vectors $a,b\in U$ such that $Y(a,z)b=0$. In particular, $Y(Y(a,z)b,w)=0$. 
	By \cite[Eq.\ (3.3.9)]{FHL93}, $Y(a,z)Y(b,w)=0$.
	This implies that for all $c\in U$, $Y(c,x)Y(a,z)Y(b,w)=0$. And by \cite[Eq.\ (3.3.9)]{FHL93} again, $Y(Y(c,x)a,z)Y(b,w)=0$. This shows that
	\begin{equation}
		I_b:=\left\{a\in U\mid Y(a,z)Y(b,w)=0\right\}
	\end{equation}
	is a non-zero ideal of $U$. If $I_b=U$, then $\Omega\in I_b$, implying that $b=0$, which is not the case. On the other hand, $U$ is simple, leading to a contradiction. 
	
	Accordingly, for all $\lambda,\mu\in\spec(\cur_0)$, choose non-zero vectors $a\in U^\lambda$, $b\in U^\mu$ and $n\in\Z$ such that $a_nb\not=0$.
	In particular, $a$ and $b$ can be chosen as $\cur_0$-eigenvectors.
	From the Borcherds commutator formula \eqref{eq:borcherds_commutator_formula}, we can see that $\cur_0$ is a derivation of $U$, that is
	\begin{equation}  \label{eq:cur_0_is_a_derivation}
		\forall c,d\in U
		\,\,\,\forall n\in\Z
		\qquad
		\cur_0(c_nd)=(\cur_0c)_nd+c_n\cur_0d \,.
	\end{equation}
	Then $\cur_0(a_nb)=(\lambda+\mu)a_nb$, so that $\lambda+\mu\in\spec(\cur_0)$. This proves that $\spec(\cur_0)$ in an additive subgroup of $\C$.
	Now, let $\lambda\in\spec(\cur_0)$ and let $n\in\Zpluseq$ be such that $U^\lambda_n\not=\{0\}$ and $U^\lambda_m=\{0\}$ for all $m<n$. Then for all $b\in U^\lambda_n$ and all $m\in\Zplus$, $\cur_mb\in U^\lambda_{n-m}=\{0\}$. Let $b\in U^\lambda_n\backslash\{0\}$ a $\cur_0$-eigenvector. From the fact that $\nu=\half \cur_{-1}\cur$, we see that
	\begin{equation}  \label{eq:conformal_weight_cur_0-eigenvalue}
		nb=L_0b=\left(\frac{\cur_0^2}{2}+\sum_{k=1}^{+\infty}\cur_{-k}\cur_k \right)b=\frac{\cur_0^2}{2}b=\frac{\lambda^2}{2}b \,.
	\end{equation}
	This shows that for all $\lambda\in\spec(\cur_0)$, $\lambda^2\in 2\Zpluseq$ and the desired result follows.
\end{proof}

Now, we are ready to prove the main result of this section, which is a characterization of the simple CFT type preunitary VOA extension of $L(1,0)$ with non-zero weight one subspace, giving a stronger version of \cite[Porposition 5.2]{DL14}, cf.\ \cite[Theorem 1.3]{DM04}: 

\begin{theo}  \label{theo:characterization_preunitary_with_currents}
	Let $U$ be a simple CFT type preunitary VOA extension of $L(1,0)$ with $U_1\not=\{0\}$. Then $U$ is isomorphic to the Heisenberg VOA $M(1)$ or to a rank-one lattice VOA $V_{\latt_{2n}}$ for some $n\in\Zplus$.
\end{theo}

\begin{proof}
	We use notations and results from Proposition \ref{prop:preunitary_exts_currents_contains_heisenberg} and Lemma \ref{lem:properties_spectrum_current_j}.
	By part $(iii)$ of the latter, we known that either $\spec(\cur_0)=\{0\}$ or $\spec(\cur_0)=\latt_{2n}$ for some $n\in\Zplus$.
	Suppose $\spec(\cur_0)=\{0\}$, then $U=U^0\supseteq U^\cur$. Suppose by contradiction that $W:=U^0\backslash U^\cur\not=\{0\}$. 
	For all $m\in\Zpluseq$, set $W_m:=W\cap U_m$.
	As $U$ is of CFT type, we have that $W_0=\{0\}$. Then let $m$ be the smallest positive integer such that $W_m\not=\{0\}$. Consider a non-zero vector $b\in W_m$. Therefore, reasoning repeatedly as in \eqref{eq:conformal_weight_cur_0-eigenvalue}, we get that for all $t\in\Zplus$, $(2m)^tb=(\cur_0)^{2t}b$, which is a contradiction as $\cur_0$ is nilpotent in $U^0_m$. Therefore, $W=\{0\}$ and $U=U^\cur\cong M(1)$.
	
	Now, assume that $\spec(\cur_0)=\latt_{2n}$ for some $n\in\Zplus$.
	For all $\lambda\in\spec(\cur_0)$, define
	\begin{equation}
		W^\lambda:=\ker(\cur_0-\lambda 1_U)\subseteq U^\lambda
		\,,\qquad
		W:=\bigoplus_{\lambda\in\spec(\cur_0)}W^\lambda \,.
	\end{equation}
	 Using that $\cur_0$ is a derivation of $U$, see \eqref{eq:cur_0_is_a_derivation}, it is not difficult to prove that $W$ is a vertex subalgebra of $U$ containing $U^\cur$. 
	 Moreover, we can define a continuous representation of the one-dimensional compact torus $\T$ on $W$ by $t\mapsto e^{i\frac{t}{\sqrt{2n}}\cur_0}$. Then $W^\T=U^\cur\cong M(1)$, so that for all $\lambda\in\spec(\cur_0)$, $W^\lambda$ is an irreducible $M(1)$-module by \eqref{eq:decomposition_voaxcompact_group-modules}.
	 It follows that $W$ is isomorphic to $V_{\latt_{2n}}$, see the construction of the latter in \cite[Section 6.4]{LL04}. 
	 It is known, see Remark \ref{rem:strongly_rationality_lattice_orbifolds}, that $\cur_0^N$ is diagonalizable on every $V_{\latt_{2n}}$-module $N$ and then it must be diagonalizable also on $U$. This means that for all $\lambda\in\spec(\cur_0)$, $U^\lambda=W^\lambda$, so that $U=W\cong V_{\latt_{2n}}$, concluding the proof.
\end{proof}

\subsection{\texorpdfstring{Preunitary VOA extensions containing $M(1)^+$}{Preunitary VOA extensions containing M(1)+}}
\label{subsec:preunitary_cw4}

We start presenting some important facts about the representation theory of the fixed-point vertex subalgebra $M(1)^+$ of $M(1)$ by the VOA involution $\phi$ introduced in Section \ref{subsec:discrete_VOA_extensions}.
To this aim, see Remark \ref{rem:strongly_rationality_lattice_orbifolds} with references therein and recall the general results and the notation from Section \ref{subsec:voas_and_modules}. 

\begin{rem}  \label{rem:M(1)^+-modules_from_M(1)-modules}
	It is known that $M(1)$ is generated by the vacuum vector and any vector $\cur$ of conformal weight $1$. Moreover, the VOA involution $\phi$ is such that $\phi(\cur)=-\cur$. Therefore, $M(1)$ decomposes as the direct sum of the eigenspaces $M(1)^\pm$ of $\phi$ of eigenvalues $\pm 1$ respectively. In particular, $M(1)^+$ is a fixed-point vertex subalgebra of $M(1)$ and $M(1)^-$ is an irreducible $M(1)^+$-module. 
	As a consequence, every $M(1)$-module is also a $M(1)^+$-module. 
	For all $h\in\C$, we denote by $\dutchcal{M}(1,h)$ the irreducible highest weight $M(1)$-module of highest weight $h\in\C$, that is $\cur$ is diagonalizable on $\dutchcal{M}(1,h)$ with eigenvalues in $h+\Zpluseq$. For all $h\in\C$, the conformal weight of $\dutchcal{M}(1,h)$ is $\half\abs{h}^2$.
	Note that $\dutchcal{M}(1,0)$ is isomorphic to the adjoint module of $M(1)$.
	Actually, we have that every irreducible $M(1)$-module is isomorphic to $\dutchcal{M}(1,h)$ for some $h\in\C$.
	For all $h\in\C$, we denote by $M(1,h)$ the corresponding $M(1)^+$-module.
	It turns out that for all $h\in\C\backslash\{0\}$, $M(1,h)$ is an irreducible $M(1)^+$-module and it is isomorphic to $M(1,-h)$, see \cite[Proposition 2.4]{DN99c}. 
	It is not difficult to see that for all $h\in\C$, $\dutchcal{M}(1,h)'\cong \dutchcal{M}(1,-h)$, so that $M(1,h)$ as well as $M(1)^\pm$ is self-contragredient. 
	Moreover, note that $\cur\in M(1)^-$, so that $M(1)^-$ must contain a $L(1,0)$-submodule isomorphic to $L(1,1)$. Indeed, we have the following decompositions into irreducible $L(1,0)$-modules see \cite[Theorem 2.7]{DG98}:
	\begin{equation}  \label{eq:decompositions_M(1)^pm_vir-modules}
		M(1)^+ \cong \bigoplus_{h\in\Zpluseq} L(1, 4h^2)
		\,,\qquad
		M(1)^-\cong \bigoplus_{h\in\Zpluseq} L(1, (2h+1)^2) \,.
	\end{equation}
\end{rem}

\begin{rem}  \label{rem:catc_{M(1)^+}_and_its_completion}
	Denote by $\catc_{M(1)^+}$ the semisimple abelian full subcategory of $\Rep(M(1)^+)$ generated by the $M(1)^+$-modules $M(1)^\pm$ and $M(1,h)$ for all $h\in\C\backslash\{0\}$, see Remark \ref{rem:M(1)^+-modules_from_M(1)-modules}.
	It follows from \cite[Theorem 4.3 and Theorem 4.8]{ALY} that $\catc_{M(1)^+}$ has a vertex tensor category structure, making it a ribbon tensor category. 
	The tensor unit $\id_{\catc_{M(1)^+}}$ is the adjoint module of $M(1)^+$, which is simple. Recall that $\catc_{M(1)^+}$ has a twist given by the action of the operator $e^{i2\pi L_0}$ on $M(1)^+$-modules.
	Note that the category $\catc_{M(1)^+}$ is a semisimple full subcategory of the category $\catc_1(M(1)^+)$ of $C_1$-cofinite grading-restricted generalized $M(1)^+$-modules, see Remark \ref{rem:generalized_modules} and Remark \ref{rem:finite-length_modules} for definitions.
	The latter is not semisimple and it coincides with the category of finite-length generalized $M(1)^+$-modules, whose composition factors are in $\catc_{M(1)^+}$, see \cite[Theorem 3.8 and Theorem 4.2]{ALY}.
	By an argument as the one in the proof of Theorem \ref{theo:discrete_series_extensions}, we can show that $\catc_{M(1)^+}$ satisfies the hypotheses $(i)$--$(vi)$ in Subsection \ref{subsubsec:ind-categories}.
	Therefore, we can consider the direct limit completion $\Ind(\catc_{M(1)^+})$ of $\catc_{M(1)^+}$, which has a vertex tensor category structure extending the one of $\catc_{M(1)^+}$ and whose objects are possibly infinite direct sum of $M(1)^+$-modules in $\catc_{M(1)^+}$. In particular, $\Ind(\catc_{M(1)^+})$ is a balanced tensor category.
\end{rem}

\begin{rem}  \label{rem:relevant_fusion_rules_for_M(1)^+-modules}
	The non-trivial fusion rules for the irreducible $M(1)^+$-modules in $\catc_{M(1)^+}$, see Remark \ref{rem:catc_{M(1)^+}_and_its_completion}, are given by the following isomorphisms, see \cite[Theorem 4.5]{Abe00} and \cite[Theorem 4.6]{ALY}: for all $h,k\in\C\backslash\{0\}$,
	\begin{align}  \label{eq:iso_M(1)+_modules}
		M(1)^- \boxtimes M(1)^- &\cong M(1)^+  \\
		M(1)^-\boxtimes M(1,h) &\cong M(1,h) \\
		M(1,h)\boxtimes M(1,k) &\cong M(1, h+k)\oplus M(1, h-k) \,.
	\end{align}
	We note that $M(1)^-$ is rigid and invertible with $(M(1)^-)^*\cong M(1)^-\cong (M(1)^-)^\vee$, so that it is a $\Z_2$-simple current in $\catc_{M(1)^-}$.
	Recall Remark \ref{rem:some_facts_vtc} that the twist $\vartheta_N$ on a generic VOA module $N$ is given by $e^{i2\pi L_0^N}$. By \eqref{eq:decompositions_M(1)^pm_vir-modules}, this immediately implies that $\vartheta_{M(1)^-}=\one_{M(1)^-}$.
	As the fusion rule $I_{M(1)^-\, M(1)^-}^{M(1)^+}=1$, if $\Y$ is an intertwining operator of type $\binom{M(1)^+}{M(1)^-\, M(1)^-}$, then for all $a,b\in M(1)^-$, $\Y(a,z)b=Y(a,z)b$, up to a constant factor, where $Y$ denotes the vertex operator of the Heisenberg VOA $M(1)$, see Remark \ref{rem:M(1)^+-modules_from_M(1)-modules}. Then from the definition of the braiding in a vertex tensor category, see e.g.\ \cite[Section 3.3.4]{CKM24} and references therein, and the skew-symmetry \eqref{eq:skew-symmetry} of $M(1)$, we can deduce that the braiding $b_{M(1)^-, M(1)^-}$ in $\catc_{M(1)^+}$ is trivial.
	As a consequence, the full subcategory $\tilde{\catc}$ of $\catc_{M(1)^+}$ generated by $M(1)^\pm$ is ribbon tensor equivalent to $\catc_{\Z_2}$, see Section \ref{sec:simple_G-algebras} for notation. By Theorem \ref{theo:DR_cross_product_VOAs} and Remark \ref{rem:DR_VOA_extension_strongly_rational} or by Corollary \ref{cor:galois_correspondence}, $M(1)$ is the DR VOA extension $M(1)^+\rtimes \tilde{\catc}$ of $M(1)^+$.
	In other words, see Example \ref{ex:simple_current_exts_as_DR_VOA_exts}, $M(1)$ is a $\Z_2$-graded simple current extension of $M(1)^+$ by the $\Z_2$-simple current $M(1)^-$.
\end{rem}

The following result says that every non-trivial simple CFT type preunitary VOA extension of $L(1,0)$ satisfying the spectrum condition contains a discrete series VOA extension of $L(1,0)$.

\begin{prop}  \label{prop:preunitary_exts_contain_discrete_series_ones}
	Let $U$ be a simple CFT type preunitary VOA extension of $L(1,0)$, that is as $L(1,0)$-module,
	\begin{equation}  
		U\cong L(1,0)\oplus \bigoplus_{h\in\Zplus} \alpha_h L(1,h)
		\,,\qquad \alpha_h\in\Zpluseq \,.
	\end{equation} 
	Then $U$ has a unique normalized non-degenerate invariant bilinear form $\bilinear$.
	Furthermore, $U$ has the following orthogonal decomposition into $L(1,0)$-modules with respect to $\bilinear$:
	\begin{equation}  \label{eq:orthogonal_decomposition_preunitary_exts}
		U\cong W\oplus W^\perp
		\,,\qquad
		W:=L(1,0)\oplus\bigoplus_{j\in\Zplus}\alpha_{j^2}L(1,j^2)
		\,,\qquad
		W^\perp :=\bigoplus_{h\in S}\alpha_h L(1,h)	
	\end{equation}
	where $S$ is a subset of $\Zplus$ not containing perfect squares, $W$ is a simple vertex subalgebra of $U$ isomorphic to $V_{\latt_2}^H$ for some $H$ isomorphic to a closed subgroup of $\SO(3)$, and $W^\perp$ is a $W$-module.
\end{prop}

\begin{proof}
	Thanks to the fusion rules \eqref{eq:fusion_rules_discrete_series}, $W$ is a vertex subalgebra of $U$. And $W^\perp$ is a $W$-module thanks to the fusion rules \eqref{eq:fusion_rules_preunitary}. 
	By $(ii)$ of Lemma \ref{lem:preunitary_extensions_have_bilinear_forms}, $U$ has a unique normalized non-degenerate invariant bilinear form $\bilinear$. 
	By $(iii)$ of Lemma \ref{lem:preunitary_extensions_have_bilinear_forms}, we have the orthogonal decomposition \eqref{eq:orthogonal_decomposition_preunitary_exts}. In particular, $\bilinear$ restricts to a normalized non-degenerate invariant bilinear form on $W$. As $U$ is of CFT type, $W$ is also simple, see Remark \ref{rem:invariant_bilinear_form}, and the conclusion follows from Theorem \ref{theo:discrete_series_extensions}. 
\end{proof}

Proposition \ref{prop:preunitary_exts_contain_discrete_series_ones} is the starting point for the following technical lemma.

\begin{lem}  \label{lem:every_preunitary_ext_of_M(1)^+_is_completely_reducible}
	Let $U$ be a simple CFT type preunitary VOA extension of $L(1,0)$ containing a vertex subalgebra isomorphic to $M(1)^+$. Then $U$ is a completely reducible $M(1)^+$-module.
\end{lem}

\begin{proof}
	Let $U$, $W$, $W^\perp$, $V_{\latt_2}^H$ and $\bilinear$ be as given by Proposition \ref{prop:preunitary_exts_contain_discrete_series_ones}.
	Note that the copy of $M(1)^+$ must be contained in $W$, so that both $W$ and $W^\perp$ are $M(1)^+$-submodules of $U$. 
	As $M(1)^+$ is isomorphic to $V_{\latt_2}^{\operatorname{D}_\infty}$ and $W\cong V_{\latt_2}^H$ is isomorphic to a $M(1)^+$-submodule of $V_{\latt_2}$, we know from  \eqref{eq:decomposition_voaxcompact_group-modules} that $W$ is a completely reducible $M(1)^+$-module. Moreover, $H$ is isomorphic to a subgroup of $\operatorname{D}_\infty$, see also Corollary \ref{cor:galois_correspondence}, Corollary \ref{cor:orthogonally_complemented} and references therein.
	If $H$ is isomorphic to a subgroup of  $\T$, then $W$ contains a copy of $M(1)\cong V_{L_2}^\T$, so that $U$ is isomorphic to either $M(1)$ or a rank-one lattice VOA by Theorem \ref{theo:characterization_preunitary_with_currents}.
	In each of these cases, $U$ is a completely reducible $M(1)^+$-module by \eqref{eq:decomposition_voaxcompact_group-modules}.
	
	Therefore, we are left with the cases where $H\cong \operatorname{D}_k$ for some $k\in\Zplus\cup \{\infty\}$.
	If $k\in\Zplus$, then $W\cong V_{L_{2k^2}}^+$, which is strongly rational, see Remark \ref{rem:strongly_rationality_lattice_orbifolds}. This implies that $U$ is a completely reducible $V_{L_{2k^2}}^+$-module. 
	On the other hand, every irreducible $V_{L_{2k^2}}^+$-module is a completely reducible $M(1)^+$-module, see \cite[Theorem 3.1(2) and Eq.s (3.5)--(3.8)]{Abe05}. Then we can conclude that $U$ is a completely reducible $M(1)^+$-module.
	It remains to check the case $H\cong D_\infty$, which means that $W\cong M(1)^+$. Thanks to the properties given by Proposition \ref{prop:preunitary_exts_contain_discrete_series_ones}, we are in charge to apply the same arguments used in \cite[Lemma 5.1 -- Lemma 5.3]{DJ15}. This shows that $U$ is a completely reducible $M(1)^+$-module, so concluding the proof.
\end{proof}

Therefore, we are ready to prove the main theorem of this section:

\begin{theo}  \label{theo:characterization_preunitary_exts_primary_cw4}
	Let $U$ be a simple CFT type preunitary VOA extension of $L(1,0)$ containing a primary vector of conformal weight $4$. Then $U$ is isomorphic to either $M(1)$ or $M(1)^+$ or a rank-one lattice VOA $V_{\latt_{2n}}$ or its vertex subalgebra $V_{\latt_{2n}}^+$ for some $n\in\Zplus$.
\end{theo}

\begin{proof}
	Note that $U$ contains a copy of the irreducible $L(1,0)$-module $L(1,4)$ and thus it satisfies the spectrum condition.
	By Proposition \ref{prop:preunitary_exts_contain_discrete_series_ones}, $U$ contains a vertex subalgebra $W$ isomorphic to $V_{\latt_2}^H$ for some group $H$ isomorphic to a proper closed Lie subgroup of $\SO(3)$.
	Considering the non-trivial invariant harmonic polynomials of smaller degree for the exceptional finite subgroups $\operatorname{S}_4$, $\operatorname{A_4}$ and $\operatorname{A}_5$ of $\SO(3)$, cf.\ \cite[Section 3]{DGR99}, one can exclude, checking the decomposition of $V_{\latt_2}$ into $(L(1,0)\times \SO(3))$-modules, that $V_{L_2}^H$ contains primary vectors of conformal weight $4$ if $H$ is isomorphic to one of these subgroups.
	Accordingly, $H$ must be isomorphic to a group as in the statement. If $H$ is isomorphic to a subgroup of $\T$, then $U_1\not=\{0\}$ and the result follows from Theorem \ref{theo:characterization_preunitary_with_currents}.
	Then it remains to analyze the cases where $H$ is isomorphic to $\operatorname{D}_k$ for any $k\in\Zplus\cup\{\infty\}$. 
	Recall that $V_{\latt_2}^{\operatorname{D}_\infty}$ is isomorphic to $M(1)^+$.
	Therefore, in all the remaining cases, $U$ is a completely reducible $M(1)^+$-module thanks to Lemma \ref{lem:every_preunitary_ext_of_M(1)^+_is_completely_reducible} 
	
	A complete list of irreducible $M(1)^+$-modules is given in \cite[Theorem 4.5]{DN99c} and their relevant numerical data are given in \cite[Section 3.4]{DN99c} or in \cite[Table 1]{ALY}. In particular, we note that
	all irreducible $M(1)^+$-modules with integer conformal weights are contained in $\catc_{M(1)^+}$ as introduced in Remark \ref{rem:catc_{M(1)^+}_and_its_completion}. This means that every preunitary VOA extension of $L(1,0)$, which is also a completely reducible $M(1)^+$-module, is a genuine object in $\Ind(\catc_{M(1)^+})$.
	Therefore, $U$ is an object in $\Ind(\catc_{M(1)^+})$, so that we have the following decomposition of $U$ into irreducible $M(1)^+$-modules, see Remark \ref{rem:M(1)^+-modules_from_M(1)-modules}:
	\begin{equation}
		U\cong M(1)^+\oplus \bigoplus_{h\in\Z_{>1}}\beta_h M(1,\sqrt{2h})
		\,,\qquad \beta_h\in\Zpluseq \,.
	\end{equation} 
	In particular, note that $U$ cannot contain neither a copy of $M(1)^-$ nor one of $M(1,\sqrt{2})$ as otherwise $U_1\not=\{0\}$.
	
	Consider the simple commutative algebra $A$ with trivial twist in $\Ind(\catc_{M(1)^+})$ corresponding to the VOA extension $M(1)^+\subseteq U$. 
	Recall from Remark \ref{rem:properties_induction_functor} the properties of the induction functor $\F:\Ind(\catc_{M(1)^+})\to \Rep(A)$, where $\Rep(A)$ is the category of (not necessarily local) $A$-modules in $\Ind(\catc_{M(1)^+})$.
	For any $M\in \Ind(\catc_{M(1)^+})$, $\F(M)$ is in the category $\Rep^0(A)$ of local $A$-modules if and only if the monodromy $\mathcal{M}_{U, M}$ between $U$ and $M$ in $\Ind(\catc_{M(1)^+})$ is equal to $\one_{U\boxtimes M}$. 
	Recall from Remark \ref{rem:some_facts_vtc} that the twist $\vartheta_N$ on a generic VOA-module $N$ is given by $e^{i2\pi L_0^N}$.
	This immediately gives that $\vartheta_{M(1)^-}=\one_{M(1)^-}$ and we already know that $\vartheta_U=\one_U$.
	Moreover, from Remark \ref{rem:relevant_fusion_rules_for_M(1)^+-modules}, we can calculate that
	\begin{equation} \label{eq:decomposition_induction_M(1)^-}
		U\boxtimes M(1)^-\cong M(1)^-\oplus \bigoplus_{h\in\Z_{>1}}\beta_h M(1,\sqrt{2h})
		\,,\qquad \beta_h\in\Zpluseq 
	\end{equation}
	so that $\vartheta_{U\boxtimes M(1)^-}$ is trivial too. Recalling from Section \ref{subsec:tensor_categories}, the relation between twist and monodromy, we have that
	\begin{equation}
		\one_{U\boxtimes M(1)^-}=\vartheta_{U\boxtimes M(1)^-}
		=(\vartheta_U\boxtimes \vartheta_{M(1)^-})\mathcal{M}_{U,M(1)^-}
		=\mathcal{M}_{U,M(1)^-} \,.
	\end{equation}
	Therefore, $\F(M(1)^-)$ is an object in $\Rep^0(A)$.
	Recall from Subsection \ref{subsubsec:ind-categories} that the category $\Rep_{\Ind(\catc_{M(1)^+})}(U)$ of weak $U$-modules, which are also objects of $\Ind(\catc_{M(1)^+})$ when seen as weak $M(1)^+$-modules, has a vertex tensor category structure and it is equivalent to $\Rep^0(A)$ as balanced tensor category.  	
	In Remark \ref{rem:relevant_fusion_rules_for_M(1)^+-modules}, we noted that $M(1)^-$ is rigid and invertible and it is in particular a $\Z_2$-simple current in $\catc_{M(1)^+}$. Moreover, the braiding $b_{M(1)^-, M(1)^-}$ and the twist $\vartheta_{M(1)^-}$ are trivial. 
	Indeed, $M(1)^\pm$ generate a full subcategory $\tilde{\catc}$ of $\catc_{M(1)^+}$, which is ribbon tensor equivalent to $\catc_{\Z_2}$.
	
	From the properties of the induction functor in Remark \ref{rem:properties_induction_functor}, we note that $\F(M(1)^-)$ is still rigid and invertible and such that $\F(M(1)^-)\boxtimes \F(M(1)^-)$ is isomorphic to the identity object of $\Rep_{\Ind(\catc_{M(1)^+})}(U)$. 
	Moreover, also the corresponding self-braiding and twist of $\F(M(1)^-)$ remain trivial.
	To sum up, $\F(M(1)^-)$ is a $\Z_2$-simple current of $U$ with trivial twist and with trivial self-braiding.
	Indeed, $\F(M(1)^\pm)$ generates the full subcategory $\F(\tilde{\catc})$ of $\Rep_{\Ind(\catc_{M(1)^+})}(U)$, which is still ribbon tensor equivalent to $\catc_{\Z_2}$.
	From the decomposition \eqref{eq:decomposition_induction_M(1)^-}, we see that $\F(M(1)^-)$ is actually a $U$-module with $L_0^{\F(M(1)^-)}$ having integer eigenvalues only. 
	By Theorem \ref{theo:DR_cross_product_VOAs}, we can consider the DR VOA extension $\tilde{U}:=U\rtimes \F(\tilde{\catc})$ of $U$, which in turn is a $\Z_2$-graded simple current extension of $U$, see Example \ref{ex:simple_current_exts_as_DR_VOA_exts}. In particular, $U$ is a fixed-point vertex subalgebra of $\tilde{U}$ by a group of VOA automorphisms isomorphic to $\Z_2$.
	We also have the following decomposition into $M(1)^+$-module:
	\begin{equation}
		\tilde{U}:=U\oplus \F(M(1)^-)\cong M(1,0)\oplus  \bigoplus_{h\in\Z_{>1}}2\beta_h M(1,\sqrt{2h})
		\,,\qquad \beta_h\in\Zpluseq \,.
	\end{equation}
	Note that $\tilde{U}$ is simple and of CFT type, and such that $\tilde{U}_1\not=\{0\}$.  
	Then $\tilde{U}$ has a non-degenerate invariant bilinear form, see Remark \ref{rem:invariant_bilinear_form}, and thus it is a simple CFT type preunitary VOA extension of $L(1,0)$ by Lemma \ref{lem:direct_sum_hw_L(1,0)_modules}. By Theorem \ref{theo:characterization_preunitary_with_currents}, $\tilde{U}$ is isomorphic to either $M(1)$ or a rank-one lattice VOA $V_{\latt_{2n}}$ for some $n\in\Zplus$. This implies that $U$ is isomorphic to either $M(1)^+$ or $V_{\latt_{2n}}^+$ respectively. Then the proof is completed.
\end{proof}

\begin{rem}
		Theorem \ref{theo:characterization_preunitary_exts_primary_cw4} gives alternative characterizations of the fixed-point vertex subalgebras $V_{\latt_{2n}}^+$ for all $n\in \Zplus$, given in \cite{DJ10, DJ13, DJ15}. Note that in Theorem \ref{theo:characterization_preunitary_exts_primary_cw4}, neither the strong rationality nor the condition that the effective central charge $\tilde{c}$ is equal to $1$ are assumed. 
\end{rem}

\subsection{Preunitary VOA extensions satisfying the spectrum condition}
\label{subsec:preunitary_spectrum_condition}

To prove the following proposition, we need to recall the definition of \textit{$g$-twisted $V$-module} for a VOA $V$ and a finite order VOA automorphism $g$ of it, see e.g.\ \cite[Definition 3.2]{DLM98}.

\begin{prop}  \label{prop:preunitary_VOA_exts_strongly_rational}
		Let $U$ be a simple preunitary VOA extension of $L(1,0)$ containing a vertex subalgebra $W$ isomorphic to $V_{\latt_2}^G$ for some finite subgroup $G$ of $\Aut(V_{\latt_2})$. Suppose that $W$ is strongly rational, then $W$ has effective central charge $\tilde{c}=1$ and $U$ is isomorphic to $V_{\latt_2}^H$ for some subgroup $H$ of $G$. 
\end{prop}

\begin{proof}
	As $W\cong V_{\latt_2}^G$ is strongly rational, by \cite[Theorem 3.3]{DRX17}, cf.\ \cite[Section 4.2]{McR21}, every irreducible $V_{\latt_2}^G$-module is contained in some $g$-twisted $V_{\latt_2}$-module for some $g\in G$. By definition, for all $g\in G$, every $g$-twisted $V_{\latt_2}$-module is also a $V_{\latt_2}^{\langle g\rangle}$-module, where $\langle g\rangle$ is the cyclic subgroup of $G$ generated by $g$. 
	Note also that every $V_{\latt_2}^{\langle g\rangle}$-module is also a $V_{\latt_2}^G$-module.
	On the other hand, for all $g\in G$, $V_{\latt_2}^{\langle g\rangle}$ is isomorphic to the rank-one lattice VOA $V_{\latt_{2n^2}}$, where $n\in\Zplus$ is the order of $g$, see Section \ref{subsec:discrete_VOA_extensions}. 
	Therefore, every irreducible $V_{\latt_2}^G$-module appears in the semisimple decomposition of some $V_{\latt_{2n^2}}$-module, when considered as a $V_{\latt_2}^G$-module.
	
	Lattice VOA modules are completely classified, see Remark \ref{rem:strongly_rationality_lattice_orbifolds} and references therein. In particular, there is an irreducible $V_{\latt_{2n^2}}$-module $M^\alpha$ for every element $\alpha$ in the coset space $\lcoset{\latt_{2n^2}^\circ}{\latt_{2n^2}}$, where $\latt_{2n^2}^\circ$ is the dual lattice of $\latt_{2n^2}$, that is
	\begin{equation}
		\latt_{2n^2}^\circ:=\{v\in \mathbb{Q}\otimes_\Z \latt_{2n^2}\mid \forall a(a\in\latt_{2n^2}) \,\,\,\Rightarrow\,\,\, B(v,a)\in\Z\}
	\end{equation} 
	where the bilinear form $B\bilinear$ on $\latt_{2n^2}$ is naturally extended to $\mathbb{Q}\otimes_\Z \latt_{2n^2}$. 
	Note that any $\alpha\in \lcoset{\latt_{2n^2}^\circ}{\latt_{2n^2}}$ is the class $\alpha_r$ of $q_r:=\frac{r}{2n^2}\sqrt{2n^2}$ for some $r\in\{0,\dots,2n^2-1\}$.
	Moreover, for all $\alpha_r\in \lcoset{\latt_{2n^2}^\circ}{\latt_{2n^2}}$, the conformal weight of $M^{\alpha_r}$ is given by:
	\begin{equation}
		\lambda_{M^{\alpha_r}}= \frac{B(q_r,q_r)}{2}=\frac{r^2}{4n^2}< 1 \,.
	\end{equation}
	
	Now, let $M$ be an irreducible $W$-module. Then $M$ must be a $W$-submodule of some $M^{\alpha_r}$ as above. This means that the conformal weight $\lambda_M\in \lambda_{M^{\alpha_r}}+\Zpluseq$.
	This immediately implies that the effective central charge of $W$ is $1$.
	Moreover, by strong rationality, $U$ is a completely reducible $W$-module. And if $M$ appears in the semisimple decomposition of $U$, then we must have that $\lambda_M\in \Zpluseq$.
	Therefore, $M$ must be a $W$-submodule of $M^{\alpha_0}$, which is the adjoint module of $V_{\latt_{2n^2}}$.   
	This implies that $U$ is actually a simple discrete series extension of $L(1,0)$ and the result follows from Theorem \ref{theo:discrete_series_extensions}.
\end{proof}

Therefore, we are ready to prove the following classification result for simple CFT type preunitary VOA extensions of $L(1,0)$ satisfying the spectrum condition:

\begin{theo}  \label{theo:classification_preunitary_spectrum_condition}
	Let $U$ be a simple CFT type preunitary VOA extension of $L(1,0)$ satisfying the spectrum condition. Then one of the following holds:
	\begin{itemize}
		\item[$(i)$] $U$ is isomorphic to a rank-one lattice VOA or to one of their fixed-point vertex subalgebras;
		
		\item[$(ii)$] $V_{\latt_2}^{\operatorname{A}_5}$ is not strongly rational and $U$ may be a hypothetical VOA extension of it, in which case, $V_{\latt_2}^{\operatorname{A}_5}$ would be also the maximal discrete series VOA extension of $L(1,0)$ contained in $U$.
	\end{itemize}
\end{theo}

\begin{proof}
	If $U$ is isomorphic to $L(1,0)$, then there is nothing to prove.
	Accordingly, suppose that $U$ is not isomorphic to $L(1,0)$.
	By Proposition \ref{prop:preunitary_exts_contain_discrete_series_ones}, $U$ contains a vertex subalgebra $W$ isomorphic to $V_{\latt_2}^H$ for some compact group $H$ isomorphic to a proper closed Lie subgroup of $\SO(3)$. 
	If $H$ is isomorphic to a subgroup of $D_\infty$, then we are in the case $(i)$ by Theorem \ref{theo:characterization_preunitary_exts_primary_cw4}. 
	Now, it remains to check the case where $H$ is isomorphic to a subgroup of $\operatorname{S}_4$ or of $\operatorname{A}_5$. In the former case, we are again in the case $(i)$ by the strong rationality of $V_{\latt_2}^{\operatorname{S}_4}$, see Remark \ref{rem:strongly_rationality_lattice_orbifolds}, and Proposition \ref{prop:preunitary_VOA_exts_strongly_rational}. Indeed, by the same result, the case $(ii)$ may occur if $V_{\latt_2}^{\operatorname{A}_5}$ is proved to not be strongly rational. This concludes the proof.
\end{proof}

\subsection{\texorpdfstring{Strongly rational VOAs with $c=1$}{Strongly rational VOAs with c=1}}
\label{subsec:strongly_rational_c=1}

In \cite[Theorem I]{DLN15}, it is proved that $q$-characters are modular functions on some congruence subgroup of the modular group $\SL(2,\Z)$ and this fact can be used to prove that the classification result in \cite{Kir89}, see also \cite{Gin88, DVV88, DVVV89}, hold in the VOA setting: 

\begin{theo}[{\cite{Kir89}, \cite[Theorem 2.9]{DJ14}}]
	\label{theo:q-charatcer_classification_strongly_rational_voas}
	Let $V$ be a strongly rational VOA with $c=1=\tilde{c}$. Then the $q$-character of $V$ is equal to the $q$-character of a rank-one lattice VOA $V_{\latt_{2n}}$ for some $n\in\Zplus$ or one of its fixed-point vertex subalgebra $V_{\latt_{2n}}^H$ for some finite subgroup $H$ of $\Aut(V_{\latt_{2n}})$.
\end{theo}

Indeed, motivated by these results, we have the following well-known and long-standing conjectural classification result, recall Section \ref{subsec:discrete_VOA_extensions}:

\begin{con}  \label{con:strongly_rational_VOAs_cc1}
	The complete list of isomorphism classes of strongly rational preunitary VOA extensions of $L(1,0)$ with effective central charge $\tilde{c}=1$ is given by rank-one lattice VOAs $V_{\latt_{2n}}$ for all $n\in\Zplus$, together with their fixed-point vertex subalgebras $V_{\latt_{2n}}^+$ and $V_{\latt_2}^H$ for all finite subgroups $H$ of $\Aut(V_{\latt_2})$, isomorphic to the exceptional subgroups $\operatorname{S_4}$, $\operatorname{A}_4$ and $\operatorname{A}_5$ of $\SO(3)$.
\end{con}

In the following, we prove Conjecture \ref{con:strongly_rational_VOAs_cc1} under some assumptions.
We start with the following crucial result, cf.\ \cite[Theorem 2.2]{Lin17}:

\begin{prop}  \label{prop:every_strongly_rational_preunitary_exts_satisfies_spectrum_condition}
	Every strongly rational preunitary VOA extension $U$ of $L(1,0)$ satisfies the spectrum condition. 
\end{prop}

\begin{proof}
	Consider the semisimple decomposition of $U$ into irreducible $L(1,0)$-modules:
	\begin{equation}
		U\cong L(1,0)\oplus \bigoplus_{h\in\Zplus} \alpha_h L(1,h)
		\,,\qquad \alpha_h\in\Zpluseq \,.
	\end{equation}
	Then $U$ has the following $q$-character, see e.g.\ \cite[p.\ 265]{DG98} and references therein,
	\begin{align}
		\chi_U(q)
			&=
		\frac{1}{\eta(q)}\left[
		1-q+\sum_{h\in\Zplus}\alpha_{h^2}\left(q^{h^2}-q^{(h+1)^2}\right)+\sum_{\substack{h\in\Zplus\\ h\not=k^2 ,\, k\in\Z}}\alpha_hq^h
		\right]	\\
			&=
		\frac{1}{\eta(q)}\left[
		1+(\alpha_1-1)q+\sum_{h\in\Z_{>1}}\left(\alpha_{h^2}-\alpha_{(h-1)^2}\right) q^{h^2}+\sum_{\substack{h\in\Zplus\\ h\not=k^2 ,\, k\in\Z}}\alpha_hq^h
		\right]
	\end{align}
	where
	\begin{equation}
		\eta(q)=q^\frac{1}{24}\prod_{n=1}^{+\infty}(1-q^n) \,.
	\end{equation}
	In the above expressions, put $q=e^{i2\pi\tau}$ for $\tau\in\C$ such that the imaginary part $\Im(\tau)>0$ and denote $\chi_U(e^{i2\pi\tau})$ and $\eta(e^{i2\pi\tau})$ by, with an abuse of notation, $\chi_U(\tau)$ and $\eta(\tau)$ respectively. Note that $\eta(\tau)$ is the Dedekind eta function and that $\chi_U(\tau)$ (the vacuum character of $U$) is a holomorphic function on the complex upper half-plane thanks to \cite[Theorem 5.3.3]{Zhu96}.
	By \cite[Theorem I]{DLN15}, $\chi_U(\tau)$ is a modular function of weight $0$ on some congruence subgroup of the modular group $\SL(2,\Z)$. As $\eta(\tau)$ is a modular form of weight $\half$ on the full modular group, $f_U(e^{i2\pi\tau}):=\eta(\tau)\chi_U(\tau)$ is a modular form of weight $\half$ on the same congruence subgroup of $\chi_U(\tau)$. By \cite[(1) and (3) of Corollary 3]{SS77}, there exists a periodic function $\varepsilon_1$ on $\Z$ such that
	\begin{equation}
		\varepsilon_1(1)=\alpha_1-1
		\,,\qquad
		\forall h\in\Z_{>1} \quad
		\varepsilon_1(h)=\alpha_{h^2}-\alpha_{(h-1)^2} \,.
	\end{equation}
	If $\alpha_{h^2}=0$ for all $h\in\Zplus$, then $\varepsilon_1(1)=-1$ and for all $h\in\Z_{>1}$, $\varepsilon_1(h)=0$, which contradicts the periodicity of $\varepsilon_1$. This shows that spectrum condition holds.
\end{proof}

By assuming the strong rationality of $V_{\latt_2}^{\operatorname{A}_5}$, we have the first main theorem of this section:
\begin{theo} \label{theo:classification_strongly rational_VOAs_cc_1}
	Let $U$ be a direct sum of lowest weight modules for the Virasoro algebra $\Vir_1$. 
	Suppose that $V_{\latt_2}^{\operatorname{A}_5}$ is strongly rational.
	Then $U$ is a strongly rational VOA with $c=1$ if and only if $U$ is isomorphic to one of the VOAs cited in Conjecture \ref{con:strongly_rational_VOAs_cc1}.
\end{theo}

\begin{proof}
	By Lemma \ref{lem:direct_sum_hw_L(1,0)_modules}, $U$ is a simple CFT type preunitary VOA extension of $L(1,0)$. By Proposition \ref{prop:every_strongly_rational_preunitary_exts_satisfies_spectrum_condition}, $U$ satisfies the spectrum condition, so that the result follows from Theorem \ref{theo:classification_preunitary_spectrum_condition} together with the strong rationality of $V_{\latt_2}^{\operatorname{A}_5}$.
\end{proof}

If one wants to avoid assuming the strong rationality of $V_{\latt_2}^{\operatorname{A}_5}$, we can do it by means of Theorem \ref{theo:q-charatcer_classification_strongly_rational_voas}.
On the other hand, to use Theorem \ref{theo:q-charatcer_classification_strongly_rational_voas}, we have to assume a further condition on the effective central charge, which we were able to avoid in the proof of Theorem \ref{theo:classification_strongly rational_VOAs_cc_1}.
Indeed, we have the following classification result for strongly rational VOAs with $c=1=\tilde{c}$:

\begin{theo}  \label{theo:classification_strongly rational_VOAs_cc_cceff_1}
	Let $U$ be a direct sum of lowest weight modules for the Virasoro algebra $\Vir_1$. 
	If $U$ is a strongly rational VOA with $c=1=\tilde{c}$, then it is isomorphic to one of the VOAs cited in Conjecture \ref{con:strongly_rational_VOAs_cc1}.
\end{theo}

\begin{proof}
	By Lemma \ref{lem:direct_sum_hw_L(1,0)_modules}, $U$ is a simple CFT type preunitary VOA extension of $L(1,0)$. By Theorem \ref{theo:q-charatcer_classification_strongly_rational_voas}, $U$ satisfies the spectrum condition, so that it must be isomorphic to a VOA as in Conjecture \ref{con:strongly_rational_VOAs_cc1} by Theorem \ref{theo:classification_preunitary_spectrum_condition}.
\end{proof}

\begin{rem}
	As already commented in the Introduction, a proof of the strong rationality of $V_{\latt_2}^{\operatorname{A}_5}$ has been recently given in \cite{Xu}.
	This can be used to suppress the hypothesis on the strong rationality of $V_{\latt_2}^{\operatorname{A}_5}$ in Theorem \ref{theo:classification_strongly rational_VOAs_cc_1}, so giving a full proof of Conjecture \ref{con:strongly_rational_VOAs_cc1}, also without the assumption on the effective central charge. Note that if $V_{\latt_2}^{\operatorname{A}_5}$ is strongly rational, then it automatically have effective central charge $\tilde{c}=1$ by Proposition \ref{prop:preunitary_VOA_exts_strongly_rational}.
	We point out that a list of $V_{\latt_2}^{\operatorname{A}_5}$-modules were given in \cite{WZ}, assuming the strong rationality of the model.
\end{rem}

\subsection{Vertex subalgebras and generalized symmetries}
\label{subsec:vertex_subalgebras_generalized_symmetries}

The main results of this final section is an application of the Galois correspondence given in Corollary \ref{cor:galois_correspondence} together with the classification results obtained in the previous sections. Indeed, we obtain classification theorem for vertex subalgebras of the rank-one lattice VOAs as stated in the following. Recall Section \ref{subsec:discrete_VOA_extensions}.

\begin{theo}  \label{theo:classificaiton_vertex_subalgebras}
	Let $U$ be a simple vertex subalgebra of $V_{\latt_{2n}}$, for $n\in\Zplus$, with the same conformal vector. Then:
	\begin{itemize}
		\item[$(i)$] if $n=1$, then $U$ is equal to a fixed-point vertex subalgebra $V_{\latt_2}^H$ of $V_{\latt_2}$ for some reductive closed subgroup $H$ of $\Aut(V_{\latt_2})\cong \operatorname{PSL}(2,\C)$;
		
		\item[$(ii)$] if $n$ is not a perfect square and $U$ is a preunitary VOA extension of $L(1,0)$ satisfying the spectrum condition (e.g.\ if $U$ is strongly rational by Proposition \ref{prop:every_strongly_rational_preunitary_exts_satisfies_spectrum_condition}), then $U=L(1,0)$ or $U$ is equal to a fixed-point vertex subalgebra $V_{\latt_{2n}}^H$ of $V_{\latt_{2n}}$ for some proper reductive closed subgroup $H$ of $\Aut(V_{\latt_{2n}})\cong \C^*\rtimes\Z_2$.
	\end{itemize}
\end{theo} 

\begin{proof}
	As explained in Section \ref{subsec:discrete_VOA_extensions}, $V_{\latt_2}^{\operatorname{PSL}(2,\C)}$ is equal to the unique copy of $L(1,0)$ contained in $V_{\latt_2}$, see \eqref{eq:decomposition_L_2}. 
	By the proof of Theorem \ref{theo:discrete_series_extensions}, the hypotheses of Corollary \ref{cor:galois_correspondence} are satisfied, so that $(i)$ is proved.
	To prove $(ii)$, first note that $U$ is of CFT type as $V_{\latt_{2n}}$ is so. 
	From the decompositions into $L(1,0)$-modules \eqref{eq:decomposition_L_2n_not_perfect_square} and \eqref{eq:decompositions_M(1)^pm_vir-modules} of $V_{\latt_{2n}}$, we note that there is only one maximal discrete series VOA extension of $L(1,0)$ contained in $V_{\latt_{2n}}$, which is isomorphic to $M(1)$.
	Recall from Section \ref{subsec:discrete_VOA_extensions} that this is given by $V_{\latt_{2n}}^{\C^*}$, so that $V_{\latt_{2n}}^{\C^*\rtimes\Z_2}$ is isomorphic to $M(1)^+$. As $U$ satisfies the spectrum condition and if it is not equal to $L(1,0)$, then it must contains $V_{\latt_{2n}}^{\C^*\rtimes\Z_2}$ by Proposition \ref{prop:preunitary_exts_contain_discrete_series_ones}.
	By Remark \ref{rem:catc_{M(1)^+}_and_its_completion}, Corollary \ref{cor:galois_correspondence} can be applied, proving $(ii)$ and completing the proof.
\end{proof}

\begin{rem}  \label{rem:classification_vertex_subalgebras_compact_group_unitarity}
	Let us work in the framework of Theorem \ref{theo:classificaiton_vertex_subalgebras}.
	By Corollary \ref{cor:galois_correspondence}, for all reductive closed subgroups $H$ of $\Aut(V_{\latt_{2n}})$, we can always find a maximal compact subgroup $K$ of $\Aut(V_{\latt_{2n}})$ such that $V_{\latt_{2n}}^H=V_{\latt_{2n}}^{K\cap H}$, where $K\cap H$ is a maximal compact subgroup of $H$.
	We remark that the rank-one lattice VOAs are \textit{unitary} by \cite[Theorem 4.12]{DL14}, see \cite[Definition 2.2]{DL14} and \cite[Chapter 5]{CKLW18}.
	By \cite[Theorem A.4]{CGGH23}, the group $\Aut_{\scalar}(V)$ of unitary VOA automorphisms of a unitary VOA $V$ is a maximal compact subgroup of $\Aut(V)$, see also Remark \ref{rem:topology_VOA_automorphism_group}.
	By \cite[Proposition A.5]{CGGH23}, we can choose a unitary structure on $V_{\latt_{2n}}$ in such a way that $K$ acts as $\Aut_{\scalar}(V_{\latt_{2n}})$. This implies that all fixed-point vertex subalgebras $V_{\latt_{2n}}^{K\cap H}$ are \textit{unitary subalgebras} in the sense of \cite[Section 5.4]{CKLW18}. Note that all unitary subalgebras of $V_{\latt_{2n}}$ satisfy the spectrum condition, see \cite[Theorem 3.8 and Theorem 3.15]{CGH19}.
\end{rem}

\begin{rem}   \label{rem:generalized_symmetries}
	Recently, the notion of \lq\lq generalized symmetry'' (possibly non-invertible i.e.\ not necessarily related to a group) is considered in several approaches to (rational) CFT.
	We first focus on the strongly rational VOA context.
	A generalized symmetry of a strongly rational VOA $V$ corresponds to a strongly rational vertex subalgebra $U$ of $V$ with the same conformal vector, and it can be thought as follows. 
	Given a pair $U\subseteq V$ of strongly rational VOAs with the same conformal vector, generalized symmetries give rise to fusion categories (or the associated fusion rings) or finite hypergroups, see e.g.\ \cite[Section D]{Ray24}, \cite[Section 7.1]{MRb}, \cite[Section 3]{DNRX} or \cite[Section 3]{Rie25} respectively, see also \cite{GR}, cf.\ \cite{Bis17}, \cite{BDVG23} in the conformal net setting. 
	Given such a pair $U\subseteq V$ of strongly rational VOAs, one may see $U$ as a finite \lq\lq generalized orbifold'', i.e.\ as a fixed-point vertex subalgebra of $V$, under some generalization of an ordinary finite group action by VOA automorphisms.
	In this framework, Theorem \ref{theo:classificaiton_vertex_subalgebras} shows for example that for $V = V_{\latt_2}$, such generalized symmetries are just ordinary group symmetries given by VOA automorphisms.
	This was previously argued in \cite[$\mathrm{Example}^{\mathrm{ph}}$ 7.8]{MRb}, assuming Conjecture \ref{con:strongly_rational_VOAs_cc1}.
	This remains true in the case where $V=V_{\latt_{2n}}$ for all $n\in\Zplus$ not a perfect square.
	For all $n=m^2$ for some $m\in\Z_{>1}$, the hypergoup corresponding to a generalized symmetry of $V=V_{\latt_{2n}}$ must be a double coset $K/\!\!/ \Z_m$, where $K$ is finite subgroup of $\Aut(V_{\latt_{2n}})\cong\C^*\rtimes \Z_2$ and $\Z_m$ is a cyclic subgroup of $K$ of order $m$.  
	More generally, if $V$ is a strongly rational VOA as listed in Conjecture \ref{con:strongly_rational_VOAs_cc1}, the hypergoup corresponding to one of its generalized symmetry must be isomorphic to a double coset $K/\!\!/ G$, where $K$ is a finite subgroup of $\operatorname{PSL}(2,\C)$ and $G$ is a subgroup of $K$.
	Thanks to Theorem \ref{theo:classification_preunitary_spectrum_condition} and to Theorem \ref{theo:classificaiton_vertex_subalgebras}, a similar situation, with some exception, holds for inclusions $U\subseteq V$ of not necessarily rational simple preunitary VOA extensions of $L(1,0)$ satisfying the spectrum condition, if one replaces the finite hypergroup with a compact one.
	For example, the compact hypergroup corresponding to the inclusion $L(1,0)\subset M(1)$ is isomorphic to the double coset 
	$\SO(3)/\!\!/ \SO(2)$, cf.\ \cite[Claim 4.8]{GR}.
	A possible exception is the inclusion $L(1,0)\subset V_{\latt_{2n}}$ with $n$ not a perfect square. 
\end{rem}

\bigskip
\noindent
{\small
	{\bf Acknowledgements.} We acknowledge support from the \emph{MIUR Excellence Department Project MatMod@TOV} awarded to the Department of Mathematics, University of Rome ``Tor Vergata'', CUP E83C23000330006 and from the University of Rome ``Tor Vergata'' funding OANGQS, CUP E83C25000580005. 
	L.G.\ also acknowledge support from the GNAMPA-INDAM project 2026 {\it Simmetrie distribuzionali per processi stocastici quantistici}, CUP E53C25002010001. The authors are supported in part by GNAMPA-INDAM. 
	S.C.\ would like to thank Chongying Dong, Terry Gannon, Cuipo Jiang, Robert McRae, Sven M{\"o}ller, Brandon C. Rayhaun and  Feng Xu for useful discussions at various stages of this research. 
	Google Gemini was used for helpful discussions and suggestions during the preparation of this manuscript.
	All mathematical statements, references, arguments, and final wording were checked and approved by the authors.

	\bigskip
	\noindent
	{\small
		\emph{Data sharing is not applicable to this article as no new data were created or analyzed in this study.} }
	
	\bigskip
	\noindent
	{\small
		\emph{On behalf of all authors, the corresponding author states that there is no conflict of interest.} }

\end{document}